\documentclass[a4paper,11pt,twoside]{article}

\usepackage[latin1]{inputenc} 
\usepackage[T1]{fontenc} 
\newlength{\minipagewidth}
\usepackage{amsmath,amsthm,amsthm}
\usepackage{amssymb}
\usepackage{a4wide}
\usepackage{graphicx}
\usepackage{color}
\usepackage{epic}
\usepackage{color}
\usepackage{enumerate}
\usepackage{authblk}

\usepackage[normalem]{ulem}

\newcommand{\pcalF}[1]{\calF^{(#1)}}

\newcommand{\filt}{{\rm filt}}
\newcommand{\card}{\mathop{\rm card}}
\renewcommand{\sup}{{\rm sup}}

\newcommand{\comment}[1]{}
\newcommand{\nrep}{n_{\rm rep}}

\newcommand{\Proba}{{\rm Proba}}

\newcommand{\PP}{\mathbb{P}}
\newcommand{\E}{\mathbb{E}}

\newcommand{\N}{\mathbb{N}}
\newcommand{\R}{\mathbb{R}}

\renewcommand{\P}{\mathbb{P}}

\newcommand{\calP}{\mathcal{P}}

\newcommand{\calF}{\mathcal{F}}
\newcommand{\calE}{\mathcal{E}}

\newcommand{\calS}{\mathcal{S}}
\newcommand{\calB}{\mathcal{B}}

\newcommand{\calG}{\mathcal{G}}

\newcommand{\calX}{\mathcal{X}}

\newcommand{\Law}{{\rm Law}}
\newcommand{\Id}{{\rm Id}}

\newcommand{\eps}{\varepsilon}
\newcommand{\ph}{\varphi}

\newcommand{\dps}{\displaystyle} 
\newcommand{\st}{  |  }
\newcommand{\eqdef}{=}

\newcommand{\as}{\,\, \text{a.s.}} 

\newcommand{\tX}{\tilde{X}}

\newcommand{\tT}{{\rm T}}

\newcommand{\abs}[1]{\left | #1\right |}
\newcommand{\set}[1]{\left\{#1\right\}}
\newcommand{\p}[1]{ \left(#1\right) }
\newcommand{\norm}[1]{\left\Vert#1\right\Vert}

\theoremstyle{plain}
\newtheorem{The}{Theorem}[section]
\newtheorem{Lem}[The]{Lemma}
\newtheorem{Pro}[The]{Proposition}

\newtheorem{Def}[The]{Definition}
\newtheorem{Ass}{Assumption}

\theoremstyle{definition}
\newtheorem{Rem}[The]{Remark}

\newcommand{\M}{\text{M}}
\newcommand{\Ens}{\mathcal{E}}

\usepackage{dsfont}

\usepackage{changes}

\title{Unbiasedness of some generalized Adaptive Multilevel Splitting algorithms}

\author[1]{Charles-Edouard
  Br\'ehier\thanks{charles-edouard.brehier@unine.ch,
    gazeauma@math.toronto.edu, goudenege@math.cnrs.fr,
    lelievre@cermics.enpc.fr, roussetm@cermics.enpc.fr}}
\author[2]{Maxime Gazeau}
\author[3]{Ludovic Gouden\`ege}
\author[4]{Tony Leli\`evre}
\author[4]{Mathias Rousset}
\affil[1]{\small Universit\'e de Neuch\^atel,  Institut de Math\'ematiques,
  Rue Emile Argand 11, CH-2000 Neuch\^atel, Switzerland.}
\affil[2]{\small University of Toronto, Department of Mathematics, 40 St. George St., Toronto
M5S~2E4, Canada.}
\affil[3]{\small F\'ed\'eration de Math\'ematiques de l'\'Ecole Centrale Paris,
  CNRS, Grande voie des vignes, 92295 Ch\^atenay-Malabry, France.}
\affil[4]{\small Universit\'e Paris-Est, CERMICS (ENPC), INRIA, 6-8 Avenue Blaise Pascal, Cit\'e Descartes,  F-77455 Marne-la-Vall\'ee, France.}

\begin{document}

\maketitle

\begin{abstract}
We introduce a generalization of the Adaptive Multilevel Splitting
algorithm in the discrete time dynamic setting, namely when it is applied to sample
rare events associated with paths of Markov chains. By interpreting the algorithm as a sequential sampler in path space, we are able to build an estimator of the
rare event probability (and of any non-normalized quantity associated with
this event) which is unbiased, whatever the choice of the importance
function and the number of replicas. This has practical consequences
on the use of this algorithm, which are illustrated through
various numerical experiments.
\end{abstract}

\section{Introduction}
The efficient sampling of rare events is a very important topic in
various application fields such as reliability analysis, computational
statistics or molecular
dynamics. Let us describe the typical problem of interest in the
context of molecular dynamics.

\subsection{Motivation and mathematical setting}

Let us consider the Markov chain
$(X_t)_{t \in \N}$ defined as the discretization of the overdamped Langevin dynamics:
\begin{equation}\label{eq:Lang}
\forall t \in \N, \, X_{t+1} - X_t = -\nabla V(X_t) \, h + \sqrt{2 \beta^{-1}}
(W_{t+h} - W_t).
\end{equation}
Typically, $X_t \in \R^{3N}$ is a high-dimensional vector giving the
positions of $N$ particles in $\R^3$ at time $t \, h$ ($h>0$ being the time step size), $V: \R^{3N} \to \R$ is the potential
function (for any set of positions $x \in \R^{3N}$, $V(x)$ is the
energy of the configuration), $\beta^{-1}=k_BT$ is the inverse temperature and $W_t$ is a standard Brownian motion
(so that $W_{t+h} - W_t$ is a vector of $3N$
i.i.d. centered Gaussian random variables with variance~$h$). In many
cases of interest, the dynamics~\eqref{eq:Lang} is metastable: the $N$ particles remain trapped for very long times in some so-called
metastable states. These are for instance regions located around local minima of $V$. This actually corresponds to a physical reality:
the timescale at the molecular level (given by $h$, which is
typically chosen at the limit of stability for the stochastic
differential equation) is much smaller than
the timescales of interest, which correspond to hopping events between
metastable states. Let us denote by $A \subset \R^{3N}$ and $B \subset
\R^{3N}$ two (disjoint) metastable states. The problem is then the following: for some initial condition outside $A$ and $B$, how to efficiently sample
paths which reach $B$ before $A$? In the context of molecular
dynamics, such paths are called reactive paths. The efficient sampling of
reactive paths is a very important subject in many applications since
it is a way to understand the mechanism of the transition between
metastable states. In mathematical terms, one is
interested in computing, for a given test function $\varphi:
(\R^{3N})^\N \to \R$ depending on the path $(X_t)_{t\in\N}$ of the Markov chain, the expectation
\begin{equation}\label{eq:interest}
\E\Bigl(\varphi \bigl( (X_t)_{t\in\N} \bigr) \mathds{1}_{\tau_B < \tau_A}\Bigr)
\end{equation}
where $\tau_A=\inf \{t \in \N :  X_t \in A\}$, $\tau_B=\inf\{t \in \N :
     X_t \in B\}$ and $X_0=x_0\notin (A\cup B)$ is assumed (for simplicity) to be a deterministic initial position close to $A$: most trajectories starting from $x_0$ hit $A$ before $B$. If $\varphi=1$, the above expectation is $\P(\tau_B
    < \tau_A)$, namely the probability that the Markov chain reaches
    $B$ before $A$. This is typically a very small probability: since $A$ is
    metastable and $x_0$ is close to $A$,  for most of the
    realizations, $\tau_A$ is smaller than $\tau_B$. This is why naive Monte Carlo methods
    will not give reliable estimates of~\eqref{eq:interest}.    
    We refer for example to \cite{BrehierGazeauGoudenegeRousset,CerouGuyaderLelievrePommier2011} for some examples in the context of molecular simulation.

\subsection{The adaptive multilevel splitting algorithm}

Many techniques have been proposed in the literature in order to
compute quantities such as~\eqref{eq:interest}, in particular control variate
techniques, importance sampling methods and splitting algorithms (see
for example the monograph~\cite{Bucklew2004} on rare event simulations). Here,
we focus on the Adaptive Multilevel Splitting (AMS) method which has been
proposed in \cite{CerouGuyader2007}. Let us roughly
describe the principle of
the method. The crucial ingredient we need is an
importance function: 
\begin{equation}
\xi: \R^{3N} \to \R
\end{equation}
which will be used to measure the advance of the paths towards $B$. 
This function is known as a reaction coordinate in the molecular
dynamics community, and this is the terminology we will use here. In this paper, we will also call $\xi(X_t)$ the
level of the process $X_t$ at time $t$. A useful requirement on $\xi$ is the existence of $z_{\rm max}\in\R$ such that
$$ B \subset \{x \in \R^{3N} : \xi(x) \in ]z_{\rm max},\infty[\}.$$
For any path of the Markov chain, we call the maximum level of this path the quantity
$$\sup \{ \xi(X_{t \wedge \tau_A})_{t \in \N}\}.$$
Then, starting from a system of $n_{rep}$ replicas (all starting from the same initial condition $x_0$ and stopped at time $\tau_A$),
the idea is to remove the \emph{worst fitted} paths and to duplicate the \emph{best fitted}
paths while keeping a fixed number of replicas (we will discuss below
generalizations of the AMS algorithm where the number of replicas may
vary). The worst fitted paths are
those with the smallest maximum levels $\sup\{\xi(X_{t \wedge
  \tau_A})_{t \in \N}\}$. As soon as one of the worst fitted paths is
removed, it is replaced by resampling of one of the best fitted
path: the new path is a copy of the best fitted path up to the maximum
level of the removed paths, and the end of the trajectory is then sampled using 
independent random numbers. The algorithm thus goes through three steps: (i) the
level computation step (to determine the level under which paths will
be removed: this level is computed as an empirical quantile over the maximum levels among the replicas);
(ii) the splitting step (to determine  which paths will be removed and which ones of the remaining
best fitted paths will be duplicated); (iii) the resampling step (to
generate new paths from the selected best fitted paths). By iterating
these three steps, one obtains successively systems of $n_{rep}$
paths with an increasing minimum of the maximum levels among the replicas. The
algorithm is stopped when the current level is larger than $z_{\rm max}$, and an
estimator of~\eqref{eq:interest} is then built using a weighted empirical
average over the replicas. The adaptive feature of the algorithm is in
the first step (the level computation step): indeed, at each
iteration, paths are
removed if their maximum level is below some threshold, and these thresholds
are determined iteratively using empirical quantiles,
rather than by fixing a priori a deterministic sequence of levels (as
it would be the case in non-adaptive splitting, or more generally in standard sequential Monte Carlo algorithms,
see~\cite{CerouDel-MoralFuronGuyader2012,Del-Moral2004}). All the details of the algorithm will be
given in Section~\ref{sec:AMS_Markov}.

In this work we  focus on the application of the AMS algorithm to sample Markov chains, namely
discrete time stochastic dynamics, and not continuous time
stochastic dynamics as in~\cite{CerouGuyader2007} for example. The reason is mainly practical:
in most cases of interest, even if the original model is
continuous in time, it is discretized in time when numerical
approximations are needed.
There are actually also many cases where the original model is
discrete in time (for example kinetic Monte Carlo or Markov State Models in the context of molecular dynamics).

The discrete time setting, which is thus of practical interest, raises
specific questions in the context of the AMS algorithm. First, in the
resampling step, a natural question is whether the path should be
copied up to the last time before or first time after it reaches the level of the removed
paths.  Second, in the discrete time context, it may happen that several paths have exactly the same
maximum level. This implies some subtleties in the implementation
of the splitting step which have a large influence on the quality of
the estimators, see Section~\ref{sect:BrownDrift}.

\subsection{Main results and outline}

The main results of this work are the following:
\begin{itemize}
\item We prove that the AMS algorithm for Markov chains with an appropriate
  implementation of the level computation and splitting steps yields an unbiased estimator
  of the rare event probability, and more generally of any
  non-normalized expectation related to the rare event of the
  form~\eqref{eq:interest}. We actually prove this unbiasedness result
  for a general class of splitting algorithms which enter into what we
  call the Generalized
  Adaptive Multilevel Splitting (GAMS) framework.
\item Using this GAMS framework, we propose various generalizations of the classical AMS
  algorithm which all yield unbiased estimators, in particular to
  remove extinction and to reduce the computational cost associated with
  sorting procedures. Moreover, we explain how to use the general setting to sample other random variables than
  trajectories of Markov chains.
\item We illustrate numerically on toy examples the importance of an appropriate
  implementation of the level computation and splitting steps in the
  AMS algorithm to get unbiasedness. We also
  discuss through various numerical experiments the
  influence of the choice of the reaction coordinate $\xi$ on the
  variance of the estimators and we end up with some practical recommendations in order to get reliable estimates
using the AMS algorithm (see Section~\ref{sec:conc}). In particular, using the
unbiasedness property proven in this paper, it is possible to compare
the results obtained using different parameters (in particular
different reaction coordinates) in order to assess the quality of the
numerical results.
\end{itemize}

Compared to previous results in the literature concerning the AMS
algorithm, the main novelty of this work is the proof of the unbiasedness in a
general setting and whatever the parameters: the number of replicas,
the (minimum) number of resampled replicas at each iteration and the
reaction coordinate~$\xi$. The proof of unbiasedness relies on the
interpretation of the AMS algorithm as a sequential Monte Carlo
algorithm in path space with the reaction coordinate as a time index,
in the spirit of~\cite{JohansenDel-MoralDoucet2005} (the selection and
mutation steps respectively corresponds to the
branching step and the resampling step in the AMS algorithm). This analogy is made precise in Section~\ref{sec:SMC}.
 In previous works, see for instance~\cite{BrehierLelievreRousset2015,GuyaderHentgartnerMatzner-Lober2011,Simonnet2014},
unbiasedness is proved in an idealized setting, namely when the reaction coordinate is given by $\xi(x)=\P_x(\tau_B < \tau_A)$ (known as the committor function; here, the subscript $x \in \R^{3N}$ indicates that the Markov chain $X_t$ has $x$ as an initial condition), and for a different resampling step, where new
replicas are sampled according to the conditional distribution of
paths conditioned to reach the level of the removed replicas. In many
cases of practical interest, these two conditions are not met.

In addition, we illustrate through extensive numerical experiments the influence of the choice of $\xi$ on the
variance. Indeed, as for any Monte Carlo algorithm, the bias is only one part of the error
when using the AMS algorithm: the statistical error (namely the
variance) also plays a crucial role in the quality of the estimator as
will be shown numerically in Section~\ref{sect:num}. There are
unfortunately very few theoretical results concerning the influence of the choice
of $\xi$ on the statistical error. We refer
to~\cite{BrehierGoudenegeTudela,CerouDel-MoralGuyaderMalrieu2014} for
an analysis of the statistical error. For discussions of the role of $\xi$ on the statistical
error, we also refer
to~\cite{GarvelsKroeseVan-Ommeren2002,GlassermanHeidelbergerShahabuddinZajic1998,RollandSimonnet2015}. In
particular, in the numerical experiments, we discuss situations for
which the confidence intervals of the estimators associated with different reaction
  coordinates do not overlap if the number of independent realizations
  of the algorithm is not sufficiently large. We relate this observation to the well-known
phenomenon of ``apparent bias'' for splitting algorithms,
see~\cite{GlassermanHeidelbergerShahabuddinZajic1998}.

We would like to stress that our results hold in the setting where a
family of resampling kernels indexed by the levels is available (see Section~\ref{sec:resampling} for a
precise definition). 
This is 
particularly well suited to the sampling of trajectories of
Markov dynamics (see Section~\ref{sec:bridge} for another possible
setting). In the terminology of~\cite{JohansenDel-MoralDoucet2005}, we
have in mind the dynamic setting (considered for example in~\cite{CerouGuyader2007}), and not  the
static setting (considered for example in~\cite{CerouDel-MoralFuronGuyader2012,CerouDel-MoralGuyaderMalrieu2014}).

The paper is organized as follows. The AMS algorithm applied
to the sampling of paths of Markov chains is described in
Section~\ref{sec:Markov}. This algorithm actually enters into a more
general framework, the Generalized Adaptive Multilevel Splitting
(GAMS) framework which is described in detail in
Section~\ref{sect:GAMS}. The interest of this generalized setting is
twofold. First, it is very useful to write variants of the
classical AMS algorithm which will still yield unbiased estimators
(some of them are described in Section~\ref{sec:variants}). Second,
it highlights the essential mathematical properties
that are required to produce unbiased estimators of quantities such
as~\eqref{eq:interest}. This is the subject of Section~\ref{sec:unbiased} which is
devoted to the main theoretical result of this work: the unbiasedness
of some estimators, including estimators
of~\eqref{eq:interest}. Finally, Section~\ref{sect:num} is entirely
devoted to some numerical experiments which illustrate the
unbiasedness result, and discuss the efficiency of the AMS algorithm
to sample rare events.

\subsection{Notation}\label{sec:notation}
Before going into the details, let us provide a few general notations
which are useful in the following.
\begin{itemize}
\item The underlying probability space is denoted by
  $(\Omega,\mathcal{F},\P)$. We recall standard notations: for $m$ $\sigma$-fields
  $\calF^1,\ldots,\calF^m \subset \calF$,
  $\calF^1\vee \ldots \vee \calF^m$ denotes the smallest
  $\sigma$-field on $\Omega$ containing all the $\sigma$-fields
  $\calF^1,\ldots,\calF^m$. For any $t,s\in\N$, $t\wedge
  s=\min\{t,s\}$ and $t\vee s=\max\{t,s\}$. We use the convention $\inf \emptyset = + \infty$. For
  two sets $A$ and $B$ which are disjoint,  $A \sqcup B$ denotes the disjoint
set union.


\item We work in the following standard setting: random variables
  take values in state spaces $\calE$ which are Polish (namely
  metrizable, complete for some distance $d_{\calE}$ and
  separable). The associated Borel $\sigma$-field is denoted by
  $\calB(\calE)$. We will give precise examples below (see for example
  Section~\ref{sec:Markov_path} for the space of trajectories for Markov chains).

Then $\Proba(\calE)$ denotes the set of probability distributions on
$\calE$. It is endowed with the standard Polish structure associated with
the Prohorov-Levy metric which metrizes convergence in distribution,
{\it i.e.} weak convergence of probabilities tested on continuous and
bounded test functions (see for example~\cite{billinsgley-99}).

The distribution of a $\calE$-valued random variable $X$ will be denoted by $\Law(X)$.


\item If $\calE_1$ and $\calE_2$ are two Polish state spaces, a Markov
  kernel (or transition probability kernel) $\Pi(x_1,dx_2)$ from $\calE_1$
  to $\calE_2$  is a measurable map from initial states in $x_1 \in
  \calE_1$, to probability measures in $\Proba(\calE_2)$.

\item We use the following standard notation associated with probability transitions: for $\ph:\calE_2\rightarrow \R$ a bounded and measurable test function, 
\begin{equation}\label{eq:pi_phi}
\Pi(\ph)(x_1)= \int_{x_2 \in \calE_2} \ph(x_2) \Pi(x_1,dx_2).
\end{equation}
Similarly, we use the notation $\pi(\ph) = \int_{x \in \calE_2} \ph(x) \pi(dx)$ for $\pi \in \Proba(\calE_2)$.
\item Let $X_1$ and $X_2$ be random variables respectively with values
  in $\calE_1$ and $\calE_2$ and $\Pi$ a Markov kernel from $\calE_1$
  to $\calE_2$. In the algorithms we describe below, we
  will use the notion of conditional sampling:
  {$X_2$ is sampled conditionally on
    $X_1$} with law $\Pi(X_1, \, . \,)$ (denoted by $X_2 \sim \Pi(X_1, \, . \,)$)
rigorously means that $X_2 = f(X_1,U) \as$ where, on the one hand, $U$ is
some random variable independent of $X_1$ and of all the random variables introduced
before (namely at previous iterations of the algorithm) and, on
the other hand, $f$ is a measurable function which is such that $\Pi(x_1, \, . \,) =
\Law(f(x_1,U))$, for $\Law(X_1)$-almost every $x_1 \in \calE_1$.


\item


A random system of replicas in $\calE$ is denoted by
\begin{equation}\label{eq:def_Erep}
\calX=\p{X^{(n)}}_{n \in I} \in \calE^{\rm rep}, \quad \card I < + \infty,
\end{equation}
where $I \subset \N^\ast$ is a random finite subset of labels and
$(X^{(n)}) _{n \in I}$ are elements of $\calE$. The
space~$\calE^{\rm rep}$ is endowed with the following distance: for $\calX^1=\p{x^{(1,n)}}_{n \in
  I^1}$ and $\calX^2=\p{x^{(2,n)}}_{n \in  I^2}$ in $\calE^{\rm rep}$,
we set
\begin{equation*}
d(\calX^1,\calX^2)=
\left\{
\begin{array}{ll}
2 & \text{ if } I^1 \neq I^2,\\
\min\left\{ \sum_{n \in I^1} d_{\calE}(x^{(1,n)},x^{(2,n)}), 1 \right\} & \text{ if } I^1=I^2.
\end{array}
\right.
\end{equation*}
Endowed with this distance, the set $\calE^{\rm rep}$ is Polish and we
denote by $\calB(\calE^{\rm rep})$ the Borel $\sigma$-field. This
$\sigma$-field can also be written as follows:
$$\calB(\calE^{\rm rep})=\bigsqcup_{I \in \mathcal{I} } \calB(\cal
E)^{\card I}$$
where $\mathcal{I}$ denotes the ensemble of finite subsets of
$\N^\ast$ (which is a discrete set).
\comment{TONY: OK avec ca ?}





\item When we consider systems of weighted replicas, to each
  replica $X^{(n)}$ of the system $\calX$ with label $n\in I$  is
  attached a weight $G^{(n)} \in \R_+$, and we use the notation
  $\calX=\bigl(X^{(n)},G^{(n)}\bigr)_{n \in I}$. The topological setting is the
  same as in the previous item, $\calE$ being replaced by the augmented state space $\calE
  \times \R$.


\end{itemize}

\section{The AMS algorithm for Markov chains}\label{sec:Markov}

The two goals of this section are to define the AMS
algorithm applied to paths of a Markov chain (namely a discrete time
stochastic process) and to introduce unbiased estimators in this setting.

A special care should be taken to treat the
situations when many replicas have the same maximum level, or the
situations when there is extinction of the population of
replicas. These aspects which are specific to the discrete time setting were not treated in details in many previous
works where continuous time diffusions were considered.

\subsection{The Markov chain setting}\label{sec:Markov_path}

Let $\tX=(\tX_{t})_{t \in \N}$ be a Markov chain defined on a probability
space $(\Omega,\calF,\PP)$, with probability transition $P$. We assume
that $\tX_t$ takes values in a Polish state space $\calS$. Without loss of generality, we assume that $\tX_0=x_0$ where
$x_0\in\calS$ is a deterministic initial condition. The generalization
to a random initial condition $\tX_0$ is straightforward.

The path space is denoted by
\begin{equation}\label{eq:P_AMS}
\calP =\left\{x=(x_t)_{t\in \N} : \, x_t\in \calS \text{ for all } t \in \N\right\}.
\end{equation}
It is well-known that, by introducing the distance $d_\calP(x,y) =
\sum_{t \in \N} \frac{1}{2^t} \bigl(1 \wedge \sup_{s \leq t}
d_\calS(x_s,y_s)\bigr)$ (which is a metric for the product topology), the space $(\calP,d_\calP)$ is complete and
separable. We denote by $\calB(\calP)$ the corresponding Borel
$\sigma$-field. We thus see $\tX$ as a random variable with values
in $\calP$.

The set of paths $(\calP,\calB(\calP))$ is endowed with the natural
filtration in time $(\calB_t)_{t \in \N}$: $\calB_t \subset
\calB(\calP)$ is the smallest $\sigma$-field such that $(x_s)_{s \in
  \N} \in \calP \mapsto (x_1, \ldots , x_t) \in \calS^t$ is
measurable. Observe that the natural filtration for a given Markov chain $\tX$ defined
on $(\Omega,\calF,\PP)$ with values in $\calS$ is then given by the
pullback of the filtration $(\calB_t)_{t \in \N}$ by
$\tX: (\Omega,\calF) \to (\calP,\calB(\calP))$.

\subsection{The rare event  of interest}

Given two disjoint Borel subsets $A$ and $B$ of $\calS$, our main objective is the efficient sampling of events such as $\left\{\tau_B<\tau_A\right\}$ where
\[
\tau_{A}=\inf\left\{t\in \N : \tX_t\in A\right\} \qquad \text{and}
\qquad  \tau_{B}=\inf\left\{t\in \N : \tX_t\in B\right\}
\] 
are respectively the first entrance times in $A$ and $B$. Both
$\tau_{A}$ and $\tau_{B}$ are stopping times with respect to the
natural filtration of the process $\tX$.

We are mainly interested in the estimation of the probability $\P(\tau_B<\tau_A)$ in the rare event
regime, namely when this probability is very small (typically less
than $10^{-8}$). This occurs for example if the initial condition $x_0
\in  A^c \cap B^c$ is such that $x_0$ is close to $A$, and $A$ and
$B$ are metastable regions for the dynamics. The Markov
chain starting from (a neighborhood of) $A$ (resp. $B$) remains for a very long time near
$A$ (resp. $B$) before exiting, and thus, the Markov chain starting from
$x_0$ reaches $A$ before $B$ with a probability close to
one. Specific examples will be given in Section~\ref{sect:num}.

Let us introduce the Markov chain stopped at time $\tau_A$: $X=\p{X_t}_{t\in\N}$ where
\begin{equation}\label{eq:X_Path}
X_t=\tX_{t\wedge \tau_A} \quad \text{for any}~ t\in \N.
\end{equation}


The probability distribution of the stopped Markov chain $X$ (seen as a $\calP$-valued random variable) is denoted by
\begin{equation}\label{eq:Pi_Path}
\pi \eqdef \Law \p{ X } \in \Proba(\calP).
\end{equation}
The probability of interest can be rewritten
\[
\P(\tau_B<\tau_A)=\E\p{\mathds{1}_{\tT_B(X)<\tT_A(X)}} 
\]
where
we denote for any path $x \in \calP$
$$\tT_{A}(x)=\inf\left\{t\in \N : x_t\in A\right\} \quad \text{ and } \quad \tT_{B}(x)=\inf\left\{t\in \N : x_t\in B\right\}.
$$

More generally, the algorithm allows us to estimate expectations of the
following form:
\begin{equation}\label{eq:phi}
\pi(\ph) = \E \p{ \ph( X)},
\end{equation}
for any observable $\ph:\calP\rightarrow \R$ such that $\pi(|\ph|)$ is
finite.

\begin{Rem}[On the stopping times $\tau_A$ and $\tau_B$]
We defined above the stopping times as first entrance times in some
sets $A$ and $B$. As will become clear below, the definition of the algorithm and
the unbiasedness result only require $\tau_A$ and $\tau_B$ to be
stopping times with respect to the natural filtration of the chain
$\tX$ ({\it i.e.} $\tT_A$ and $\tT_B$ to be stopping times on
$(\calP,\calB(\calP))$ endowed with the natural
filtration). 
\end{Rem}

\subsection{Reaction coordinate}

The crucial ingredient we need to introduce the AMS algorithm is an
importance function, also known as a reaction coordinate or an order
parameter in the context of molecular dynamics. This is a measurable $\R$-valued mapping defined on the state space $\calS$:
\[
\xi : \calS \to \R.
\]
The choice of a good function $\xi$ for given sets $A$ and $B$ is a
difficult problem in general. One of the main aims of this paper is to show that
whatever the choice of $\xi$, it is possible to define an unbiased
estimator of~\eqref{eq:phi}. The only requirement we impose on $\xi$
is that there exists a constant $z_{\rm max} \in \R$ such that
\begin{equation}\label{eq:hyp_B}
B \subset \xi^{-1}\p{ ] z_{\rm max} , + \infty [ }.
\end{equation}

In what follows, the values of $\xi$ are called levels and
we will very often refer to the maximum level of a path, defined as follows:
\begin{Def}
For any path $x\in\calP$, the {\em maximum level} of $x$ is defined as the supremum of $\xi$ along the path $x$ stopped at $\tT_A(x)$:
\begin{equation}
  \label{eq:max}
 \Xi(x) \eqdef \sup \{ \xi(x_{t \wedge \tT_A(x)})_{t \in \N} \} \in\R\cup\left\{+\infty\right\}.
\end{equation}
\end{Def}
The function $\Xi$ can be seen as a reaction coordinate on the path
space $\calP$.


We also introduce for any level $z\in\R$ and any path $x\in\calP$
\begin{equation}\label{eq:Tz}
\tT_{z}(x) \eqdef \inf \{ t \in \{0, \ldots,\tT_A(x)\} :   \xi(x_t) > z \},
\end{equation}
which is the first entrance time of the path $x$ stopped at
$\tT_A(x)$ in the set $\xi^{-1}(]z,+\infty[)$. We emphasize on the strict inequality
in the above definition of the entrance times $\tT_z(x)$: it is one of the
important ingredients of the proof of the unbiasedness of the
estimator of~\eqref{eq:phi}. Notice that the above
assumption~\eqref{eq:hyp_B} on $B$ is equivalent to the
inequality $$\forall x\in\calP, \, \tT_{z_{\rm
    max}}(x) \leq \inf \{ t \in \{0, \ldots,\tT_A(x)\} :   x_t \in B \}.$$


We denote by
\begin{equation}\label{eq:tauz}
\tau_{z} \eqdef \tT_z(X)= \inf \{ t \in \{0, \ldots,\tau_A\} : \xi(X_t) > z \}.
\end{equation}
the entrance time associated with the (stopped) Markov chain $X$. It is a stopping time for the natural filtration of the Markov chain.

\begin{Rem}[On Assumption~\eqref{eq:hyp_B}]
Assumption~\eqref{eq:hyp_B} is extremely useful in practice when
computing approximations of averages of the form $\E\p{\ph(X)
  \mathds{1}_{\tau_B<\tau_A}}$: it allows to remove from
memory the replicas which are declared ``retired'' in the splitting
step at each iteration in the AMS algorithm described in the sequel, since by
construction we know in advance that they will not contribute to the
computation of the associated estimator. The algorithm thus only
requires to retain a fixed number of replicas, denoted by $\nrep$ below.
\end{Rem}

\subsection{Resampling kernel}\label{sec:resampling_kernel}

The AMS algorithm is based on an interacting system of weighted replicas. At each iteration, copies of the Markov chain, called replicas, are simulated (in parallel) and are ranked
according to their maximum level \eqref{eq:max}.
The less fitted trajectories, {\it i.e.} those with the lowest values of the maximum levels, are
resampled according to a resampling kernel.
For any $z\in\R$ the resampling kernel $\pi_z$ (which is in fact a
transition probability kernel from $\R \times \calP$ to $\calP$) is denoted by
\begin{equation}\label{eq:piz_chain_1}
\pi_z:\left\{
\begin{aligned}
&\calP\rightarrow \text{Proba}(\calP)\\
&x\mapsto \pi_z(x,dx')
\end{aligned}
\right.
\end{equation}
and is defined as follows: for any $x\in\calP$, $\pi_z(x,dx')$ is the
law of the $\calP$-valued  random variable $Y$ such that
\begin{equation}
  \label{eq:piz_chain}
  \begin{cases}
Y_t = x_t  & \text{if} \, t \leq \tT_z(x) \\
\Law( Y_t \st   Y_s, \, 0 \leq s \leq t-1 ) = P( Y_{t-1}, \, . \, ) & \text{if} \,  t  > \tT_z(x)
\end{cases}
\end{equation}
and is stopped at $\tT_A(Y)$ when $Y$ hits $A$. We recall that $P$ is the transition kernel of the Markov chain $X$. In other words, for $t\leq \tT_z(x)$, $Y$ is identically $x$, while for $t>\tT_z(x)$, $Y_t$ is generated according to the Markov dynamics on $\calS$, with probability transition $P$, and stopped when reaching $A$. We thus perform a branching of the path $x$ at time $\tT_z(x)$ and position $x_{\tT_z(x)}$.

Notice from the definition of the resampling kernel that if $\Xi(x)\leq z$, then the resampling kernel does not modify $x$: $\tT_z(x)=+\infty$ and $\pi_z(x,dx')$ is a Dirac mass: $Y_t=x_{t\wedge \tT_A(x)}$ for any $t\in\N$.


One of the important ingredients to prove the unbiasedness of the estimator
of~\eqref{eq:phi} is the following
right-continuity property: for all $x \in \calP$, and for any continuous bounded test function $\ph:\calP \to \R$, the mapping $$z \mapsto \int_{\calP} \ph(x') \pi_z(x,dx')$$ is right-continuous, see Section~\ref{sec:justif_Markov} for a proof of this statement.


\subsection{The AMS algorithm}\label{sec:AMS_Markov}

In this section, we introduce the AMS algorithm in the specific
context of sampling of Markov chain trajectories. The associated unbiased estimator of~\eqref{eq:phi} is given
in the next section. A generalized adaptive multilevel splitting
framework which encompasses the AMS algorithm described here will be
provided in Section~\ref{sect:GAMS}.

\paragraph{Short description of the AMS algorithm}


In addition to the choice of the reaction coordinate $\xi$, two other parameters of the algorithm need to be specified: $\nrep$, the number of replicas,
and $k \in \{1, \ldots, \nrep-1\}$ the (minimum) number of replicas
resampled at each step of the algorithm.  At the $q$-th iteration, the replicas
of the Markov chain $X$ are denoted by $X^{(n,q)}$ 
 where $n$ is the label of the replica. The algorithm defines a non-decreasing
 sequence of random levels $Z^{(q)}$.  At the beginning of iteration $q$, the level $Z^{(q)}$
 is the $k$-th order statistics of the maximum levels of
 the ``working'' replicas $X^{(n,q)}$ (the notion of ``working'' replicas is defined in the full description of the algorithm). Then, all replicas
 with maximum levels lower or equal to $Z^{(q)}$ are declared  ``retired", and  resampled in
 order to keep a fixed number $\nrep$ of replicas with maximum level
 strictly larger than $Z^{(q)}$.  As explained above, the resampling procedure consists in duplicating one of
 the replica such that its maximum level is larger than $Z^{(q)}$ up
 to the time $\tau_{Z^{(q)}}$, and then in using the resampling kernel, which amounts in
 completing the trajectory up to time $\tau_A$ with the Markov
 transition kernel $P$. The set of labels of the $\nrep$ replicas obtained at the end of the
 $q$-th iteration and which, by construction, all have a maximum
 level larger than  $Z^{(q)}$ is denoted by~$I_{\rm
   on}^{(q+1)}$. The subscript ``on'' indicates that these are the
 replicas which have to be retained to pursue the algorithm (the
 so-called ``working'' replicas in the terminology introduced below). Consistently, the subscript ``off'' refers to the ``retired'' replicas at a given iteration.
 The algorithm stops either if the level $z_{\rm max}$ is reached
 (namely $Z^{(q)}>z_{\rm max}$) or if all the replicas at the end of
 iteration $q$ have maximum levels lower than or equal to the $k$-th
 order statistics. When the algorithm stops, an unbiased estimator of~\eqref{eq:phi} is then defined as a weighted empirical average over the replicas.

\paragraph{Full description  of the AMS algorithm}

The AMS algorithm generates iteratively a system of weighted
replicas in the state space $\calP^{\rm rep}$, using selection and
resampling steps.

In order to define an estimator of $\pi(\varphi)$ for any observable
$\varphi$, we will need to consider the set
$I_{\rm off}^{(q)}$ of all the labels of the replicas which have been declared retired before
iteration $q$ (namely those with a maximum level smaller or equal than $Z^{(q-1)}$). We will denote by $I^{(q)}= I_{\rm on}^{(q)} \sqcup I_{\rm
  off}^{(q)}$ the set of all the labels of the replicas generated by the
algorithm up to iteration $q$. We recall that $\sqcup$ denotes the disjoint
set union. Notice that the
 cardinal of $I^{(q)}$ is increasing, while $\card I_{\rm on}^{(q)} =
 \nrep$ for any $q$. At step $q$, the replicas with labels in $I_{\rm
   on}^{(q)}$ are referred to as the working replicas, and
 the replicas with labels in $I_{\rm off}^{(q)}$ as the retired replicas.
The unbiased estimator of~\eqref{eq:phi} associated with the 
 AMS algorithm will be defined in the next section.

We are now in position to introduce the AMS algorithm in full detail (see Figure~\ref{Fig:AMS} for a schematic representation of one iteration of the algorithm).

\paragraph{The initialization step ($q=0$)}
\begin{enumerate}[(i)]
\item Let $\p{ X^{(n,0)} }_{1 \leq n \leq \nrep }$ be i.i.d. replicas
  in $\calP $ distributed according to $\pi$, defined
  by~\eqref{eq:Pi_Path}. At this initial stage, all replicas are
  working replicas {\it i.e.} $I^{(0)}=I_{\rm on}^{(0)} =\{1,\ldots,\nrep\} $ and $I_{\rm off}^{(0)}=\emptyset$.
\item Initialize uniformly the weights: $G^{(n,0)} = 1/\nrep $ for $n
  \in \{1, \ldots , \nrep\}$.
\item Compute the order statistics of $\p{ \Xi( X^{(n,0)}) }_{n \in
    I_{\rm on}^{(0)}}$:  namely a permutation  $\Sigma^{(0)}$ of the set of labels $I_{\rm on}^{(0)}=\left\{1,\ldots,n_{\rm rep}\right\}$ such that
\[
\Xi( X^{(\Sigma^{(0)}(1),0)}) \leq \ldots \leq \Xi( X^{(\Sigma^{(0)}(n_{\rm rep}),0)}) 
\]
and set the initial level as the $k$-th order
statistics\footnote{Notice that $\Sigma^{(0)}$ is not necessarily
  unique since several replicas may have the same
 maximum  level. Nevertheless, the level $Z^{(0)}$ does not depend on the
  choice of $\Sigma^{(0)}$. The same remark applies to the definition
  of the level $Z^{(q+1)}$ at iteration $q \ge 0$, see Remark~\ref{rem:same_max}.}: 
\[
Z^{(0)} = \Xi( X^{(\Sigma^{(0)}(k),0)}).
\]
\item If $\card \left\{n \in
    I_{\rm on}^{(0)} :  \Xi( X^{(n,0)}) \leq Z^{(0)} \right\}=\nrep$, then set $Z^{(0)} = +\infty.$

\end{enumerate}

\noindent

\paragraph{Iterations} Iterate on $q\geq 0$, while the following stopping criterion is not satisfied.


\paragraph{The stopping criterion}

If $Z^{(q)}>z_{\rm max}$, then the algorithm stops. When it is the case, set $Q_{\rm iter}=q$.
Else perform the following four steps.

\paragraph{The splitting (branching) step}

\begin{enumerate}[(i)]
\item Consider the following partition of the working replicas' labels in $I_{\rm on}^{(q)}$:
\[
I_{\rm on}^{(q)} = I_{ {\rm on}, \leq Z^{(q)}}^{(q)} \sqcup I_{ {\rm on}, > Z^{(q)}}^{(q)} 
\]
where replicas with maximum level smaller or equal than $Z^{(q)}$ have labels in
\[ 
I_{ {\rm on}, \leq Z^{(q)}}^{(q)}  =\set{ n \in  I_{\rm on}^{(q)}, \,   \Xi( X^{(n,q)})  \leq Z^{(q)}  },
\]
while the set of replicas' labels with maximum level strictly larger than $Z^{(q)}$ is
\[ 
I_{ {\rm on}, > Z^{(q)}}^{(q)}  =\set{ n \in  I_{\rm on}^{(q)}, \,    \Xi( X^{(n,q)})  > Z^{(q)}  }.
\]
We set $K^{(q+1)}=\card I_{ {\rm on}, \leq Z^{(q)}}^{(q)}=\nrep-\card
I_{ {\rm on}, > Z^{(q)}}^{(q)}$. Notice that $K^{(q+1)}\geq k$. The set $I_{\rm on}^{(q)}$ denotes
the working replicas at the beginning of iteration $q$. Among them, the replicas with
labels in~$I_{ {\rm on}, \leq Z^{(q)}}^{(q)}$ will be declared retired and
replaced by
branching replicas with labels in~$I_{ {\rm on}, > Z^{(q)}}^{(q)}$
using the resampling kernel $\pi_{Z^{(q)}}$, as
explained below.
Notice that necessarily, $\card I_{ {\rm on}, >
  Z^{(q)}}^{(q)}  \geq 1$ (otherwise $Z^{(q)}=+\infty$ and the stopping criterion has been fulfilled before entering the splitting step of iteration~$q$).

\item Introduce a new set  $I^{(q+1)}_{\rm new} = \set{\card I^{(q)}
    +1, \ldots , \card I^{(q)} +K^{(q+1)}}  \in \N^\ast \setminus I^{(q)}$ of labels for the new replicas sampled at iteration $q$.

\item Define the children-parent mapping $P^{(q+1)}:
I_{\rm
    new}^{(q+1)}\rightarrow I_{{\rm on}, >Z^{(q)}}^{(q)}$
    as follows: the labels $\Bigl(P^{(q+1)} \left( \card I^{(q)} +\ell \right) \Bigr)_{1\leq
  \ell\leq K^{(q+1)}}$ are $K^{(q+1)}$ random labels independently and uniformly distributed in $I_{ {\rm on}, > Z^{(q)}}^{(q)}$.

This mapping associates to the label of a new replica the label of its parent. The parent replica (with label in $I_{{\rm on}, >Z^{(q)}}^{(q)}$) is used in the resampling procedure to create the new replica (with label in $I_{\rm new}^{(q+1)}$).

\item For any $n\in I_{ {\rm on}, > Z^{(q)}}^{(q)}$, the branching number
\begin{equation}\label{eq:BN_AMS}
B^{(n,q+1)}=1+\card\left\{ n'\in I_{\rm new}^{(q+1)} : \, 
  P^{(q+1)}(n')=n\right\}
\end{equation}
represents the number of offsprings of $X^{(n,q)}$.
The replica $X^{(n,q)}$ will be split into
$B^{(n,q+1)}$ replicas: the old one $X^{(n,q)}$ with label $n\in I_{{\rm on},
  >Z^{(q)}}^{(q)}$ and, if  $B^{(n,q+1)} > 1$, $B^{(n,q+1)}-1$ new ones with labels $n'\in I_{\rm
  new}^{(q+1)}$ such that $P^{(q+1)}(n')=n$ .

\item The sets of new labels are then updated as follows:
$$
I_{\rm on}^{(q+1)}=I_{ {\rm on}, > Z^{(q)}}^{(q)}\sqcup I_{\rm new}^{(q+1)}, \quad
I_{\rm off}^{(q+1)}=I_{\rm off}^{(q)}\sqcup I_{ {\rm on}, \leq Z^{(q)}}^{(q)},\quad
I^{(q+1)} = I_{\rm on}^{(q+1)} \sqcup I_{\rm off}^{(q+1)}.
$$


Notice that by construction $\card I_{\rm on}^{(q+1)} =\nrep$.

The weights are updated with the following rule:
\begin{equation}\label{eq:weight_AMS}
  \begin{cases}
    G^{(n,q+1)} \eqdef  G^{(n,q)}  & n \in I_{\rm off}^{(q+1)} \\
G^{(n,q+1)} \eqdef  \frac{ \nrep - K^{(q+1)} }{ \nrep }G^{(n,q)} & n \in   I_{ {\rm on}, > Z^{(q)}}^{(q)} \\
G^{(n,q+1)}  \eqdef G^{(P^{(q+1)}(n),q+1)}
& n\in I_{\rm new}^{(q+1)}.
  \end{cases}
  \end{equation}
Observe that for $n\in I_{ {\rm on}}^{(q+1)}$, \[G^{(n,q+1)}=\frac{\nrep-K^{(q+1)}}{\nrep}\frac{\nrep-K^{(q)}}{\nrep}\ldots \frac{\nrep-K^{(1)}}{\nrep}\frac{1}{\nrep}.\]
Moreover, the weight of a replica remains constant as soon as it is retired
(namely from the first iteration $q$ such that its label is in $I_{\rm off}^{(q+1)}$).

\end{enumerate}

\paragraph{The resampling step}

\begin{enumerate}[(i)]
\item Replicas in $I^{(q)}$ are not
  resampled: for $n\in I^{(q)}$, $X^{(n,q+1)}=X^{(n,q)}$.
\item For $n'\in I_{\rm new}^{(q+1)}$, $X^{(n',q+1)}$ is sampled according
  to the resampling kernel (defined in Section~\ref{sec:resampling_kernel}) $\pi_{Z^{(q)}}(X^{(P^{(q+1)}(n'),q)},dx')$. The new replica
  $X^{(n',q+1)}$ is thus obtained by branching its parent replica $X^{(P^{(q+1)}(n'),q)}$.
\end{enumerate}

\paragraph{The level computation step}

Compute the order statistics of $\p{ \Xi( X^{(n,q+1)}) }_{n \in
    I_{\rm on}^{(q+1)}}$, namely a bijective mapping
  $\Sigma^{(q+1)}:\left\{1,\ldots,n_{\rm rep}\right\}\rightarrow
  I_{\rm on}^{(q+1)}$ (we recall that $\card I_{\rm on}^{(q+1)}=\nrep$)  such that
\[
\Xi( X^{(\Sigma^{(q+1)}(1),q+1)}) \leq \ldots \leq \Xi( X^{(\Sigma^{(q+1)}(n_{\rm rep}),q+1)}) 
\]
and set the new level as the $k$-th order statistics: 
\begin{equation}\label{eq:Zq_AMS}
Z^{(q+1)} = \Xi( X^{(\Sigma^{(q+1)}(k),q+1)}).
\end{equation}
If $\card \left\{n \in
    I_{\rm on}^{(q+1)} :  \Xi( X^{(n,q+1)}) \leq Z^{(q+1)} \right\}=\nrep$ then set  $Z^{(q+1)} = +\infty$.

\paragraph{Increment}
Increment $q\leftarrow
q+1$, and go back to the stopping criterion step.

\bigskip

Notice that $Q_{\rm
  iter}$ is such that 
$$Q_{\rm iter}= \inf \{q \ge 0 : Z^{(q)} >  z_{\rm max} \}.$$
The number of times the loop consisting of the three steps (splitting
/ resampling / level computation) is performed is exactly $Q_{\rm
  iter}$.

If $Z^{(Q_{\rm
  iter})}=+\infty$, none of the working replicas at the
iteration $Q_{\rm iter}-1$ is above the new level $\Xi( X^{(\Sigma^{(Q_{\rm
  iter})}(k),Q_{\rm
  iter})})$ and thus, all of them would
have been declared retired at the iteration $Q_{\rm
  iter}$: this situation is referred to as extinction.

\begin{figure}[h!]
\centering
\includegraphics[scale=0.6]{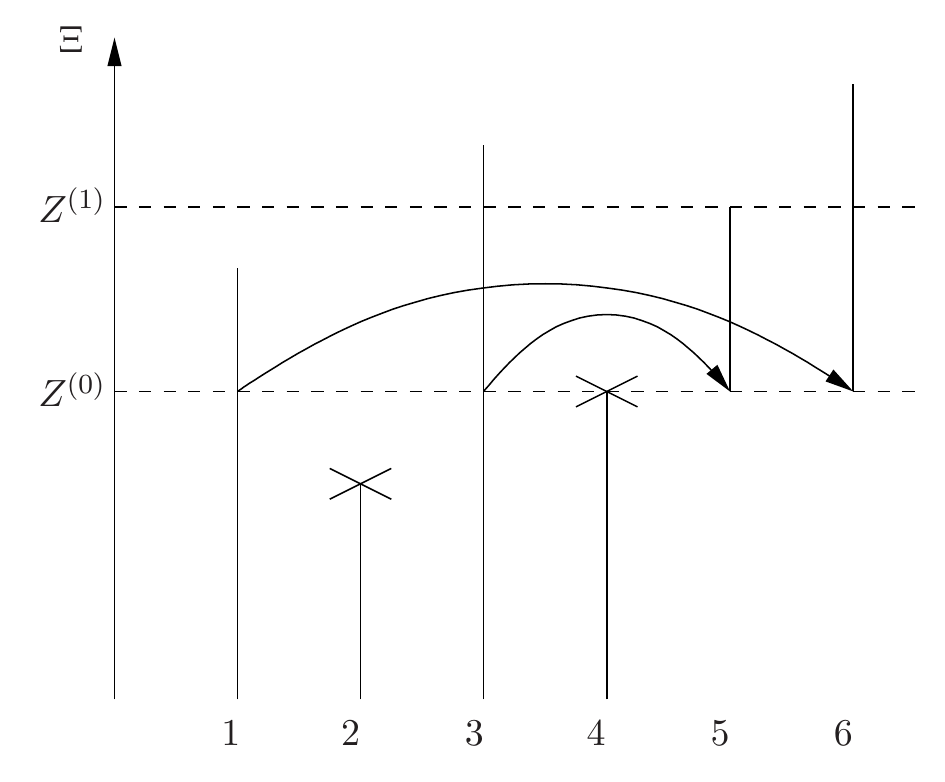}
\includegraphics[scale=0.6]{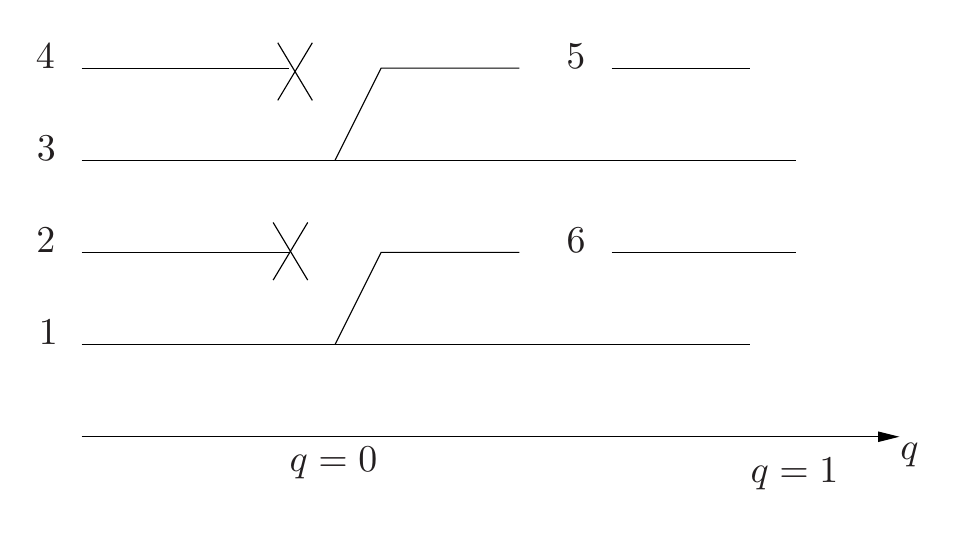}
\caption{Schematic representation of the first iteration of the AMS algorithm, with
  $\nrep=4$ and $k=2$. The replicas numbered 2 and 4 are declared retired at
  the first iteration, and are replaced by the replicas
  with label 5 and 6, which are respectively
  resampled from the replicas with labels 3 and 1.}
\label{Fig:AMS}
\end{figure}

\begin{Rem}[On the number of resampled replicas]\label{rem:same_max}
It is very important to notice that the number of resampled replicas
is at least $k$, but may be larger than $k$. In other words, at
iteration $q$, with the
above notation, $K^{(q+1)}$ is not necessarily equal to $k$. This
requires at least two replicas to have $Z^{(q)}$ as the maximum level at the
beginning of iteration~$q$. Actually, it may even happen that, in the
level computation step, all the
replicas in $I_{\rm on}^{(q+1)}$ have $Z^{(q+1)}$ as the maximum level, which  implies
extinction: $\card \left\{n \in
    I_{\rm on}^{(q+1)} :  \Xi( X^{n,q+1}) \leq Z^{(q+1)}
  \right\}=\nrep$, $Z^{(q+1)} = +\infty$ and the algorithm stops.

As an example, let us explain a three step procedure which leads two replicas to
have the same maximum level, which in addition is the minimum of the maximum levels over
all the working replicas, in the case $k=1$. 

Assume that in the splitting and resampling steps, the
three following events occur (see Figure \ref{Fig:Drifted_BM} for a
schematic representation):
\begin{enumerate}
\item One of the selected replica (referred to as $X$) has the
smallest maximum level among all the others in $I^{(q)}_{{\rm on},>Z^{(q)}}$.
\item The first time $\tT_{Z^{(q)}}(X)$ where this replica
goes beyond the current level $Z^{(q)}$ also corresponds to the
time at which this replica reaches its maximum level:
$$\xi\left(X_{\tT_{Z^{(q)}}(X)}\right)=\Xi(X).$$
\item By the
resampling procedure, the new replica (referred to as $Y$) which is generated (starting
from $X_{\tT_{Z^{(q)}}(X)}$ at time $\tT_{Z^{(q)}}(X)$, see the resampling kernel~\eqref{eq:piz_chain_1}--\eqref{eq:piz_chain}) is such that $\xi(Y_k) \leq
\xi(X_{\tT_{Z^{(q)}}(X)})=\Xi(X)$ for all $k > \tT_{Z^{(q)}}(X)$. Thus the new replica $Y$
has the same maximum level as the selected replica $X$: $\Xi(Y)=\Xi(X)$.
\end{enumerate}

\begin{figure}[h!]
\centering
\includegraphics[scale=0.4]{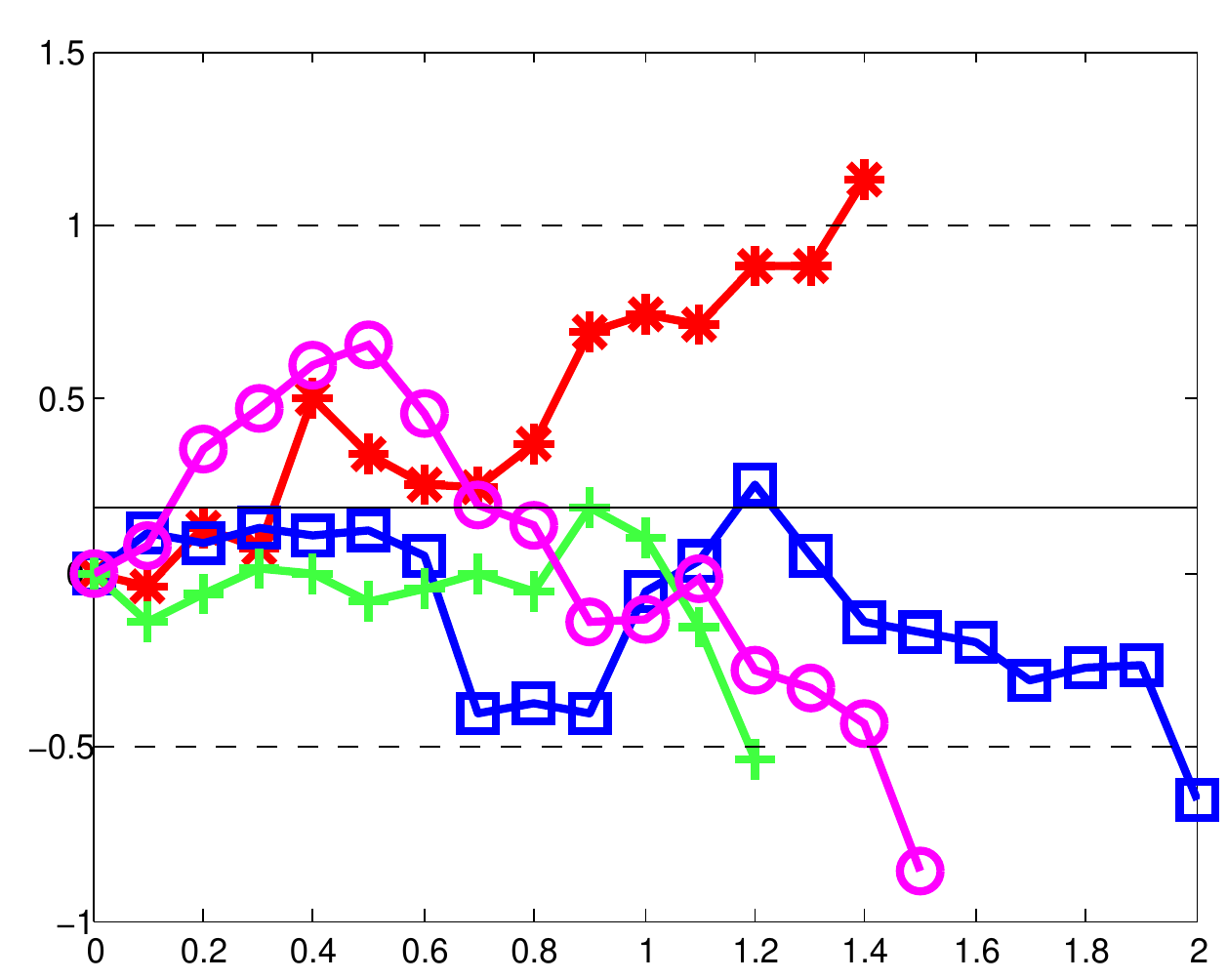}
\includegraphics[scale=0.4]{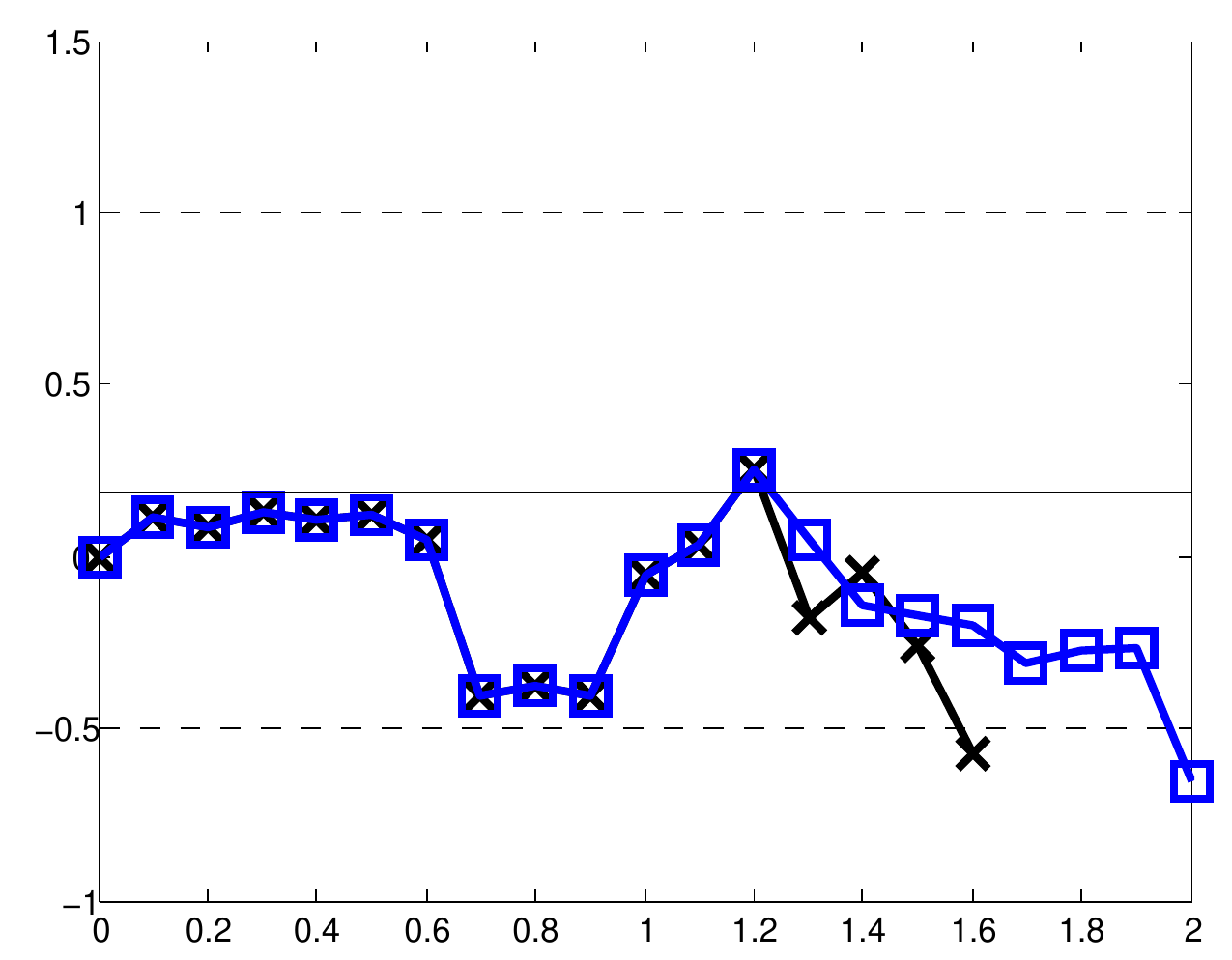}
\caption{Schematic representation of a procedure leading to the equality of the maximum levels
  ($y$-axis) of two replicas. We represent the evolution of the level of
  each replica at discrete times ($x$-axis). Left: the current level
  $Z^{(q)}$ is the maximum level of the green-crosses replica. The
  blue-squares replica has been selected to be resampled. Right:
  action of the resampling, where the new black-crosses replica has
  the same maximum level as the selected replica. Here $A=\left\{x 
    : \xi(x)<-0.5\right\}$ and $B=\left\{x :\xi(x)>1\right\}$, and $z_{\rm max}=1$.
}
\label{Fig:Drifted_BM}
\end{figure}

Then, at the next iteration, one obtains two replicas which have the
same maximum level, which is the minimum of the maximum levels over
all working replicas.
As a consequence, both replicas will be resampled at the next
iteration of the algorithm, even if $k=1$.

By iterating the procedure, one can thus obtain
many replicas having this same maximum level, or even all replicas
having the same maximum level (which leads to extinction).
%
%
%
Other similar procedures can lead to the equality of the maximum levels of two (or more) different replicas. For instance, even if the selected replica in the first step of the procedure above does not have the smallest maximum level among others (the first step is thus skipped), the two next steps will still create two replicas with the same maximum level. This implies that in some next iteration of the algorithm more than $k$ replicas will be declared retired and resampled.

The three events above have small probabilities especially if $\nrep$ is
large, or if the time-step size is small if one
thinks of the Markov Chain as the time discretization of a
continuous time diffusion process (such as~\eqref{eq:Lang}). But in
practice, over many iterations and many independent runs, such
situations are actually observed and must be taken into account
carefully in the definition of the algorithm. In Section~\ref{sect:BrownDrift}, we investigate in a simple test case
the phenomenon described in this remark, and we
illustrate the importance of a proper implementation of the splitting
and resampling steps in such situations, in order to obtain unbiased estimators.
\end{Rem}

\subsection{The AMS estimator}\label{sec:AMS_estimator}
For any bounded observable $\ph: \calP \to \R$, a realization of the above
algorithm gives an estimator of the average $\pi(\ph)$ (see~\eqref{eq:phi}) defined by:
\begin{equation}\label{eq:phi_hat}
\hat{\ph}\eqdef \sum_{n \in I^{(Q_{\rm iter})}} G^{(n,Q_{\rm iter})}
\ph(X^{(n,Q _{\rm iter} )} ).
\end{equation}
One of the main aim of this paper is to prove that this estimator is
unbiased (see Theorem~\ref{th:unbiased} below):
\[\E(\hat{\ph})=\pi(\ph).\]
In order to highlight the main features
which make the estimator unbiased, we will actually prove this result for
a larger class of models and algorithms introduced
in Section~\ref{sect:GAMS}.

A particular choice of interest for some applications is
$\ph(x)=\mathds{1}_{\tT_B(x)<\tT_A(x)}$, in which case one obtains
an unbiased estimator of the probability $p=\PP(\tau_B<\tau_A)$. In this
case, due to the assumption~\eqref{eq:hyp_B} on $B$, only replicas
with labels in $I^{(Q_{\rm iter})}_{\rm on}$  contribute to the
  estimation, and thus, from one iteration to the other, only replicas
  with labels
  in $I^{(q)}_{\rm on}$ have to be retained, namely a system of
  $\nrep$ replicas. For this specific observable
  $\ph(x)=\mathds{1}_{\tT_B(x)<\tT_A(x)}$, the estimator of $p=\PP(\tau_B<\tau_A)$ is denoted
 by  $\hat{p}$ and is defined from~\eqref{eq:phi_hat} as:
\begin{align}
\hat{p}&=\sum_{n\in I^{(Q_{\rm iter})}_{\rm on}}G^{(n,Q_{\rm iter})}\mathds{1}_{\tT_B(X^{(n,Q_{\rm iter})})<\tT_A(X^{(n,Q_{\rm iter})})} \nonumber \\
&=\frac{\nrep-K^{(Q_{\rm iter})}}{\nrep}\ldots \frac{\nrep-K^{(1)}}{\nrep} P_{\rm corr} \label{eq:p_hat}
\end{align}
where the so-called ``corrector term'' is given by
\begin{equation}\label{eq:Pcorr}
P_{\rm corr} =\frac{1}{\nrep}\sum_{n\in I^{(Q_{\rm
        iter})}_{\rm on}}\mathds{1}_{\tT_B(X^{(n,Q_{\rm iter})})<\tT_A(X^{(n,Q_{\rm iter})})}
\end{equation}
namely the proportion of working replicas that have reached $B$
before $A$ at the final iteration. The properties of this estimator
will be numerically investigated in Section~\ref{sect:num}.

\section{Generalized Adaptative Multilevel Splitting}\label{sect:GAMS}

In this section, we introduce a general framework for adaptive multilevel
splitting algorithms, which contains in particular the AMS algorithm of
Section~\ref{sec:Markov}. We refer to this framework as the
Generalized Adaptive Multilevel Splitting (GAMS) framework. In particular, we prove  in Section~\ref{sec:justif_Markov} that
the AMS algorithm of Section~\ref{sec:Markov} fits in the GAMS
framework. The
interest of this abstract presentation is twofold. First, it highlights the essential
mathematical properties that are required to produce unbiased
estimators of quantities such as~\eqref{eq:phi}. 
 As will be
proven in Section~\ref{sec:unbiased}, any algorithm which enters into the GAMS
framework yields unbiased estimators of quantities such as~\eqref{eq:phi}. Second, it is very useful to
propose variants of the classical AMS algorithm which still
yield unbiased estimators: we propose some of them in Section~\ref{sec:variants}.

The section is organized as follows. In Section~\ref{sec:gen_set}, we
introduce in a general setting the quantities we are interested in
computing, and the main ingredients we need to state the GAMS framework. 
In Section~\ref{sec:GAMS_def}, the GAMS framework is presented.
In Section~\ref{sec:justif_Markov}, we prove that the AMS algorithm
introduced in Section~\ref{sec:Markov}  for the sampling of paths of
Markov chains enters into the GAMS framework. There is a strong
analogy between GAMS and Sequential Monte Carlo algorithms (the
branching step and the resampling step below correspond respectively
to the so-called selection and mutation steps): this is made precise
in Section~\ref{sec:SMC}.
Finally, we propose in Section~\ref{sec:variants} some variants of the
classical AMS algorithm to illustrate the flexibility of the GAMS
framework.

\subsection{The general setting}\label{sec:gen_set}

In this section, we introduce the ingredients and the main assumptions
we need in order to introduce the GAMS framework.
Throughout this section, the notations are consistent with those used in the context of the AMS algorithm. \comment{TONY: OK ?}


Let $(\Omega,\calF, \P)$ be a probability space. Let us introduce the state space
$\calP$, which is assumed to be a Polish space and let us denote $\calB(\calP)$ its Borel $\sigma$-field.
For example, in Section~\ref{sec:Markov}, the state space is the
path space of Markov chains (see~\eqref{eq:P_AMS}).
Let $X$ be a random variable with 
values in $(\calP,\calB(\calP))$ and probability distribution
\[
\pi  = \P \circ X^{-1} \in \Proba(\calP).
\]
The aim of the algorithms we present is to estimate
\begin{equation}\label{eq:pi(phi)}
\pi(\varphi)=\int_{\calP} \varphi(x) \pi(dx)
\end{equation}
for a given bounded measurable observable $\varphi: \calP \to \R$.


The two main ingredients we need in addition to $(\calP,\calB(\calP),\pi)$ is a filtration on
$\p{\calP,\calB(\calP)}$ (from which we build filtrations on
$\p{\calP^{\rm rep},\calB(\calP^{\rm rep})}$ and on $(\Omega,\calF,
\P)$) and some probability kernels $\pi_z(x, \cdot)$ from $\R \times
\calP$ to $\calP$. Let us introduce them in the next two sections.

\subsubsection{The filtrations}

In order to define the GAMS framework, we need an additional structure
on $\p{\calP,\calB(\calP)}$, namely a filtration indexed by real
numbers that we call in the following ``levels''. Therefore, we assume
in the following that the space $\p{\calP,\calB(\calP)}$ is endowed with a filtration indexed by levels $z\in \R$
\begin{equation}\label{eq:filtration}
\p{\filt_z}_{z \in \R} \subset \calB(\calP),
\end{equation}
namely a non-decreasing family of $\sigma$-fields: for any $z < z'$, $\filt_z \subset \filt_{z'} \subset \calB(\calP)$.
For example, in the context of Section~\ref{sec:Markov}, the filtration is defined
as follows: for any $z \in \R$, $\filt_z$ is  the smallest $\sigma$-field on
$\calP$ which makes the application $x \in \calP \mapsto (x_{t \wedge
  \tT_z(x)})_{t \ge 0} \in (\cal P,\calB(\calP))$ measurable.
%
%

From the filtration given by~\eqref{eq:filtration}, we
construct a filtration $\p{\filt_z^{\rm rep}}_{z\in \R}$ on the space
of replicas $\p{\calP^{\rm rep},\calB(\calP^{\rm rep})}$ (defined
by~\eqref{eq:def_Erep}) by considering the disjoint union of the
filtration $\filt_z$
on~$\calP$:
$$\filt_z^{\rm rep}=\bigsqcup_{ I \in {\mathcal I} }
\bigl(\filt_z\bigr)^{\card I}$$
where, we recall, $\mathcal{I}$ denotes the ensemble of finite subsets of
$\N^\ast$.



%
%





For any random variable $X:(\Omega,\calF)\rightarrow
\p{\calP,\calB(\calP)}$, we define a filtration
$\p{\filt_z^X}_{z\in\R}$ on the probability space by pulling-back the
filtration $\p{\filt_z}_{z\in\R}$:
\begin{equation}\label{eq:filtration_X}
\filt_z^X=X^{-1}\bigl(\filt_z\bigr).
\end{equation}

If $\calX = \p{X^{(n)}}_{n \in I} \in \calP^{\rm rep}$ denotes a
random system of replicas -- {\it i.e.} a random variable $\calX:(\Omega,\calF)\rightarrow \p{\calP^{\rm rep},\calB(\calP^{\rm rep})}$ -- we also define the filtration $\p{\filt_{z}^{\calX}}_{z\in \R}$ by the same pulling-back procedure:
$$\filt_z^\calX=\calX^{-1}\bigl(\filt^{\rm rep}_z\bigr).$$






By convention, we set
$$\filt^{\chi}_{-\infty}=\sigma(I) \quad \text{and} \quad \filt^{\chi}_{+\infty}=\calF,$$
where $\sigma(I)$ is the $\sigma$-field generated by the random set of labels~$I$. As a consequence  for any $z\in \R$ we
have $\filt_{-\infty}^{\calX} \subset\filt_{z}^{\calX} \subset \filt_{+\infty}^{\calX}$.


We finally introduce the notion of stopping level, which is simply a
reformulation of the notion of stopping time in our context where the
filtrations are indexed by levels instead of times.
  \begin{Def}[Stopping level, Stopped $\sigma$-field]\label{def:stopping_level}
    Let $\p{ \calF_z}_{z \in \R}$ be a filtration on $(\Omega,\mathcal{F},\P)$.
A stopping level $Z$ with respect to $\p{
      \calF_z}_{z \in \R}$ is a random variable with values in $\R$ such that $\set{ Z
      \leq z} \in \calF_z $ for any $z \in \R\cup \{-\infty,+\infty\}$.
 The stopped
    $\sigma$-field, denoted by $\calF_Z$, is characterized as follows:    
    $$A \in \calF_Z \quad  \text{ if and only if } \quad  \forall z \in \R, \, A \cap \set{ Z \leq z}   \in \calF_z .$$
    In particular, $Z$ is a $\calF_Z$-measurable random variable.
  \end{Def}

\begin{Rem}[On the definition of the filtrations]
In many cases of practical interest, for any~$z \in \R$, $\filt_z$ is defined as the smallest filtration which
makes an application $F_z: \calP \to (\calE,
\calB(\calE))$ measurable, for some Polish space $\calE$. Then, $\filt^{\rm rep}_z$ is the smallest filtration which
makes the application $G_z: \calP^{\rm rep} \to \calE^{\rm
  rep}$ measurable with $G_z( (X^{(n)})_{n \in I}) = 
(F_z(X^{(n)}))_{n \in I}$. For
example, in the setting of Section~\ref{sec:Markov}, $F_z: \calP \to \calP$ and $F_z(x)=(x_{t \wedge \tT_z(x)})_{t
  \ge 0}$.
\end{Rem}



\subsubsection{The resampling kernels $\pi_z(x, \cdot)$}\label{sec:resampling}

The second ingredient we need in addition to the filtrations
introduced above is a transition probability kernel from $\R \times
\calP$ to $\calP$ ($\R \times \calP$ being endowed with the Borel
$\sigma$-field $\calB(\R \times \calP)$):  $(z,x) \in \R \times
\calP
   \mapsto  \pi_z(x, \cdot) \in \Proba(\calP) $. 
By convention, for any $x\in \calP$ , we set
$\pi_{-\infty}(x,\cdot)=\pi$ (which is consistent with
Assumption~\ref{ass:key_2} below) and $\pi_{+\infty}(x,\cdot)=\delta_{x}$ . For an explicit example of a resampling kernel in the Markov chain example of Section~\ref{sec:Markov}, we refer to Section~\ref{sec:resampling_kernel}.

This kernel is  used in the resampling step as a family of transition
probabilities from $\calP$ to $\calP$, indexed by the level $z$. For a given level $z \in \R$ and a given state $x \in \cal P$,
$\pi_z(x, d x')$ is the probability distribution of the resampling of
the state $x$ from level $z$. In the following,
we will refer to this transition probability kernel as a resampling
kernel, since it is used in the resampling step.


%


\subsubsection{Assumptions on $(\filt^X_z)_{z \in \R}$ and $(\pi_z)_{z \in \R}$.}

We will need two assumptions on $(\filt^X_z)_{z \in \R}$ and
$(\pi_z)_{z \in \R}$.  The first assumption states a right continuity property of the mapping $\pi_z(\phi)(x)$ with respect to $z$ and is required to apply the Doob's optional stopping theorem in the proof of Lemma~\ref{lem:Doob}.
\begin{Ass}\label{ass:key_2}
  For any $x \in \calP$, and any continuous bounded test function
  $\ph:\calP \to \R$,
\[
\left\{
\begin{aligned}
&\R \to \R\\
&z \mapsto \pi_z(\ph)(x)
\end{aligned}
\right.
\]
is right-continuous. Moreover, $\lim_{z \to -\infty}
\pi_z(\ph)(x)=\pi_{-\infty}(\ph)(x)=\pi(\ph)$.
\end{Ass}
Recall the notation introduced in~\eqref{eq:pi_phi}: $\forall z \in
\R$, $\forall x \in
\calP$, $\pi_{z}(\ph)(x) =\int_{ y \in \calP} \varphi(y) \pi_z(x,dy)$.



Second, we require a consistency relation between the filtration $(\filt^X_z)_{z \in \R}$ and the transition probability kernel $(\pi_z)_{z \in \R}$. 
%
%
%
%
\begin{Ass}\label{ass:key_1}
Let us consider a random variable $X$, $(\filt_{z}^X)_{z\in\R}$ and $(\pi_z)_{z \in \R}$
as introduced above. We assume the following consistency relation:
 if $X$ is distributed according to 
$\pi_z(x,\cdot)$ for some $(z,x) \in \R \times \calP$, then for any $z' \ge z$ and for any bounded measurable test function $\ph: \calP \to \R$,
$$\E \p{ \ph(X) \st \filt_{z'}^X } = \pi_{z'}(\ph)(X)  \as$$
\end{Ass}
%
%
%
%
As a consequence (by letting $z \to -\infty$ in the previous assumption), if $X$ distributed according to $\pi$, then for any
$z' \in \R$ $\pi_{z'}(X,\cdot)$
is a version of the law of $X$ conditional on $\filt_{z'}^X$. Therefore, the $\sigma$-field
$\filt_{z'}^X$ can be interpreted as containing all the information on a
replica $X$ necessary to perform the resampling with $\pi_{z'}(X,\cdot)$ from $X$ at a given level $z' \in \R$.

Let us finally mention that in addition to these two assumptions and
from a more practical point of view, it is also implicitly assumed that it is
possible to sample
according to the probability measure $\pi$ (step (ii) of the
initialization step below) and according to the probability distribution
$\pi_z(x,\cdot)$, for any $x\in \calP$ and $z\in\R$ (step (ii) of the resampling
step below).

We will check in Section~\ref{sec:justif_Markov} that the Markov chain
example of Section~\ref{sec:Markov} enters into the general setting
introduced in this section.

We are now in position to introduce the GAMS framework in the following section.

\subsection{The Generalized Adaptive Multilevel Splitting framework}\label{sec:GAMS_def}

The aim of this section is to introduce a general framework  for splitting algorithms (which we
refer to as the Generalized Adaptive Multilevel Splitting (GAMS)
framework in the sequel).
%
%
%
The structure of the GAMS framework described in this section is quite similar to the one for the AMS algorithm of Section~\ref{sec:Markov}. One important difference is the introduction of a family of filtrations in the general setting. 
It iterates over three successive steps: (1) the
branching or splitting step, (2) the resampling step and (3) the level computation step. 
These steps are performed until a suitable stopping criterion is
satisfied.

We denote by $Q_{\rm iter}$ the number of iterations, which
in general is a random variable. At each iteration
step $q \geq 0$ of the algorithm the distribution
$\pi $ is approximated by an empirical distribution over a system
of weighted replicas $\calX^{(q)} := (X^{(n,q)},
  G^{(n,q)} )_{ n \in I^{(q)}} \in \calP^{\rm rep}$, where $I^{(q)}
  \subset \N^\ast$
  is the (random) finite set of labels at step $q$ of the algorithm
  and $G^{(n,q)} \in \R_+$ is the (random) weight attached to the replica $X^{(n,q)}$.

As it will become clear, in order to obtain a fully implementable algorithm from the GAMS
framework, three procedures need to be made precise (i) the stopping
criterion, (ii) the computation rule of the branching numbers and
(iii) the computation of the stopping levels.  These procedures
require to define three sets of random variables: $(S^{(q)})_{q \ge 0}$,
$(B^{(n,q+1)})_{q \ge 0,n \in
    I^{(q)}})$ and $(Z^{(q)})_{q \ge 0}$, that are used in the GAMS
  framework presented in the next section~\ref{sec:GAMS_precise_def}. The precise assumptions on these random variables
  will be stated in Section~\ref{sec:GAMS_to_algo} (see Assumption~\ref{ass:algo} below).
As already mentioned
above, the AMS algorithm of Section~\ref{sec:Markov} corresponds to 
specific choices of these three items, but the GAMS framework allows for
many variants (see Section~\ref{sec:variants}). The estimator associated with the GAMS framework is finally
defined in Section~\ref{sec:GAMS_estimator}.

\subsubsection{Precise definition of the GAMS framework}\label{sec:GAMS_precise_def}

We now introduce the Generalized Adaptive Multilevel Splitting (GAMS)
framework, which is an iterative procedure on an integer index $q \ge 0$.
\paragraph{The initialization step ($q=0$)}
\begin{enumerate}[(i)]
\item Define the initial set of labels $I^{(0)}=\{1,
  \ldots, \card I^{(0)} \} \subset \N^\ast$, where $\card I^{(0)}$ is assumed to be positive and finite.
\item Let $(X^{(n,0)})_{n \in I^{(0)}}$ be a sequence of $\calP$-valued i.i.d. random variables, and distributed according to the probability measure $\pi$. 
\item Initialize uniformly the weights: for any $n \in  I^{(0)}$ set
  $G^{(n,0)} \eqdef 1/ \card I^{(0)}$.
%

\item Define the system of weighted replicas $\calX^{(0)} = (G^{(n,0)}, X^{(n,0)})_{n \in I^{(0)}}$ and for any $z\in\R$, define the $\sigma$-field of events 
$
\calF^{(0)}_z \eqdef  \filt_z^{\calX^{(0)}}. 
$
\item Sample the initial level $Z^{(0)}$ (it is assumed to be a $(\calF^{(0)}_z )_{z \in \R }$ -- stopping level).
\item Define the $\sigma$-field of events 
$
\pcalF{0} = \calF^{(0)}_{Z^{(0)}}.
$
\end{enumerate}
\paragraph{Iteration}
Iterate on $q\geq 0$, while the stopping criterion is not satisfied.

\paragraph{The stopping criterion}
Sample the random variable $S^{(q)}\in\left\{0,1\right\}$ (which
is assumed to be $\pcalF{q}$-measurable). If $S^{(q)}=0$ then the algorithm stops and we set
$
Q_{\rm iter} = q.
$
Otherwise, if $S^{(q)}=1$, the three following steps are performed.

\paragraph{The splitting (branching) step}
\begin{enumerate}[(i)]


%
%
%

\item Conditionally on $\calF^{(q)}$, sample the $\N$-valued random branching numbers $\p{B^{(n,q+1)}}_{n\in I^{(q)}}$ which are assumed to satisfy: for any $n\in I^{(q)}$
\[\E \p{B^{(n,q+1)} \st \pcalF{q}} > 0 \as \] 
The random variable $B^{(n,q+1)}$ represents the number of offsprings of the replica $X^{(n,q)}$.
If $B^{(n,q+1)}\geq 1$, the replica $X^{(n,q)}$ will be split into
$B^{(n,q+1)}$ replicas: the old one (parent) $X^{(n,q)}$ with label $n\in
I^{(q)}$ and, if $B^{(n,q+1)}> 1$, 
$B^{(n,q+1)}-1$ new ones (children) that are defined in the resampling step below.
If $B^{(n,q+1)}=0$, the replica is removed from the system. Let us
thus introduce the set of labels of such replicas: $I^{(q+1)}_{\rm killed}=\left\{n\in I^{(q)} : B^{(n,q+1)}=0\right\}$.

\item Compute the total number of new replicas $K^{(q+1)} = \sum_{n \in I^{(q)}} \max\{B^{(n,q+1)}-1,0\} $.

\item Introduce the set $I^{(q+1)}_{\rm new} = \set{\max I^{(q)} +1,
    \ldots , \max I^{(q)}+K^{(q+1)}}  \subset \N^\ast \setminus I^{(q)}$ for new labels and update the total set of labels 
\begin{equation}\label{set_of_labels}
I^{(q+1)} \eqdef \Bigl(I^{(q)} \setminus I^{(q+1)}_{\rm killed}\Bigr)  \sqcup I^{(q+1)}_{\rm new}\Bigr..
\end{equation}

\item
Set a children-parent map $P^{(q+1)}:I_{\rm new}^{(q+1)}\rightarrow I^{(q)}\setminus I^{(q+1)}_{\rm killed}$ such that for any $n\in I^{(q)}\setminus I^{(q+1)}_{\rm killed}$ we have
\[
\card \left\{n'\in I_{\rm new}^{(q+1)} : P^{(q+1)}(n')=n \right\}=B^{(n,q+1)}-1.
\]
This map associates to the label of a new replica the label of its
parent. The parent replica (with label $n \in I^{(q)}\setminus
I^{(q+1)}_{\rm killed}$) is used in the resampling procedure to create
the new replica with label $n' \in I_{\rm new}^{(q+1)}$, where $n$ and
$n'$ are related through the children-parent map by $P^{(q+1)}(n')=n$.
 Notice that this map is determined up to a permutation of $I_{\rm new}^{(q+1)}$.
For notational convenience, we extend the map to $I^{(q+1)}$ as follows: $P^{(q+1)}(n)=n$ for any $n\in I^{(q)}\setminus I^{(q+1)}_{\rm killed}$.



\item Update the weights as follows: for all $n' \in  I^{(q+1)}$ and $ n
  \in I^{(q)}\setminus I^{(q+1)}_{\rm killed}$ such that $P^{(q+1)}(n')=n$,
  \begin{equation}\label{weights}
  G^{(n',q+1)} \eqdef \frac{G^{(n,q)}}{\E \p{B^{(n,q+1)} \st \calF^{(q)}}}.
  \end{equation}

\end{enumerate}

\paragraph{The resampling step}

\begin{enumerate}[(i)]
\item Replicas in $I^{(q)}\setminus I^{(q+1)}_{\rm killed}$ are not
  resampled {\it i.e.} for any $n\in I^{(q)}\setminus I^{(q+1)}_{\rm killed}$, $X^{(n,q+1)}= X^{(n,q)}$.
\item For $n'\in I_{\rm new}^{(q+1)}$, $X^{(n',q+1)}$ is sampled by branching its parent replica $X^{(P^{(q+1)}(n'),q)}$, {\it i.e.}
  according to
   the resampling kernel $\pi_{Z^{(q)}}(X^{(P^{(q+1)}(n'),q)},dx)$.
\end{enumerate}

Then set $\calX^{(q+1)} = (X^{(n,q+1)},G^{(n,q+1)})_{n \in I^{(q+1)}}$.

\paragraph{The level computation step}
\begin{enumerate}[(i)]
\item For any $z\in\R$, define the $\sigma$-field of events
\begin{equation}\label{def_filt_calF}
\calF^{(q+1)}_z \eqdef \calF^{(q)} \vee \sigma( P^{(q+1)} ) \vee \filt_z^{\calX^{(q+1)}}.
\end{equation}
The $\sigma$-field  generated by $P^{(q+1)}$ contains in particular the
$\sigma$-field generated by $(B^{(n,q+1)})_{n
  \in I^{(q)}}$.
\item Sample the next level $Z^{(q+1)} \in \R $, which is assumed
  to satisfy:
\begin{itemize}
\item $ Z^{(q+1)} \geq Z^{(q)} \as $
\item $Z^{(q+1)}$ is a stopping level with respect to $\p{ \calF^{(q+1)}_z }_{z \in \R }$.
\end{itemize}
\item Define the $\sigma$-field of events $\calF^{(q+1)} \eqdef \calF^{(q+1)}_{Z^{(q+1)}}$.
\end{enumerate}

\paragraph{Increment}
Increment $q\leftarrow q+1$ and go back to the stopping criterion.

\medskip

For theoretical purposes, we need in the following to define the
system of weighted replicas $(\calX^{(q)})_{q \ge 0}$ and the associated
filtration $\p{\calF^{(q)}}_{q\ge 0}$ for all $q \ge 0$ (and not only up to
the iteration~$Q_{\rm iter}$). This is simply done by considering the
iterative procedure above with $S^{(q)}=0$ for all $q \ge 0$.

\begin{Rem}[On the labeling]
The way the replicas are labeled is purely conventional.
\end{Rem}

\subsubsection{From the GAMS framework to a practical
  algorithm}\label{sec:GAMS_to_algo}

In the GAMS framework, we have defined (see~\eqref{def_filt_calF}) a family of $\sigma$-fields which is indexed both by the
level $z \in \R$ and by the iteration index $q\geq 0$ and which is
denoted by $\bigl( \calF^{(q)}_z\bigr)_{q\geq 0, z \in \R}$.
By construction, this family of $\sigma$-fields satisfies
$\calF^{(q)}_z \subset \calF^{(q')}_{z'}$ if $q < q'$ or  $\left\{q=q'
  \text{ and } z \leq z'\right\}$: in other words, the family $\p{
  \calF^{(q)}_z }_{q\geq 0, z \in \R}$ is a filtration if 
$\N\times \R$ is endowed with the lexicographic ordering. 
At the end of the $q$-th iteration of the algorithm ($q\geq 0$), one can think
of the $\sigma$-field $\calF^{(q+1)}=\calF^{(q+1)}_{Z^{(q+1)}}$ as containing all the
necessary information required to perform the next step of the
algorithm. 

To make a practical splitting algorithm which enters into the GAMS
framework, three sets of random variables need to be
defined: $(S^{(q)})_{q \ge 0}$,
$(B^{(n,q+1)})_{q \ge 0,n \in
    I^{(q)}})$ and $(Z^{(q)})_{q \ge 0}$. As already stated above, we
  assume the following on these random variables.
\begin{Ass}\label{ass:algo}
The random variables $(S^{(q)})_{q \ge 0}$,
$(B^{(n,q+1)})_{q \ge 0,n \in
    I^{(q)}}$, and $(Z^{(q)})_{q \ge 0})$ satisfy the
following properties:
\begin{itemize}
\item the sequence of random variables $(S^{(q)})_{q \ge 0}$ needed for defining the stopping
  criterion, are such that $S^{(q)}$ is with values in $\{0,1\}$ and
  is $\pcalF{q}$-measurable;
\item the sequence of branching numbers $(B^{(n,q+1)})_{q \ge 0,n \in
    I^{(q)}}$ are with values in $\N$, are assumed to be sampled conditionally on
  $\pcalF{q}$ (see Section~\ref{sec:notation} for a precise
  definition) and such that $\E \p{B^{(n,q+1)} \st \pcalF{q}} > 0 \as$;
\item the
  sequence $(Z^{(q)})_{q \ge 0}$ of stopping levels are with
  values in $\R$, satisfy
  $Z^{(q+1)} \ge Z^{(q)}$ and are such that $Z^{(q)}$ is a stopping level with
  respect to $\p{ \calF^{(q)}_z }_{z \in \R }$  (see Definition~\ref{def:stopping_level}).
\end{itemize}
\end{Ass}

As explained above, once these three sets of random variables have
been defined, the GAMS framework becomes a practical splitting
algorithm which yields an unbiased estimator of~\eqref{eq:pi(phi)}
(this is the claim of Theorem~\ref{th:unbiased} proved in Section~\ref{sec:unbiased}).



Let us emphasize that the requirement that $Z^{(q)}$ is a
$(\calF_{z}^{(q)} )_{z\in\R}$-stopping level is fundamental to obtain
unbiased estimators. It will be instrumental to apply Doob's optimal
stopping Theorem for appropriate martingales in the proof of unbiasedness.

As a consequence of the measurability property of $(S^{(q)})_{q \ge
  0}$, one easily gets the following property on $Q_{\rm iter}$:
\begin{Pro}\label{propo_Qiter_stopping}
The random variable $Q_{\rm iter}$ is a stopping time with respect to the filtration $\p{\calF^{(q)}}_{q\ge 0}$.
\end{Pro}

\begin{Rem}[On the measurability of the system of replicas with
  respect to $(\calF^{(q)}_z)_{z \in \R}$]
Let us emphasize that for any $q \ge 0$, the system of replicas
$(X^{(n,q)})_{n \in I^{(q)}}$ is $\calF$-measurable but it is not measurable with respect to
$\calF^{(q)}=\calF^{(q)}_{Z^{(q)}}$ (which indeed stores the
information only up to the stopping level $Z^{(q)}$). 
%
%
\end{Rem}

\subsubsection{The estimator}\label{sec:GAMS_estimator}
 
For any integer $q \geq 0$ and any bounded test function $\ph: \calP \to \R$, we define the estimator
\begin{equation}\label{GAMS_estimator}
 \dps \hat{\pi}^{(q)}(\ph)  = \sum_{n \in I^{(q)}} G^{(n,q)} \ph(X^{(n,q)})
\end{equation}
of $\pi(\ph)$.
As it will be proven in Section~\ref{sec:unbiased}, any algorithm which
enters into the GAMS framework is such that $\hat{\pi}^{(q)}(\ph)$ is an unbiased estimator of $\pi(\ph)$: for any $q\geq 0$, $\E\p{\hat{\pi}^{(q)}(\ph)}=\pi(\ph)$.
Moreover, under appropriate assumptions (see Theorem~\ref{th:unbiased}), this statement can be generalized when $q$ is replaced by the random number of iterations $Q_{\rm iter}$ of the algorithm:
\[
\E\p{ \sum_{n \in I^{(Q_{\rm iter})}} G^{(n,Q_{\rm iter})} \ph(X^{(n,Q_{\rm iter})}) } = \pi(\ph).
\]
The proof of this result is given in Sections~\ref{sec:proof_th} and~\ref{sec:proofs} and is
based on martingale arguments.


%

\subsection{The AMS algorithm enters into the GAMS framework}\label{sec:justif_Markov}

In this section, we explain how the GAMS framework encompasses the AMS algorithm of Section~\ref{sec:Markov}. We thus go back to the setting described there and prove that the modelling and algorithmic assumptions of sections~\ref{sec:gen_set} and~\ref{sec:GAMS_def} are satisfied in this case. 

\subsubsection{Modelling assumptions}

Let us first check that the so-called dynamical setting (namely the
sampling of paths of Markov chains) that we considered in
Section~\ref{sec:Markov} for the AMS algorithm enters into the general setting of Section~\ref{sec:gen_set}.

In Section~\ref{sec:Markov}, $(\calP,\calB(\calP),\pi)$ is the path space for Markov
chains, endowed with the standard topology, as explained in
Section~\ref{sec:Markov_path}.  The filtration $(\filt_z)_{z \in \R}$ on $(\calP,\calB(\calP))$ is
defined by: for all $z \in \R$, $\filt_z$ is the smallest $\sigma$-field which makes the application $x \in \calP \mapsto x_{t \wedge (\tT_z(x))} \in \cal P$ measurable:
\begin{equation}\label{eq:filtz_Markov}
\filt_z=\sigma\big(x \mapsto (x_{t \wedge (\tT_z(x))})_{t \ge 0}\big).
\end{equation}
Finally, for any $z\in \R$ and $x\in \calP$, the resampling kernel $\pi_z(x,\cdot)$ is defined by~\eqref{eq:piz_chain_1}--\eqref{eq:piz_chain}.

Let us now check that Assumptions~\ref{ass:key_2} and~\ref{ass:key_1} are satisfied.
The conditions of Assumption~\ref{ass:key_1} are direct consequences of the strong Markov property applied to the chain $t \mapsto X_t \in \calS$ defined by \eqref{eq:X_Path} at the stopping time $\tau_z$  (the strong Markov property always holds true for discrete-time Markov processes).

The right-continuity property of Assumption~\ref{ass:key_2} crucially relies on the definition~\eqref{eq:Tz} of $\tT_z(x)$ as the entrance time of the path $t\mapsto x_t$ in the level set $\xi^{-1}\bigl]z,+\infty\bigr[$: the fact that $]z,\infty[$ is an open set implies $z\mapsto \tT_z(x)$ is right continuous. More precisely, we have the following Lemma.


\begin{Lem} Assumption~\ref{ass:key_2} is satisfied for the resampling
  kernel defined by~\eqref{eq:piz_chain_1}-\eqref{eq:piz_chain}. More precisely, for any $x \in \calP$, the resampling kernel $z \in \R \mapsto \pi_z(x, \, . \, ) \in \Proba(\calP)$ is piecewise constant and right continuous.
\end{Lem}
\begin{proof}
First, assume that $\tT_z(x)=+\infty$, which means that $\Xi(x)\leq
z$. Then, for any $\eps \geq 0$ we still have
$\tT_{z+\eps}(x)=+\infty$. In that case $\pi_z(x, \, . \,)$ is a Dirac mass:
$\pi_z(x, \, . \,) = \pi_{z + \eps}(x, \, . \, )=\delta_{(x_{t\wedge \tT_A(x)})_{t\ge 0}}$.

Now, assume that $\tT_z(x)<+\infty$. Then, for $\eps\in ]0,\xi(x_{\tT_z(x)})-z[$, $\tT_z(x)=\tT_{z+\eps}(x)$, and by the definition of the resampling kernel,
$\pi_z(x, \, . \,) = \pi_{z + \eps}(x, \, . \, )$.
\end{proof}


\subsubsection{Algorithmic assumptions}

As explained in Section~\ref{sec:GAMS_def}, to obtain a practical
splitting algorithm which enters into the GAMS framework, three procedures
need to be made precise: the stopping criterion, the computation rule
of the branching numbers and the computation of the stopping
levels. These procedures should satisfy the measurability requirements
of Assumption~\ref{ass:algo}.


\paragraph{The stopping criterion}

In the AMS algorithm, we set $S^{(q)}=\mathds{1}_{Z^{(q)}>z_{\rm max}}$
which is indeed a $\calF^{(q)}$-measurable random variable, since
$Z^{(q)}$ is a $(\calF^{(q)}_z)_{z \in \R}$-stopping level, see
Lemma~\ref{lem:stop} below.

\paragraph{The computation rule
of the branching numbers}

The branching numbers $B^{(n,q+1)}$ are defined in the splitting step (iv) of the
AMS algorithm by \eqref{eq:BN_AMS}, for $n\in I_{{\rm on},>
  Z^{(q)}}^{(q)}$. We extend the definition for $n\in I^{(q)}\setminus
I_{{\rm on},> Z^{(q)}}^{(q)}$ by simply setting $B^{(n,q+1)}=1$. It is then
easy to check that they satisfy the requirements of
Assumption~\ref{ass:algo}.  Notice that in the AMS algorithm, the total number of new replicas
$K^{(q+1)}=\sum_{n \in I^{(q)}} \max\{B^{(n,q+1)}-1,0\}$ is given by
$K^{(q+1)}=\card  I_{{\rm on}, \le Z^{(q)}}^{(q)}$.
Moreover, all branching numbers are positive, so that
$I^{(q+1)}_{\rm killed}=\emptyset$. Another particular feature of the AMS algorithm is that the map $P^{(q+1)}$ takes
values in the strict subset $I_{{\rm on}, >Z^{(q)}}^{(q)}$ of 
$I^{(q)}$.



Let us check that the computation rule~\eqref{eq:weight_AMS} for the
weights in the AMS algorithm is indeed consistent with the
formula~\eqref{weights} given in the GAMS framework. First, for $n \in
I^{(q+1)}_{\rm off}=I_{{\rm on}, \leq Z^{(q)}}^{(q)}\sqcup I_{\rm
  off}^{(q)}$, $B^{(n,q+1)}=1$, $P^{(q+1)}(n)=n$ and, consistently, $G^{(n,q+1)}=G^{(n,q)}$.

Second, for $n\in I_{{\rm on}, >Z^{(q)}}^{(q)}$, it is clear that $\E
\p{B^{(n,q+1)} \st \calF^{(q)} }$ does not depend on $n$
(since the random variables are exchangeable in $n \in I_{{\rm on},
  >Z^{(q)}}^{(q)}$). In addition, by construction, $\sum_{n' \in
  I_{{\rm on},> Z^{(q)}}^{(q)}} B^{(n',q+1)}=\nrep$. Thus, we
have by a simple counting argument:  for any $n\in I_{{\rm on},> Z^{(q)}}^{(q)}$,
\begin{align*}
\E\bigl(B^{(n,q+1)}\st \calF^{(q)}\bigr)&=\frac{1}{\card I_{{\rm on},> Z^{(q)}}^{(q)}}\sum_{n' \in I_{{\rm on},> Z^{(q)}}^{(q)}} \E\bigl(B^{(n',q+1)}\st \calF^{(q)}\bigr)\\
&=\frac{ \E\left( \sum_{n' \in I_{{\rm on},> Z^{(q)}}^{(q)}}
  B^{(n',q+1)}\st \calF^{(q)}\right)}{\card I_{{\rm on},> Z^{(q)}}^{(q)}}=\frac{\nrep}{ \nrep - K^{(q+1)}}.
\end{align*}
Thus for $n \in   I_{ {\rm on}, > Z^{(q)}}^{(q)}$  (and since $P^{(q+1)}(n)=n$) the formula
$G^{(n,q+1)} \eqdef  \frac{ \nrep - K^{(q+1)} }{ \nrep }G^{(n,q)}$
in~\eqref{eq:weight_AMS} for the AMS algorithm is indeed consistent with the
updating formula~\eqref{weights} for the weights in the GAMS framework.

Third, for $n\in I_{\rm new}^{(q+1)}$, $G^{(n,q+1)} = G^{(P^{(q+1)}(n),q+1)}= \frac{ \nrep - K^{(q+1)} }{ \nrep }G^{(P^{(q+1)}(n),q)}$ which is again consistent with the
updating formula~\eqref{weights} for the weights in the GAMS
framework since $\frac{ \nrep - K^{(q+1)} }{ \nrep }=1 / \E\bigl(B^{(P^{(q+1)}(n),q+1)}\st \calF^{(q)}\bigr)$.



\paragraph{Computation of the stopping levels}

Let us now
check that the requirements on $Z^{(q)}$ in Assumption~\ref{ass:algo}
are satisfied. By
definition of $Z^{(q+1)}$ (see the level computation step of the AMS
algorithm), it is clear that $Z^{(q+1)} \ge
Z^{(q)}$ (actually, the strict inequality $Z^{(q+1)} >
Z^{(q)}$ holds). It remains to prove that $Z^{(q)}$ is a stopping level for the filtration $(\calF^{(q)}_z)_{z \in \R}$.

We start with an elementary result, which again highlights the
importance of the strict inequality $>z$ in the
definitions~\eqref{eq:Tz} and~\eqref{eq:tauz} of $\tT_z(x)$ and $\tau_z$.

\begin{Lem}\label{lem:stop} Let $X: \Omega \to \calP$ be a Markov
  chain over the state space $\calS$ (see
  Equation~\eqref{eq:P_AMS}). Then the random variable $\Xi(X)$
  (where, we recall, the maximum level mapping $\Xi$ is defined by~\eqref{eq:max}) is a $(\filt^X_z)_{z \in \R }$-stopping level: for any $z\in \R$,
$\left\{ \Xi(X) \leq z\right\}\in \filt^X_z.$
\end{Lem}

\begin{proof}
On the one hand, we clearly have the equality of subsets of $\calP$:
$$\left\{ x\in \calP : \Xi(x)\leq z\right\}=\left\{x\in\calP : \tT_z(x)=+\infty\right\}.$$
On the other hand, $\tau_z=\tT_z(X)$ is a $\filt_{z}^{X}$-measurable
random variable. The result is then a consequence of these two facts.
\end{proof}

We are now in position to prove the last results which is needed for
Assumption~\ref{ass:algo} to hold.

\begin{Lem}\label{lem:Z^q_stop_level}
For any $q\geq 0$, $Z^{(q)}$ is a stopping level with respect to the
filtration $(\calF^{(q)}_z)_{z \in \R}$: for any $z \in \R$, $\{Z^{(q)} \le
z\} \in \calF^{(q)}_z$.
\end{Lem}

\begin{proof}
\comment{TONY: Vérifier cette preuve...}
Set by convention $Z^{(-1)}=-\infty$ and let us consider $q \ge 0$. Let us introduce the $k$-th
order statistics over the maximum levels at iteration $q$:
$L^{(q+1)}=\Xi(X^{(\Sigma^{(q+1)}(k),q+1)})$. Let us also introduce
$M^{(q+1)}=\max \{\Xi(X^{(n,q+1)}) : n \in I^{(q+1)}_{\rm on}\}$. By
definition of $Z^{(q+1)}$ (see the level computation step of the AMS
algorithm), $$Z^{(q+1)}=L^{(q+1)} \mathds{1}_{\{L^{(q+1)} < M^{(q+1)}\}}+
(+\infty) \mathds{1}_{\{L^{(q+1)} = M^{(q+1)}\}}.$$
Therefore, for any $z \in \R$, (using the 
partition $\Omega=\{M^{(q+1)} \le z\}\sqcup  \{M^{(q+1)} > z\}$)
\begin{align*}
\{Z^{(q+1)}\leq z\}&=\{L^{(q+1)} \le z\} \cap \{L^{(q+1)} <
                     M^{(q+1)}\}\\
&=\{L^{(q+1)} < M^{(q+1)} \le z\} \sqcup \{L^{(q+1)} \le z <
                     M^{(q+1)}\}.
\end{align*}
These events are all in the $\sigma$-field $\sigma\p{\set{\Xi(X^{(n,q+1)}) \leq z},
  \set{\Xi(X^{(n,q+1)}) \leq Z^{(q)}}, n \in I^{(q+1)}}$ (in
particular, the set of labels $I^{(q+1)}_{\rm on}$ is measurable
with respect to $\set{\Xi(X^{(n,q+1)}) \leq Z^{(q)}}$). To conclude, note that by construction (level computation step, $(i)$) and thanks to Lemma~\ref{lem:stop}: for any $z\in\R$,
$$\sigma\p{\set{\Xi(X^{(n,q+1)}) \leq z}, \set{\Xi(X^{(n,q+1)}) \leq Z^{(q)}}, n \in I^{(q+1)}}\subset \filt^{\calX^{(q+1)}}_{ z \vee Z^{(q)}}\subset \calF^{(q+1)}_z.$$
\end{proof}
%
%
%


\subsubsection{Almost sure mass conservation}

The classical AMS algorithm satisfies an additional
nice property, namely it conserves almost surely the mass in the following sense:
\begin{Def}
A splitting algorithm which enters into the GAMS framework satisfies
the almost sure mass conservation property if 
\begin{equation}
  \label{eq:mass_cons}
\forall q \geq 0, \,   \sum_{n \in I^{(q)}}  G^{(n,q)} = 1 \as
\end{equation}
\end{Def}

Indeed, using the definition~\eqref{eq:weight_AMS} of the weights and
in particular the fact that all the weights $(G^{(n,q)})_{n \in
  I^{(q)}_{{\rm on},Z^{(q)}}}$ are the same: for
any $q \ge 0$,
\begin{align*}
 \sum_{n' \in I^{(q+1)}}  G^{(n',q+1)}&=\sum_{n' \in I_{\rm off}^{(q+1)}}G^{(n',q+1)}+\sum_{n' \in I_{{\rm on}}^{(q+1)}}G^{(n',q+1)}\\
&= \sum_{n' \in I_{{\rm on}, \leq Z^{(q)}}^{(q)}\sqcup I_{\rm off}^{(q)}}G^{(n',q)}+\sum_{n \in I_{{\rm on}, > Z^{(q)}}^{(q)}\sqcup I_{\rm new}^{(q+1)}}\frac{ \nrep - K^{(q+1)} }{ \nrep }G^{(P^{(q+1)}(n),q)}\\
&= \sum_{n' \in I_{{\rm on}, \leq Z^{(q)}}^{(q)}\sqcup I_{\rm off}^{(q)}}G^{(n',q)}+\sum_{n \in I_{{\rm on}, > Z^{(q)}}^{(q)}}G^{(n,q)}=\sum_{n' \in I^{(q)}}  G^{(n',q)}.
\end{align*}
Thus, since $\sum_{n' \in I^{(q)}}  G^{(n',0)}=1$, by induction on
$q$, \eqref{eq:mass_cons} is satisfied. This property will be useful in 
Theorem~\ref{th:unbiased} below: it is one of the two sufficient conditions to prove the
unbiasedness of the estimator $\hat{\ph}=\hat{\pi}^{(Q_{\rm iter})}(\ph)$ ($\hat{\pi}^{(q)}(\ph)$
being defined, we recall, by~\eqref{GAMS_estimator}).

Notice that this property
is not generally satisfied for any algorithm which enters into the
GAMS framework. Actually, it is only true in general on average: by taking
$\ph(x)=1$ and $Q_{\rm iter}=q$ in Theorem \ref{th:unbiased} below,
one indeed obtains that $\forall q \geq 0$,   $\E \left(\sum_{n' \in I^{(q)}}  G^{(n',q)} \right)= 1$.

\subsection{Reformulation of the AMS algorithm as a Sequential Monte-Carlo method}\label{sec:SMC}

The aim of this section is to make more explicit the link between the
 AMS algorithm and a Sequential Monte Carlo (SMC) sampler, for
readers who are familiar with SMC methods. For those who are not, this
section can be easily skipped.

For a reaction coordinate with discrete values, the AMS algorithm presented in Section~\ref{sec:Markov} can be
understood as a sequential importance sampling algorithm, where
weights are assigned to replicas, and replicas are then duplicated and
killed to compensate for these weights and obtain unbiased
estimators (see for example~\cite{DoucetDe-FreitasGordon2001} for a
nice introduction to SMC methods
and~\cite{CerouDel-MoralFuronGuyader2012} for a discussion of the
relationship between SMC algorithms and
multilevel splitting algorithms).

To highlight the similarity between the AMS algorithm and a SMC sampler, let us assume that the reaction coordinate takes values in the finite set $\{1,2, \ldots, z_{\rm max}\}$
\[
{\xi: \calS \to \{1,2, \ldots,z_{\rm max}\}}.
\]
Let us now introduce a new way to label the successive iterations of
the algorithm, by using the
levels $z \in \{1,2, \ldots, z_{\rm max}\}$ rather than the iteration
index $q \geq 0$. Notice indeed that for each $z$, there exists a unique
iteration index $q \ge 0$ such that $z \in [Z^{(q-1)}, Z^{(q)})$. Let
us then set: for all $z \in \{1,2, \ldots, z_{\rm max}\}$ and $q$ such that $z \in [Z^{(q-1)}, Z^{(q)})$,
\begin{equation*}
  \begin{cases}
  J^{(z)} = I^{(q)},\\
    Y^{(n,z)} = (X^{(n,q)}_{t \wedge \tT_z(X^{(n,q)})}, t \in \N),
    \quad \forall n \in I^{(q)}, \\
    H^{(n,z)} = G^{(n,q)}, \quad \forall n \in I^{(q)}.  
  \end{cases}
\end{equation*}
The random variables $J^{(z)}$,  $Y^{(n,z)}$ and $H^{(n,z)}$ are thus
respectively the new set of labels, the new system of replicas and
the new system of weights, indexed by the levels $z$ rather than the
iteration index $q$.
One can then check that the sequence of weighted replicas
$(Y^{(n,z)},H^{(n,z)})_{n \in J^{(z)}}$ is obtained by applying a standard
sequential Monte Carlo algorithm which iterates the following two steps:
\begin{enumerate}
\item A splitting step, equivalent to the splitting step of Section~\ref{sec:AMS_Markov}: replicas that have reached the $z$-level set are split and weighted according to the splitting rule which conserves the total number of replicas above $z$. The weights of replicas that have not reached the $z$-level set are not modified.
\item A mutation step, where all replicas are resampled independently
  according to $\pi_z$, but with paths stopped at the stopping time
  $\tT_{z+1}$ (defined by~\eqref{eq:Tz}).
\end{enumerate}

In the SMC algorithm presented above, all the replicas are resampled
which is not the case for the classical AMS algorithm. The following
lemma is then crucial to reformulate the AMS algorithm as a SMC
sampler. We recall that the resampling kernel $\pi_z(x,dx')$ has been defined in
Section~\ref{sec:resampling_kernel}. Moreover, the children-parent
mapping has been extended to $I^{(q+1)}_{\rm on}$
by setting $P^{(q+1)}(n)=n$ for $n \in I^{(q+1)}_{\rm on} \setminus
I^{(q+1)}_{\rm new}$.
\begin{Lem}
Consider the algorithm AMS introduced in
Section~\ref{sec:AMS_Markov}. Assume that in the resampling step,
\emph{all} replicas are resampled. More precisely, replace the two
items $(i)$ and $(ii)$ in the resampling step by a single one:
\begin{enumerate}[(i)]
\item For all $n\in I^{(q+1)}_{\rm on}$, $X^{(n,q+1)}$ is sampled with the resampling kernel $\pi_{Z^{(q)}}(X^{(P^{(q+1)}(n),q)}, d x')$.
\end{enumerate}

Then, the probability distribution of the algorithm is unchanged:
the random variables $(X^{(n,Q_{\rm iter})},G^{(n,Q_{\rm iter})})_{n
  \in I^{(Q_{\rm iter})}}$ have the same law for the modified
algorithm as for the original one.
\end{Lem}
This lemma is easily checked using Proposition~\ref{pro:1}, $(ii)_q$ and an induction argument on $q \geq 0$.


The discussion above thus shows that the AMS algorithm can be recast in the
framework of sequential sampling. The interpretation of multilevel splitting methods as a sequential
sampling method is not new (see {\it
  e.g.}~\cite{JohansenDel-MoralDoucet2005}). We refer to the classical
monographs~\cite{DoucetDe-FreitasGordon2001,Del-Moral2004} for
respectively applications of  Sequential Monte-Carlo methods in
Bayesian statistics, and a comprehensive associated
mathematical analysis. In particular, from the point view
of~\cite{Del-Moral2004}, the Adaptive Multilevel Splitting method for
Markov chains (namely the dynamical setting) considered here can be interpreted as
a time-dependent Feynman-Kac particle model with hard obstacles
where: (i) the time index is given by the discrete levels $z$, (ii)
the particles are paths of the Markov chain stopped at
$\tT_{z}$, and (iii) the hard obstacle at level $z$ corresponds with
reaching $A$ before the $z$-level set. Note however that strictly
speaking, the version presented in the present section slightly differs
from the classical presentation of Feynman-Kac particle
models in~\cite{Del-Moral2004} since  all the replicas are used in
the estimators, including those who have reached $A$. But this does
not change the global picture.

To conclude, let us recall that the construction of unbiased
estimators for averages of the form~\eqref{eq:phi} is standard for SMC
algorithms. In the SMC language, averages such as~\eqref{eq:phi} are
called non-normalized averages, normalized averages being 
conditional expectations, namely ratios of two such averages.
From this point of view, the 
unbiasedness result of the present work (see
Theorem~\ref{th:unbiased}) is therefore not a surprise. Actually, another
strategy of proof of Theorem~\ref{th:unbiased} would be to rely on
general unbiasedness results for SMC samplers, using the equivalence
between AMS and SMC described above for discrete reaction coordinates,
and then to extend the result to continuous reaction coordinates by
considering a continuous limit of discrete levels. 



\subsection{Examples of algorithmic variants}\label{sec:variants}

In this section, we consider the setting and the AMS
algorithm of Section~\ref{sec:Markov}, and we propose variants which
fit into the Generalized Adaptive Multilevel Splitting framework and may improve the efficiency
of the algorithm in several directions (see
Sections~\ref{sec:remove},~\ref{sec:random} and~\ref{sec:select}). In particular,
Theorem~\ref{th:unbiased} applies to the three examples detailed
below. Moreover, we also illustrate the interest of the general
setting we have introduced by providing in Section~\ref{sec:bridge}  an example 
which does not enter into the standard dynamical setting (sampling of paths of
Markov chains) and for which the AMS algorithm could be used.

\subsubsection{Removing extinction}\label{sec:remove}

We first introduce a variant of the AMS algorithm in
the Markov chain setting (Section~\ref{sec:Markov}), which is designed
in order to remove the possibility of equality of levels for two
different replicas -- this phenomenon is explained in
Remark~\ref{rem:same_max} for the AMS algorithm. This variant enters
into the GAMS framework and thus leads to unbiased estimators. With
this variant, exactly $k$ replicas are resampled at each iterations. In particular, extinction of the
system of replicas cannot occur. However, the algorithm requires the
use of a rejection procedure for each resampling, which may slow down
the simulation. Let us now describe this variant in detail.

Let $z\in\R$ be a level and $x\in\calP$. The definition
(see~\eqref{eq:Tz}) of $\tT_z(x)$ remains the same, but the resampling
kernel $\pi_z$ defined by~\eqref{eq:piz_chain_1}--\eqref{eq:piz_chain}
is modified as follows. Given $x\in\calP$, the probability law $\pi_z(x,dx')$ is the distribution of a random variable $Y \in \calP$ sampled as follows:
\begin{itemize}
\item For $t\leq \tT_z(x)-1$, $Y_{t}=x_{t}$.
\item When $\tT_z(x)<+\infty$, $Y_{\tT_z(x)}$ is sampled using the
  transition kernel $P(x_{\tT_z(x)-1} , \, . \, )$ of the Markov
  chain, conditionally on $\xi(Y_{\tT_z(x)})>z$. This can be done for example using a rejection procedure: a sequence
$(\mathcal{Y}_\ell)_{\ell\in\N^*}$ of i.i.d. random variables
distributed according to $P(x_{\tT_z(x)-1}, \, . \, )$ is sampled, and
one considers $Y_{\tT_z(x)}=\mathcal{Y}_{L}$ where
$L=\inf\left\{\ell\in\N^* : \xi(\mathcal{Y}_{\ell})>z\right\}$. 


\item For $t > \tT_z(x)$, the Markov transition kernel $P$ is used to
  sample the end of the trajectory, up to the stopping time $\tT_A(Y)$
  where the path is stopped: $$\Law( Y_t \st   (Y_s)_{0 \leq s \leq t-1} ) = P( Y_{t-1}, \, . \, ).$$
\end{itemize}

The definition of the filtration $(\filt_z)_{z\in\R}$ needs to be adapted in order to check Assumption~\ref{ass:key_1}.
The filtration $\filt_z$ is the smallest $\sigma$-field which makes the application
$x \in \calP \mapsto (\tT_z(x), (x_{t \wedge (\tT_z(x)-1)})_{t \ge 0})
\in  \N \times \cal P$ measurable:
\begin{equation}\label{eq:filtz_Markov_bis}
\filt_z=\sigma(x \mapsto (\tT_z(x),(x_{t \wedge (\tT_z(x)-1)})_{t \ge 0})).
\end{equation}
Notice that we need $\tT_z(x)$ in addition to $(x_{t \wedge
  (\tT_z(x)-1)})_{t \ge 0}$ since we need to know the time at which the
chain reaches the level $z$.
%
%
%
%
\comment{Tony: A checker...}
In order to check Assumption~\ref{ass:key_1}, let us introduce the auxiliary Markov chain with values in $\calS\times\left\{0,1\right\}$:
 $\tilde{X}_t=(X_{t},\mathds{1}_{\xi(X_{t+1})>z})$. Then
 Assumption~\ref{ass:key_1} follows from the strong Markov property
 applied to $\tilde{X}_t$ and the family of stopping times indexed by $z$ defined by
 $\tilde{\tau}_z=\inf\left\{t \geq 0 :  \mathds{1}_{\xi(X_{t+1})>z}
   =1\right\}$. Indeed, $\calF_{\tilde{\tau}_z}=\filt^X_z$, and $\E(\varphi(X)|\calF_{\tilde{\tau}_z})=\pi_z(\varphi)(X)$.



With this modification of the classical AMS algorithm of
Section~\ref{sec:Markov}, it is easy to check that the event that two
replicas have the same maximum level is of probability zero, at least if
the natural additional property is satisfied: if $Y_1$ and $Y_2$ are
generated according to $P(x_{\tT_z(x)-1}, \, . \, )$, where
$x\in\calP$ is such that $\tT_z(x)<+\infty$, then
$\P(\xi(Y_1)=\xi(Y_2))=0$. This additional condition is satisfied in
many practical cases, for example if the Markov Chain is defined as
the Euler-Maruyama discretization of a Langevin dynamics,
see~\eqref{eq:Lang}.

\comment{TONY: J'ai viré la modification supplémentaire qui ne me
  semble pas nécessaire: ou bien le lecteur sait faire, ou bien il n'a
pas compris et ce n'est pas grave.}


\subsubsection{Randomized level computation}\label{sec:random}

To run the AMS algorithm of Section~\ref{sec:Markov}, a sorting
procedure of the replicas according to their maximum levels is
required. More precisely,  at the initialization step, all replicas
must be sorted according to their maximum levels; at further
iterations, the procedure is faster, since only the new replicas that
have been resampled need to be sorted.

It is possible to propose algorithms within the GAMS framework which
never require the sorting of the entire system of replicas. The idea
is to sample at iteration $q$ a (small) random subset
$\mathcal{I}^{(q+1)}\subset I^{(q+1)}_{\rm on}$. The level 
$Z^{(q+1)}$ is then defined as the $k$-th order statistics of maximum levels computed
only on the replicas with labels in $\mathcal{I}^{(q+1)}$.
%
%
Notice that such algorithms introduce some flexibility in the implementation of the level computation, which may be useful to design efficient parallelization strategies to speed up the computation.

Notice that Assumption~\ref{ass:algo} on the stopping-levels
$(Z^{(q)})_{q \ge 0}$ is then satisfied by slightly modifying the definition of
the $\sigma$-fields indexed by $z$ in the level computation step as follows:
$$
\calF^{(q+1)}_z \eqdef \calF^{(q)} \vee \sigma( P^{(q+1)} ) \vee \filt_z^{\calX^{(q+1)}} \vee \sigma(\mathcal{I}^{(q+1)}).
$$

\subsubsection{Additional selection}\label{sec:select}

It is also possible to modify the branching rules so that larger
branching numbers are affected to replicas which are in areas which
have been identified as important to get an accurate estimate of
$\pi(\ph)$ (in the spirit of a sequential importance sampling algorithm). For instance, in the bi-channel case of
Section~\ref{sec:bi-channel}, it is possible to enforce a higher
probability of branching for replicas which visit the channel which is
not sampled sufficiently well. The only requirements to implement
these strategies is that the branching numbers are defined in such a
way that Assumption~\ref{ass:algo} is satisfied.



\subsubsection{Application to the sampling of a Gaussian bridge}\label{sec:bridge}

We presented above variants of the AMS algorithm. The GAMS framework also allows for
different general setting: splitting algorithms can be
used to sample other random variables than paths of Markov chains with levels defined as
$\sup\{\xi(X_{t \wedge \tau_A})_{t \ge 0}\}$ for some stopping time $\tau_A$ and
some reaction coordinate function $\xi$. Actually, under appropriate assumptions,
the following cases also enter into the setting of the GAMS framework:
path-dependent reaction coordinates (duration of the path, integral over
the path), sampling of continuous time
stochastic processes (diffusions, jump processes, branching processes), sampling of non-homogeneous stochastic processes,
 etc... \comment{TONY: OK ?}
%
Let us discuss in this section as an example the sampling of
a Gaussian bridge.

 Let $\kappa \in \N^\ast$ be given, and consider the following Gaussian bridge distribution in $\calP=\R^\kappa$:
\[
\pi( d x_1 \ldots d x_\kappa) = \frac{1}{\mathcal{Z}_\kappa} {\rm e}^{ -\frac{1}{2}\p{x_1^2+(x_1-x_2)^2+ \ldots +(x_{\kappa-1}-x_\kappa)^2 +x_\kappa^2 }} d x_1 \ldots d x_\kappa 
\]
where $\mathcal{Z}_\kappa > 0$ is the appropriate normalization constant. This distribution is
a discrete version of a Brownian Bridge, and can be interpreted as a
Gaussian random walk $(X_1,X_2, \ldots, X_{\kappa+1})$ starting from
$X_0=0$ and conditioned on $\{X_{\kappa+1}=0\}$.

The definition of the maximum level is
$\Xi(x)= \max \{ x_i: i \in \{1, \ldots, \kappa\}\}$
and we wish to implement the AMS algorithm to compute small probabilities of the form $\P(\Xi(X) > z) $ for some $z >0$. 

For this purpose, let us define
$
\tT_z(x) = \inf\{i \in \{ 1, \ldots, \kappa\} : x_i > z \}\in\left\{1,\ldots,\kappa,+\infty\right\},
$
and consider the filtration
$
\filt_z = \sigma(x_1, \ldots , x_{\tT_z(x)}). 
$

Let us now define the resampling kernels $\pi_z(x,dx')$. For a given
$x \in \R^\kappa$ assuming that $\tT_z(x) < + \infty$, let us introduce a
random variable $X' \in \R^\kappa$ such that
$X'_i = x_i$ for $i\in \{1, \ldots ,\tT_z(x)\}$ 
and
$
(X'_{\tT_z(x) +1}, \ldots, X'_\kappa) \sim \calB_{\kappa-\tT_z(x)}(x_{\tT_z(x)},0) 
$
where  for each $m \geq 1$, and $x_0,x_{m+1} \in \R$, $\calB_m(x_0,x_{m+1})$ denotes the Gaussian bridge distribution
\[
\calB_m(x_0,x_{m+1}) = \frac{1}{\mathcal{Z}_m} {\rm e}^{ -\frac{1}{2}\p{(x_0-x_1)^2+(x_1-x_2)^2+ \ldots +(x_{m-1}-x_m)^2 +(x_m-x_{m+1})^2 }} d x_1 \ldots d x_m.
\]
We then define
\begin{equation*}
  \pi_z(x,dx') = \Law(X').
\end{equation*}

This general setting enters into the GAMS framework, and satisfies in
particular Assumptions~\ref{ass:key_2} and~\ref{ass:key_1}
above. Assumption~\ref{ass:key_1} is a consequence of the following
Lemma, applied to $K=\tT_z(X) \wedge (\kappa-1)$ (using the fact that $\sigma(X_1, \ldots,
X_{\tT_z(X)})=\filt^X_z$). \comment{TONY: OK ?}
\begin{Lem}
Let $(X_1, \ldots, X_\kappa) \sim \calB_\kappa(x_0,x_{\kappa+1})$ and let $K$ be a
stopping time with respect to the natural filtration of $(X_1, \ldots,
X_\kappa)$ ({\it i.e.} for any $l \in \{1, \ldots, \kappa\}$, $\{K \le l\}
\subset \sigma(X_1, \ldots,
X_l)$) such that $K<\kappa$. Then, 
\[
\Law((X_{K+1}, \ldots, X_\kappa) \st (X_1, \ldots, X_{K})) = \calB_{\kappa-K}(X_{K},x_{\kappa+1}). 
\]
\end{Lem}
\begin{proof}
  First, the lemma is easily checked for $K$ a deterministic integer, using 
  the formula for conditional densities. Then the result is proven by conditioning on each value of $K$ and using the fact that $K$ is a stopping time.  
\end{proof}
This example can be generalized in various ways. First, it is possible
to build resampling kernels $\pi_z(x,dx')$ such that only the components
 $x_i$ such that $x_i > z$ are resampled. Second, the same kind of
 algorithms can be applied to
discrete Gaussian Markov random fields.


\section{The unbiasedness theorem}\label{sec:unbiased}
In the present section, the unbiasedness of the empirical distribution over weighted replicas is proven. This is the content of Theorem~\ref{th:unbiased}.
We first provide in Section~\ref{sec:summary} a summary of the notation used in the GAMS framework
of Section~\ref{sect:GAMS}. The latter will be helpful to follow the
statements and proofs of the present section. The main result is
stated in Section~\ref{sec:statement} and the last two sections~\ref{sec:proof_th} and~\ref{sec:proofs}
are devoted to the proof of this result.

\subsection{Summary of GAMS notation}\label{sec:summary}
We follow the algorithmic order of the GAMS framework of
Section~\ref{sect:GAMS}. In the following, $\ph: \calP \to \R$ denotes
a  bounded test function. We also introduce below a new notation for
an intermediate empirical distribution (see~\eqref{eq:pi_qhalf}).

\paragraph{The initialization step ($q=0$)} The system of weighted replicas is
denoted by $\calX^{(0)} = (X^{(n,0)},G^{(n,0)})_{n \in I^{(0)}}$ with
  uniform weights $G^{(n,0)}=1/ \card I^{(0)}$
for $n \in
  I^{(0)}$. The first level is $Z^{(0)}$ with the associated $\sigma$-field $\pcalF{0}=\filt_{Z^{(0)}}^{\calX^{(0)}}$.

\paragraph{Iterations} Iterate on $q \ge 0$ the following steps.

\paragraph{The stopping criterion}  If the stopping criterion is
satisfied, the algorithm stops at this stage, and we set $q=Q_{\rm
  iter}$. At the beginning of iteration $q$, the weighted empirical
distribution estimator (defined by~\eqref{GAMS_estimator}) is: 
\[
 \dps \hat{\pi}^{(q)}(\ph)  = \sum_{n \in I^{(q)}} G^{(n,q)} \ph(X^{(n,q)}).
\]

\paragraph{The splitting (branching) step}
The random branching numbers are denoted by $(B^{(n,q+1)})_{n \in
  I^{(q)}}$, the updated set of labels $I^{(q+1)}=(I^{(q)} \setminus
I^{(q+1)}_{\rm killed})\sqcup I^{(q+1)}_{\rm new} $, the associated
children-parent map $P^{(q+1)}:I^{(q+1)} \to I^{(q)} \setminus I^{(q+1)}_{\rm killed}$, and the associated new weights $(G^{(n',q+1)})_{n' \in I^{(q+1)}}$. All the latter variables are sampled conditionally on $\calF^{(q)}$, and are $\calF^{(q)} \vee \sigma(P^{(q+1)})$ measurable.
The weighted empirical distribution at this stage is denoted by
\begin{equation}
  \label{eq:pi_qhalf}
  \dps \hat{\pi}^{(q+1/2)}(\ph)  = \sum_{n' \in I^{(q+1)}} G^{(n',q+1)} \ph(X^{(P^{(q+1)}(n'),q)}).
\end{equation}

\paragraph{The resampling step} The system of weighted replicas after
resampling is denoted by $\calX^{(q+1)} = (X^{(n',q+1)},G^{(n',q+1)})_{n' \in I^{(q+1)}}$.

\paragraph{The level computation step} The new level is denoted by $Z^{(q+1)}$, the associated $\sigma$-field is defined as
$\calF^{(q+1)}\eqdef \calF^{(q)} \vee \sigma( P^{(q+1)} ) \vee
\filt_{Z^{(q+1)}}^{\calX^{(q+1)}}.$

\medskip

As already explained at the end of
Section~\ref{sec:GAMS_precise_def}, we will use use in the
following the whole
sequence $(\calX^{(q)})_{q \ge 0}$ of weighted replicas as well as the
whole sequence of related filtrations, which are simply obtained by
considering the algorithm without stopping criterion.

\subsection{Statement of the main result}\label{sec:statement}
The main theoretical result of this paper is the following.
\begin{The}\label{th:unbiased}
Let $(\calX^{(q)})_{0 \le q \le Q_{\rm iter}}$ be the sequence of
random systems of weighted replicas generated by an algorithm which
enters into the
GAMS framework of Section~\ref{sect:GAMS}. In particular, the
Assumptions~\ref{ass:key_2} and~\ref{ass:key_1}   on the general
setting (see Section~\ref{sec:gen_set}) as
well as the Assumption~\ref{ass:algo} 
on the stopping criterion, branching numbers and level computations
(see Section~\ref{sec:GAMS_to_algo}) are supposed to hold.

Assume moreover that the number of iterations $Q_{\rm iter}$ is almost
surely finite (this condition writes $\P(Q_{\rm iter} < + \infty)=1$) and that one of
the following conditions is satisfied:
\begin{itemize}
\item $Q_{\rm iter}$ is bounded from above by a deterministic constant,
\item or the almost sure mass conservation~\eqref{eq:mass_cons} is satisfied.
\end{itemize}
Then, for any bounded measurable test function $\ph:\calP \to \R$, $$\E\p{ \hat{\pi}^{(Q_{\rm iter})}(\ph) } = \pi(\ph).$$ 
\end{The}

Notice that a deterministic number of iterations $Q_{\rm iter}=q_{0}
\in \N$ satisfy the assumptions\footnote{To obtain
  $Q_{\rm iter}=q_{0}$, one simply has to choose
  $S^{(q)}=\begin{cases} 0 \quad \text{if}~q<q_{0}\\ 1 \quad
    \text{if}~q\geq q_{0}\end{cases}$.} of Theorem 4.1 so that in the
above setting (namely under Assumptions~\ref{ass:key_2}-\ref{ass:key_1}-\ref{ass:algo}): $$ \forall q_0 \ge 0,\, \E\p{ \hat{\pi}^{(q_0)}(\ph) } = \pi(\ph).$$


As a corollary of Theorem~\ref{th:unbiased} and thanks to the
discussion in Section~\ref{sec:justif_Markov} which shows that the
AMS algorithm of Section~\ref{sec:AMS_Markov} enters into the GAMS framework, we also obtain that the
AMS estimator~$\hat{\ph}=\hat{\pi}^{(Q_{\rm iter})}(\ph)$ defined by~\eqref{eq:phi_hat} in Section~\ref{sec:AMS_estimator} is an
unbiased estimator of $\pi(\varphi)$.


The strategy to prove this theorem is to introduce the
sequence of random variables
\begin{equation}\label{eq:def_martingale}
  \dps M^{(q)}(\ph) = \E\p{\hat{\pi}^{(q)}(\ph) \st \calF^{(q)}}
\end{equation}
for a fixed bounded measurable test function $\ph:\calP \to \R$
and to show that the process $\p{M^{(q)}(\ph) }_{q\in \N}$ indexed by
$q$ is a martingale with respect to the filtration
$\p{\calF^{(q)}}_{q\in \N}$. Since, by Proposition~\ref{propo_Qiter_stopping}, $Q_{\rm iter}$ is a
stopping time for this filtration, Doob's stopping theorem for
discrete-time martingales can then be applied to obtain
Theorem~\ref{th:unbiased}. The next two sections are devoted to the
proof of Theorem~\ref{th:unbiased}.

\subsection{Proof of Theorem~\ref{th:unbiased}}\label{sec:proof_th}

\comment{TONY: Le théorème n'utilise pas cette definition ni la
  proposition qui suit. On pourrait mettre ça dans la section
  suivante. Ceci dit, la proposition 4.3 permet de comprendre ce qui
  se passe. A discuter...}

The following definition of conditionally independent replicas will be
useful in the proof.
\begin{Def}\label{def:cond}
Let $Z$ be a random level, $I \subset \N^\ast$ a finite random set of indices and $\calG$ a $\sigma$-field of events. We
assume that $\sigma(I) \vee  \sigma(Z)\subset \calG$. We say that the
random system of replicas $(X^{(n)})_{n \in I}$ is independently
distributed with distribution $(\pi_{Z}(X^{(n)},\, . \,))_{ n \in I}$
conditionally on $\calG$, if for any sequence of bounded measurable
functions $(\ph_n)_{n\in I}$ from $\calP$ to $\R$, we have
\[
\E \p{ \prod_{n\in I} \ph_n(X^{(n)}) \st \calG } = \prod_{n\in I}  \pi_Z(\ph_n)(X^{(n)}). 
\]
\end{Def}

Let us now state two intermediate propositions before proving Theorem~\ref{th:unbiased}.
%
%
The first proposition states that, in the sense of
Definition~\ref{def:cond}, the set of replicas with indices in
$I^{(q)}$ (resp. $I^{(q+1)}$) are $\calF^{(q)}$-conditionally
independent (resp. $\calF^{(q)} \vee \sigma(P^{(q+1)})$-conditionally independent) with explicit distributions.
\begin{Pro}\label{pro:1}
Let us consider the setting of Theorem~\ref{th:unbiased}.  For any integer $q \geq 0$,
  \begin{enumerate}
\item[$(i)_{q}$] The replicas $( X^{(n,q)} )_{n \in I^{(q)}}$ are independent with distribution $\p{\pi_{Z^{(q)} } (X^{(n,q)}, \, . \, )}_{n \in I^{(q)}}$ conditionally on $\calF^{(q)}$.
  \item[$(ii)_{q}$]  The replicas $(X^{(n',q+1)} )_{n' \in I^{(q+1)}}$ are independent conditionally on $\calF^{(q)} \vee \sigma(P^{(q+1)})$, with distribution $\p{\pi_{Z^{(q)} } (X^{(n',q+1)} , \, . \,  )}_{n' \in I^{(q+1)}}$.
\end{enumerate}
\end{Pro}

The second proposition states intermediate equalities between
conditional averages of the empirical distributions, required to
obtain the desired martingale property of $\p{M^{(q)}(\ph)}_{q \geq
  0}$, and which are easily obtained from Proposition~\ref{pro:1}.

\begin{Pro}\label{pro:2}
Let us consider the setting of Theorem~\ref{th:unbiased}.  For any integer $q \geq 0$  and for any bounded measurable test function $\ph:\calP \to \R$,
\begin{enumerate}
  \item[$(iii)_{q}$] $
\E\p{ \hat{\pi}^{(q+1/2)}(\ph) 
\st \calF^{(q)}} = \E\p{  
\hat{\pi}^{(q)}(\ph)  \st \calF^{(q)}}
$.
\item[$(iv)_{q}$]
$
\E\p{  \hat{\pi}^{(q+1)}(\ph) 
\st \calF^{(q)} \vee \sigma(P^{(q+1)})} = \E\p{  \hat{\pi}^{(q+1/2)}(\ph) 
\st \calF^{(q)}\vee \sigma(P^{(q+1)})}
$.
  \end{enumerate}
\end{Pro}
The proofs of both  Proposition~\ref{pro:1} and Proposition~\ref{pro:2} are postponed to Section~\ref{sec:proofs}.
%
%
%
We are now in position to prove Theorem~\ref{th:unbiased}.
\begin{proof}[Proof of Theorem~\ref{th:unbiased}]
The proof consists in first proving that $\p{M^{(q)}(\ph)}_{q\geq 0}$
defined by \eqref{eq:def_martingale} is a $\p{\calF^{(q)}}_{q\geq
  0}$-martingale and then applying the Doob's optional stopping theorem.

%
Notice that $\E\p{M^{(q+1)}(\ph)  \st
  \calF^{(q)}}=\E(\hat{\pi}^{(q+1)}(\ph) \st \calF^{(q)})$ and let us compute the right-hand side. First, from point $(iv)_q$ of Proposition~\ref{pro:2} and since $\calF^{(q)}\subset \calF^{(q)}\vee \sigma(P^{(q+1)})$, we get
\begin{align*}
\E\p{ \hat{\pi}^{(q+1)}(\ph) \st \calF^{(q)}}& =\E\p{\E\p{ \hat{\pi}^{(q+1)}(\ph)  \st \calF^{(q)}\vee \sigma(P^{(q+1)})} \st \calF^{(q)}}\\
&=\E\p{   \hat{\pi}^{(q+1/2)}(\ph)  \st \calF^{(q)}}.
\end{align*}
Second, 
from point $(iii)_q$ of Proposition~\ref{pro:2} we have
\begin{align*}
\E\p{  \hat{\pi}^{(q+1/2)}(\ph)  \st \calF^{(q)}} 
=\E\p{ \hat{\pi}^{(q)}(\ph) \st \calF^{(q)}}.
\end{align*}
We thus have for any $q\geq 0$, 
\begin{equation}\label{martingale}
\E\p{ \hat{\pi}^{(q+1)}(\ph) \st \calF^{(q)}}=\E\p{ \hat{\pi}^{(q)}(\ph) \st \calF^{(q)}}
\end{equation}
and $\p{M^{(q)}(\ph) }_{q\in \N}$ is therefore a $\p{\calF^{(q)}}_{q\in \N}$-martingale.

We now focus on stopping the latter martingale at the random index
$Q_{\rm iter}$. By assumption, either the almost 
sure mass conservation property~\eqref{eq:mass_cons} is satisfied, in
which case $\p{M^{(q)}(\ph)}_{q\in
  \N}$ is a bounded martingale (since $\abs{M^{(q)}(\ph)} \leq \norm{ \ph
}_{\infty} $), or $Q_{\rm iter}\leq q_{\rm max} $ for some
deterministic real number $q_{\rm max} \in \R$. In both cases, we apply the Doob's optional stopping Theorem (see for instance~\cite{Shiryayev1984}, Chapter $7$, Section $2$, Theorem $1$ and Corollaries $1$ and $2$) to the martingale $\p{M^{(q)}(\ph)}_{q\in \N}$ and with the stopping time $Q_{\rm iter}$ with respect to the filtration $\p{\calF^{(q)}}_{q\in \N}$. We obtain
\[
\E\p{ \hat{\pi}^{(Q_{\rm iter})}(\ph) } = \E \left(M^{(0)}(\ph)\right)= \pi(\ph)
\]
which concludes the proof of Theorem~\ref{th:unbiased}.
\end{proof}

\subsection{Proofs of Propositions~\ref{pro:1} and~\ref{pro:2}}\label{sec:proofs}

Proposition~\ref{pro:1} requires an additional intermediate result,
namely the propagation Lemma~\ref{lem:Doob} below. This lemma
gives rigorous conditions under which the property on a system of
replicas $(X^{(n)})_{n \in I}$  of being
independently distributed with  distribution $(\pi_{Z}(X^{(n)}, \,
. \,))_{n\in I }$ conditionally on $\calF$ can be transported from the $\sigma$-field
$\calF$ to a larger $\sigma$-field.  It is based on Doob's optional
stopping theorem for martingales indexed by the level variable
$z$. Notice that it is the only result where the right
continuity property of Assumption~\ref{ass:key_2} is explicitly used.

\begin{Lem}\label{lem:Doob}
Let us assume that Assumptions~\ref{ass:key_2} and~\ref{ass:key_1} hold. Let $Z \in \R \cup\set{-\infty,+\infty}$ be a random level, $\calG$ a $\sigma$-field, and $I \subset \N^\ast$ a finite random set of labels. Assume that $\sigma(I)\vee\sigma(Z) \subset \calG$. Consider a random system of replicas $\calX= (X^{(n)})_{n \in I}$, which is independently distributed with distribution $(\pi_{Z}(X^{(n)}, \, . \,))_{n\in I }$ conditionally on $\calG$ (in the sense of Definition~\ref{def:cond}).
%
%
Set
\begin{equation}\label{filt:calGz}
\forall z \in \R, \, \calG_z \eqdef \calG \vee \filt_z^{\calX},
\end{equation}
and assume that $Z' \in  \R \cup\set{-\infty,+\infty}$ is a stopping
level for the filtration $\p{\calG_z}_{z\in \R}$ such that, almost surely, $Z' \geq  Z$.

Then the replicas $(X^{(n)})_{n \in I}$ are independently distributed conditionally on  $\calG_{Z'}$, with distribution $\p{\pi_{Z'}(X^{(n)}, \, . \,)}_{n\in I} $.
\end{Lem}
\begin{proof}[Proof of Lemma~\ref{lem:Doob}]
\underline{Step 1}. 
The first step consists in proving that for any fixed $z\in \R$, the system of replicas is independently distributed with distribution $\p{\pi_{Z \vee z}( X^{(n)} , \, . \, )}_{n\in I}$ conditionally on $\calG \vee \filt_{z}^{\calX}$. By a standard monotone class argument, it is sufficient to show that
\[
\E \p{\prod_{n \in I} \ph_n(X^{(n)}) \psi_n(X^{(n)}) \, Y} = \E \p{\prod_{n \in I} \pi_{Z \vee z}(\ph_n)(X^{(n)}) \psi_n(X^{(n)}) \, Y},
\]
where $(\ph_n)_{n \geq1}$ ranges over bounded measurable test
functions from $\calP$ to $\R$, $(\psi_n)_{n \geq1}$ ranges over
$\filt_z$-measurable test functions from $\calP$ to $\R$, and $Y$ over bounded
$\calG$-measurable random variables. Let us denote $\mathcal{I}=\E
\p{\prod_{n \in I} \ph_n(X^{(n)}) \psi_n(X^{(n)}) \, Y}$ the left-hand side.
Since $Y$ is $\calG$-measurable, by Definition~\ref{def:cond} of the conditional independence we get that
\[
  \mathcal{I}=\E \p{\prod_{n\in I} \pi_{Z}( \ph_n\psi_n)(X^{(n)} )   Y }.
\]
The functions $(\psi_n)_{n \geq 1}$ being $\filt_z$-measurable, they are {\it a fortiori} $\filt_{z \vee z'}$-measurable for any $z' \in \R$. Assumption~\ref{ass:key_1} on the resampling family $\p{\pi_z}_{z\in\R}$ then yields
\begin{align*}
\pi_{z'}( \ph_n \psi_n)(x) =\pi_{z'}(\pi_{z' \vee z}(\psi_n  \ph_n)) ( x)  = \pi_{z'}(\psi_n \pi_{z' \vee z}( \ph_n)) (x) .
\end{align*}
As a consequence, using again that the system of replicas
$(X^{(n)})_{n \in I}$ is independently distributed with distribution
$\p{\pi_{Z}( X^{(n)} , \, . \, )}_{n\in I}$ conditionally on $\calG$
and that $Y$ is $\calG$-measurable, we get the following identity
\begin{align*}
\mathcal{I}= \E \left( \prod_{n \in I}\pi_{Z}\left(\psi_n \pi_{Z \vee z}( \ph_n)\right) ( X^{(n)}) Y\right) =\E \p{ \prod_{n \in I} \pi_{Z \vee z }( \ph_n)(X^{(n)} )  \psi_n(X^{(n)}) Y },
\end{align*}
and this concludes the first step.

%


\medskip

\noindent\underline{Step 2.} We now prove the main claim of this
lemma, namely the fact that
the replicas $(X^{(n)})_{n \in I}$ are independent with distribution
$\p{\pi_{Z'}(X^{(n)}, \, . \,)}_{n\in I}$ conditionally on
$\calG_{Z'}$. Let us first assume that the test functions
$\p{\ph_n}_{n\in I}$ are continuous from $\calP$ to $\R$.

In order to come back to a classical setting to apply Doob's optional
stopping theorem, let us introduce a  continuous, one-to-one and strictly increasing
change of level parametrization ${\mathcal Z}:[0,1] \to \R \cup
\set{-\infty,+\infty}$. Let us consider the following stochastic process indexed by $t \in [0,1]$:
\[
N_t = \E \p{ \prod_{n \in I} \ph_n(X^{(n)})  \,\Big|\, \calG_{{\mathcal Z}(t)}}.
\]
It is a bounded (since $I$ is $\calG_{{\mathcal Z}(t)}$-measurable for all $t$) and thus uniformly integrable martingale with respect to the filtration $(\calG_{{\mathcal Z}(t)})_{t
  \in [0,1]}$. In addition, $N_1=\E\left( \prod_{n \in I}
\ph_n(X^{(n)}) \st \calG_{+\infty}\right)$ where $\calG_{+\infty}=\calG
\vee \filt^{\chi}_{+\infty}$.\comment{Tony: OK pour tout le monde ?}
Thanks to Step~$1$ above, we get: almost surely, for all $t \in [0,1]$,
\[
N_{t} = \prod_{n\in I} \pi_{Z \vee {\mathcal Z}(t) }( \ph_n)(X^{(n)} ).
\]
Therefore, $N_t$ is almost surely a right-continuous bounded martingale from Assumption~\ref{ass:key_2} on $(\pi_z)_{z \in \R}$.
By assumption, $T' = {\mathcal Z}^{-1}(Z')$ is a $\p{\calG_{{\mathcal Z}(t)}}_{t \in
  [0,1]}$-stopping level, and we can use a Doob's optional stopping
argument for right continuous bounded martingales (see for instance~\cite[Theorem 3.22]{KaratzasShreve91}) to get
\[
\E\p{N_{1} \st \calG_{{\mathcal Z}(T')}}=N_{T'}
\]
which can be rewritten as (since $Z' \ge Z \as$)
\[
\E\p{\prod_{n\in I} \ph_n(X^{(n)} ) \st \calG_{Z'} } = \prod_{n\in I} \pi_{Z'}(\ph_n)(X^{(n)}).
\]
This equality actually holds for any sequence of bounded measurable functions $\p{\ph_n}_{n\in I}$ since continuous bounded functions are separating. This concludes the proof of Lemma~\ref{lem:Doob}.


\end{proof}

Thanks to Lemma~\ref{lem:Doob}, we can now prove Proposition~\ref{pro:1}.

\begin{proof}[Proof of Proposition~\ref{pro:1}]
We proceed by induction on the iteration index $q\geq 0$. We first prove directly the statement $(i)_{q} \Rightarrow (ii)_{q}$
and then $(ii)_{q} \Rightarrow (i)_{q+1}$ using
Lemma~\ref{lem:Doob}. The initialization step consists in proving
$(i)_{0}$ using Lemma~\ref{lem:Doob}. In this proof, $(\ph_n)_{n
  \geq1}$ denotes a sequence of bounded measurable test
functions from $\calP$ to $\R$.

\medskip

\noindent\emph{Proof of $(i)_0$.}
The statement $(i)_0$ reads
\[
     \E \p{ \prod_{n \in I^{(0)}} \ph_n\p{X^{(n,0)}}  \st \calF^{(0)}} =  \prod_{n \in I^{(0)}}  \pi_{Z^{(0)}} (\ph_n)(X^{(n,0)} ) , 
  \]
where $\calF^{(0)}=\filt^{\calX^{(0)}}_{Z^{(0)}}$. This is exactly the result of Lemma~\ref{lem:Doob}, taking $Z=-\infty$, $Z'=Z^{(0)}$, $\calG =\sigma(I^{(0)})$, and recalling that the replicas are initially independent and distributed according to $\pi$.

\medskip

\noindent\emph{Proof of $(i)_{q} \Rightarrow (ii)_{q}$.}
Assume that $(i)_{q}$ holds. We rewrite property $(ii)_q$ as follows
\begin{equation}
  \label{eq:indepq}
  \E \p{ \prod_{n' \in I^{(q+1)}}  \ph_{n'} \p{X^{(n',q+1)}} \st \calF^{(q)}  \vee \sigma(P^{(q+1)}) } = \prod_{n' \in I^{(q+1)}}   \pi_{Z^{(q)} } (\ph_{n'})( X^{(n',q+1)}),
\end{equation}
and we now prove this identity.



Let us recall that in the resampling step, the replicas with labels in $I^{(q)}$ are not
  resampled and the replicas $(X^{(n',q+1)})_{n' \in I_{\rm
      new}^{(q+1)}}$ are sampled in such a way that they are
  independently distributed with distribution $(
  \pi_{Z^{(q)}}(X^{(P^{(q+1)}(n'),q)}, \,  . \,) )_{n' \in I_{\rm
      new}^{(q+1)}}$ conditionally on $\calF^{(q)} \vee
  \sigma(P^{(q+1)})$. Therefore, by definition of the total set of
  labels $I^{(q+1)}$ given in~\eqref{set_of_labels}, one obtains
\begin{align*}
&\E \p{ \prod_{n' \in I^{(q+1)}}  \ph_{n'} \p{X^{(n',q+1)}} \st \calF^{(q)}  \vee \sigma(P^{(q+1)}) }  \\
& \qquad  =\E \p{ \prod_{n \in I^{(q)} \setminus I_{\rm killed}^{(q+1)} }  \ph_{n} \p{X^{(n,q)}} \prod_{n' \in I_{\rm new}^{(q+1)} }  \ph_{n'} \p{X^{(n',q+1)}} \st \calF^{(q)} \vee \sigma(P^{(q+1)}) } \\
& \qquad=\prod_{n' \in I_{\rm new}^{(q+1)} } \pi_{Z^{(q)}}(\ph_{n'})(X^{(P^{(q+1)}(n'),q)}) \E \p{\prod_{n \in I^{(q)} \setminus I_{\rm killed}^{(q+1)} }  \ph_{n} \p{X^{(n,q)}}  \st \calF^{(q)} \vee \sigma(P^{(q+1)}) }.
\end{align*}
Next, from the induction hypothesis $(i)_q$, the replicas
$(X^{(n,q)})_{n \in I^{(q)}}$ are independent with distribution
$\p{\pi_{Z^{(q)}}(X^{(n,q)}, \, . \,)}_{n\in I^{(q)}}$ conditionally
on  $\calF^{(q)}$. Since $P^{(q+1)}$ is sampled conditionally on
$\calF^{(q)}$, the replicas $(X^{(n,q)})_{n \in I^{(q)}}$ are also
independent conditionally on $\calF^{(q)}\vee \sigma(P^{(q+1)})$, with
the same distributions. Therefore (notice that $I^{(q)}$ and $ I_{\rm
    killed}^{(q+1)}$ are $\calF^{(q)}\vee \sigma(P^{(q+1)})$-measurable)
\begin{align*}
\E \p{\prod_{n \in I^{(q)} \setminus I_{\rm killed}^{(q+1)} }  \ph_{n} \p{X^{(n,q)}}  \st \calF^{(q)} \vee \sigma(P^{(q+1)}) } 
&=\prod_{n \in I^{(q)} \setminus I_{\rm killed}^{(q+1)} } \pi_{Z^{(q)}}(\ph_{n})(X^{(n,q)})\\
&= \prod_{n' \in I^{(q)} \setminus I_{\rm killed}^{(q+1)} } \pi_{Z^{(q)}}(\ph_{n'})(X^{(P^{(q+1)}(n'),q)}).
\end{align*}
Gathering the results leads to
$$  \E \p{ \prod_{n' \in I^{(q+1)}}  \ph_{n'} \p{X^{(n',q+1)}} \st \calF^{(q)}  \vee \sigma(P^{(q+1)}) } = \prod_{n' \in I^{(q+1)}}   \pi_{Z^{(q)} } (\ph_{n'})( X^{(P^{(q+1)}(n'),q)}).
$$
From the resampling step and Assumption~\ref{ass:key_1}, the following identity holds: 
\begin{equation}\label{eq:resampling_identity}
\forall q \ge 0, \, \forall n' \in I^{(q+1)},\qquad \pi_{Z^{(q)}}(X^{(n',q+1)}, \, . \,)=\pi_{Z^{(q)}}(X^{(P^{(q+1)}(n'),q)}, \,  . \,).
\end{equation} This concludes the proof of~\eqref{eq:indepq}.

\noindent \emph{Proof of $(ii)_q \Rightarrow (i)_{q+1}$}.
Let us now assume that $(ii)_{q}$ holds. To prove that $(i)_{q+1}$ holds, it is sufficient to check that
\[
    \E \p{ \prod_{n \in I^{(q+1)}} \ph_n\p{X^{(n,q+1)}}  \st \calF^{(q+1)} } = 
\quad \prod_{n \in I^{(q+1)}}  \pi_{Z^{(q+1)}}(\ph_n)( X^{(n,q+1)} )
. 
  \]
This is again exactly the result of Lemma~\ref{lem:Doob} applied to
$\calX^{(q+1)}$, taking $Z=Z^{(q)}$, $Z'=Z^{(q+1)}$ and $\calG =
\calF^{(q)} \vee \sigma(P^{(q+1)})$ so that $\calG_z = \calF^{(q+1)}_z$
(where, we recall, $\calF^{(q+1)}_z$ is  defined by Equation~\eqref{def_filt_calF}).
\end{proof}




Finally, let us prove Proposition~\ref{pro:2}.
\begin{proof}[Proof of Proposition~\ref{pro:2}]
The first equality $(iii)_{q}$ is a direct consequence of the
definition of the branching numbers. The second equality
$(iv)_{q}$ is obtained as a consequence of  Proposition~\ref{pro:1} by combining $(i)_{q}$ and $(ii)_{q}$.

\medskip

\noindent\emph{Proof of $(iii)_q$}.
The proof of this assertion is a direct application of the branching rule. Indeed, by definition of the weights 
$G^{(n',q+1)}$ given in \eqref{weights}, by definition of the branching numbers $(B^{(n,q+1)} )_{n \in I^{(q)}}$ as the number of offsprings of the $n$-th replica,
and because these branching numbers are independent of $(G^{(n,q)},X^{(n,q)})_{n \in I^{(q)}}$ conditionally on $\calF^{(q)}$, we get
  \begin{align*}
 \E\p{ \dps \hat{\pi}^{(q+1/2)}(\ph)\st \calF^{(q)}} 
& = \E\p{ \sum_{n' \in I^{(q+1)} } \frac{G^{(P^{(q+1)}(n'),q)}}{ \E\p{ B^{(P^{(q+1)}(n'),q+1)} \st \calF^{(q)}} }  \ph(X^{(P^{(q+1)}(n'),q)} )  \st \calF^{(q)} }  \\
&  = \E\p{  \sum_{n \in I^{(q)} } \frac{G^{(n,q)}}{ \E\p{ B^{(n,q+1)} \st \calF^{(q)}} } B^{(n,q+1)} \ph(X^{(n,q)} ) \st \calF^{(q)} }  \\
& = \E\p{ \dps \hat{\pi}^{(q)}(\ph) \st \calF^{(q)} } .
\end{align*}

\medskip

\noindent\emph{Proof of $(i)_{q} + (ii)_{q} \Rightarrow (iv)_q$}.
Using successively $(i)_{q}$, the
identity~\eqref{eq:resampling_identity} and  $(ii)_q$, we have:
\begin{align*}
\E\p{ \dps \hat{\pi}^{(q+1/2)}(\ph) \st \calF^{(q)} \vee \sigma(P^{(q+1)}) } 
&= \sum_{n' \in I^{(q+1)}} G^{(n',q+1)} \pi_{Z^{(q)}} ( \ph )( X^{(P^{(q+1)}(n'),q)} )\\
&= \sum_{n' \in I^{(q+1)}} G^{(n',q+1)} \pi_{Z^{(q)}}( \ph )( X^{(n',q+1)} )  \\
&= \E\p{ \dps \hat{\pi}^{(q+1)}(\ph) \st \calF^{(q)} \vee \sigma(P^{(q+1)}) }.
\end{align*}
\end{proof}

\section{Numerical illustration}\label{sect:num}

The aim of this section is to illustrate the behavior of the AMS algorithm as defined in
Section~\ref{sec:Markov}, in various situations involving
discrete-time approximations~\eqref{eq:Lang} of the overdamped Langevin dynamics in
dimension $1$ and $2$. 


We would like to discuss in particular the unbiasedness of the AMS estimator
$\hat{p}$ of $p=\P(\tau_B < \tau_A)$ (see the formula~\eqref{eq:p_hat}) whatever the choice of the reaction coordinate~$\xi$, the number of replicas~$\nrep$ and the minimal number~$k$ of
replicas which are declared retired and resampled at each iteration of
the AMS algorithm.  Indeed, from Theorem~\ref{th:unbiased},
we know that
\[
\forall~\xi, \, \nrep,\, k, \quad \E\left(\hat{p}\right) = p=\P(\tau_B < \tau_A).
\]

In the following, $(\hat{p}_m)_{1 \leqslant m
  \leqslant N}$ refers to
independent realizations of the estimator $\hat{p}$ obtained by $N$ independent runs
of the algorithm and the associated empirical mean is denoted by
\begin{equation}\label{eq:emp_mean}
\overline{p}_N = \frac{1}{N} \sum_{m=1}^N\hat{p}_m.
\end{equation}

The variance of the estimator
$\hat{p}$ is also investigated numerically, and it is shown in some
two-dimensional situations that the variance heavily depends on the choice of
the reaction coordinate. In the following, we will denote by
\begin{equation}\label{eq:emp_delta}
\delta_N=2 \times \frac{1.96}{\sqrt{N}} \times \sqrt{\frac{1}{N} \sum_{m=1}^N (\hat{p}_m)^2 - (\overline{p}_N)^2}
\end{equation}
the size of the 95\% empirical confidence interval computed using the
empirical variance obtained over $N$ independent runs of the algorithm.

The section is organized as follows. In Section~\ref{sect:BrownDrift},
we illustrate on a simple one-dimensional test case the importance of
a proper implementation of the branching and splitting steps in order
to obtain an unbiased estimator of~\eqref{eq:phi}. Then, in
Sections~\ref{sec:bi-channel} and~\ref{sec:allen-cahen}, we give two examples in dimension 2 on which we discuss the efficiency of
the AMS algorithm by studying how the convergence of the estimator
depends on the parameters $\xi$,  $\nrep$ and $k$. Finally, in
Section~\ref{sec:conc}, we draw some conclusions and practical
recommendations from these numerical experiments.

\subsection{One-dimensional example: Brownian-drift dynamics}\label{sect:BrownDrift}

Let us first consider a one-dimensional example, with a reaction
coordinate $\xi: \R \to \R$ which is an increasing function. Of course, in this
situation, the AMS algorithm does not depend on
$\xi$. The aim is thus here to show the unbiasedness of the estimator
whatever $\nrep$ and $k$. Moreover, we would like to illustrate the fact that incorrect implementations of the branching and splitting steps may lead to strongly biased results. 

\paragraph{The model}

Let $(X_t)_{t\geqslant 0} \in \R $ be a drifted Brownian motion, starting at $x_0$, with drift $-\mu <0$ and inverse temperature $\beta$: for any $t\geq 0$, $X_t=x_0-\mu t+\sqrt{2\beta^{-1}}W_t$, where $(W_t)_{t\geqslant 0}$ is a standard Brownian motion. We use the explicit Euler-Maruyama method with a time-step size $\Delta t>0$: for any $i\in\N$
\[
X_{i+1}=X_i - \mu \Delta t + \sqrt{2 \beta^{-1} \Delta t} \, G_i, \quad X_0=x_0,
\]
where the random variables $(G_i)_{i\in \N}$ are independent standard Gaussian random variables. Given $a < x_0 < b$, we consider the estimation of
\[
p= \mathbb{P}( \tau_b < \tau_{a}) 
\]
where  $\tau_b \in \N$ and $\tau_{a} \in \N$ are the first hitting times of $ ]b, + \infty[$ and $]-\infty,a[$ respectively.

In the sequel, we choose the initial condition $x_0=1$, as well as the
two barriers $a=0.1$ and $b=z_{\rm max}=1.9$. We choose $\mu=1$ and we use the
following values for the inverse temperature
$\beta\in\left\{8,24\right\}$ in order to have a range of
estimated probabilities over several orders of magnitude. Moreover the
time-step size is $\Delta t=0.1$.  The number of independent runs is $N=6.10^6$.

\paragraph{Biased algorithms}

In order to highlight the importance of a proper implementation of the
splitting and resampling steps when many replicas
have $Z^{(q)}$ as a maximum level, we perform tests with slightly modified versions of the AMS algorithm, which happen to yield biased estimators.
Two biased versions are considered.
\begin{itemize}
\item Version 1: We first consider the algorithm where the number of resampled
  replicas is exactly $k$ even if $K^{(q+1)} > k$, namely if more than $k$ replicas have a
  maximum level smaller than the current level $Z^{(q)}$ of the
  algorithm.
\item Version 2: We then consider a situation where the replicas are possibly
  resampled from a state with a reaction coordinate equal to the current level $Z^{(q)}$.
\end{itemize}

More precisely, the first version (Version $1$) of a biased algorithm is obtained by
modifying the AMS algorithm of Section~\ref{sec:AMS_Markov} as
follows: 
\begin{enumerate}[(i)]
\item At iteration $q$, the working replicas are ordered according to
  their maximum level (possibly with equalities) and the
$k$ working replicas with smallest maximum level are declared
  retired, and the others remain working replicas. The level $Z^{(q)}$ is still defined as the maximum level of
  the $k$-th newly retired replica. As explained in
  Remark~\ref{rem:same_max}, some of the remaining working replicas may have
  their maximum levels equal to $Z^{(q)}$. Then, as in the classical
  AMS algorithm, $k$ new replicas are resampled by picking (randomly) an
  initial condition among the current working replicas
  with maximum level strictly larger than $Z^{(q)}$ (namely with
  labels in $I^{(q)}_{{\rm on}, > Z^{(q)}}$). Replicas with
  maximum equal to $Z^{(q)}$ are not split. The resampling kernel $\pi_z$
  is the same as in the AMS algorithm, see Section~\ref{sec:resampling_kernel}.
\item At the end of the algorithm, the probability $p$ is estimated by
\begin{equation}\label{eq:proba_biased_1}
\p{\frac{\nrep-k}{\nrep}}^{Q_{\rm iter}} \left(\frac{1}{\nrep}\sum_{n\in I^{(Q_{\rm
        iter})}_{\rm on}}\mathds{1}_{\tT_B(X^{(n,Q_{\rm iter})})<\tT_A(X^{(n,Q_{\rm iter})})}\right),
        \end{equation}
which is consistent with the general updating formula~\eqref{weights} for the
weights in the GAMS framework.
\end{enumerate}

The second version (Version $2$) of a biased algorithm is obtained by
modifying the AMS algorithm of Section~\ref{sec:AMS_Markov} exactly as
for Version~$1$, except for the resampling step. Item (i) is modified as:
\begin{itemize}
\item[(i-bis)] At iteration $q$, the $k$ working replicas with
  smallest maximum are declared  ``retired''. Then, $k$ new replicas are resampled by picking (randomly) an initial
  condition among the states of \emph{all} the current 
  replicas with maximum levels larger than $Z^{(q)}$, including those
  with maximum levels equal to $Z^{(q)}$. The resampling kernel $\pi_z$
  (defined by~\eqref{eq:piz_chain_1}--\eqref{eq:piz_chain})
  is modified accordingly: the parent trajectory $x$
  is copied up to the time $\tilde{T}_{z}(x)=\inf\left\{t \geq 0 : \xi(x_t)\geq z\right\}$ (with
  a large inequality) instead of $\tT_z(x)$ (which involves a strict
  inequality, see~\eqref{eq:Tz}) and then completed using the Markov dynamics.
\end{itemize}

These biased versions do not enter into
the GAMS framework. Note that for both modified versions and contrary to the
classical AMS algorithm, the sequence of levels $Z^{(q)}$ is not
necessarily strictly increasing. These modifications will increase the
number of iterations of the algorithm, which in view of
\eqref{eq:proba_biased_1} explains the bias towards lower probability
(see results below).

\paragraph{Results}

In Table~\ref{TableAMS_beta_8} and Table~\ref{TableAMS_beta_24},
estimations with the AMS algorithm are given. We observe that
the estimated probability is stable under changes of $n_{\rm rep}$ and
$k$, with a small confidence interval. This yields the reference
values $p=3.60 10^{-4}$ for $\beta =8$, and $p=1.2 10^{-10}$ for $\beta=24$. For $\beta=8$, we have checked
that these results are
in agreement with
those obtained by a standard direct Monte-Carlo estimation with $6.10^8$ realizations.

In Table~\ref{TableBiais} (where e-$n$ stands for $10^{-n}$), estimations with the two biased versions of
the AMS algorithm with $k=1$ are given. Even for $\beta=8$ (for which the target probability is $p=3.60 10^{-44}$), the bias is non-negligible. For $\beta=24$, the probability can be underestimated by a multiplicative factor $100$ with Version $1$.
The bias induced by using incorrect implementations of the AMS algorithm can thus be very large.

Notice that we have chosen a relatively large timestep ($\Delta
t=0.1$) in order to easily extract the bias introduced by
inappropriate modifications of the AMS algorithm. When the timestep
size $\Delta t$ goes to 0, the two variants we
discussed above get close to the classical AMS algorithm and the bias
disappears, since the probability to observe two replicas with the
same maximum level goes to zero (see
Remark~\ref{rem:same_max}) and the probability that
$\tilde{T}_z(x)$ is different of $\tT_z(x)$ for a parent trajectory
also goes to zero. Finally, we recall that an unbiased variant where exactly $k$
replicas are resampled at each iteration has been presented in Section~\ref{sec:remove}.

\begin{table}[htbp]
\begin{center}
\begin{tabular}{|c||c||c|c|c||c||c||}
$\nrep$ & 10 & 50 & 50 & 50 & 100 & 200  \\
$k$& 1 & 1 & 10 & 20 & 1 & 1  \\
$\overline{p}_N$ & 3.60e-4 & 3.596e-4 & 3.596e-4 & 3.597e-4 & 3.597e-4 & 3.597e-4\\
$\delta_N$ & 0.01e-4 & 0.004e-4 & 0.004e-4 & 0.006e-4 & 0.003e-4 & 0.002e-4
\end{tabular}
\caption{Results obtained on the 1d test case with the
  classical AMS algorithm and $\beta=8$.}
\label{TableAMS_beta_8}
\end{center}
\end{table}
\begin{table}[htbp]
\begin{center}
\begin{tabular}{|c||c||c|c|c||c||c||}
$\nrep$ & 10 & 50 & 50 & 50 & 100 & 200  \\
$k$& 1 & 1 & 10 & 20 & 1 & 1  \\
$\overline{p}_N$ & 1.20e-10 & 1.21e-10 & 1.21e-10 & 1.21e-10 & 1.205e-10 & 1.203e-10 \\
$\delta_N$ & 0.3e-10 & 0.03e-10 &  0.03e-10 & 0.03e-10 &  0.01e-10 & 0.005e-10 \\
\end{tabular}
\caption{Results obtained on the 1d test case with the
  classical AMS algorithm and $\beta=24$.}
\label{TableAMS_beta_24}
\end{center}
\end{table}
\begin{table}[htbp]
\begin{center}
\begin{tabular}{|c||c|c|c||c|c|c||}
Version & 1 & 2 & 2 & 1 & 2 & 2\\
$\nrep$ & 100 &100 & 10 & 100 & 100 & 10 \\
$\beta$ & 8 & 8 & 8 & 24 & 24 & 24\\
$p$ & 3.60e-4  & 3.60e-4 & 3.60e-4  & 1.2e-10 & 1.2e-10 & 1.2e-10 \\
$\overline{p}_N$ & 1.74e-4  & 3.257e-4 & 2.96e-4 & 1.40e-12 & 6.05e-11 & 5e-11 \\
$\delta_N$ & 0.03e-4 & 0.003e-4 & 0.01e-4 & 0.1e-12 &  0.07e-11 & 1.5e-11 \\
\end{tabular}
\caption{Results obtained on the 1d test case with biased versions of AMS.}
\label{TableBiais}
\end{center}
\end{table}

\subsection{The first two-dimensional example: the bi-channel problem}\label{sec:bi-channel}

The aim of this two-dimensional example is to investigate the
importance of the choice of the reaction coordinate on the efficiency
of the algorithm, on a typical example which has been used in previous
numerical studies, see~\cite{CerouGuyaderLelievrePommier2011,metzner-schuette-vanden-eijnden-06,park-sener-lu-schulten-03}.

\subsubsection{The model}

We consider the following two-dimensional overdamped Langevin dynamics:
\begin{equation}\label{eq:overdamped}
dX(s)=-\nabla \mathcal{E}(X(s))ds+\sqrt{2\beta^{-1}}dW_s,
\end{equation}
where $W$ is a two-dimensional Wiener process and $\beta>0$ is the inverse temperature.

We use again the Euler-Maruyama method for the time-discretization of
the process $X$. The time step is denoted by $\Delta t$. For our
numerical simulations we take $\Delta t=0.05$. For an initial
condition $X_0=x_0 \in \R^2$ and for $m\in\N$, the numerical scheme reads
\[
X_{m+1}=X_{m}-\Delta t\,\nabla \mathcal{E}(X_{m})+\sqrt{2\beta^{-1}}\,(W_{(m+1)\Delta t}-W_{m\Delta t}),
\]
where $W_{(m+1)\Delta t}-W_{m\Delta t}=\sqrt{\Delta t} \, G_{m}$ and
$(G_m)_{n\in\N}$ are independent Gaussian random variables
in $\R^2$ with zero mean and covariance $\Id$.

In the simulations below, the initial condition is
$X_0=x_{0}=(-0.9,0)$. The potential $\mathcal{E}:\R^2\rightarrow \R$ is given by
\[
\mathcal{E}(x,y)=0.2x^4+0.2\left(y-\frac{1}{3}\right)^4+3e^{-x^2-\left(y-\frac{1}{3}\right)^2}-3e^{-x^2-\left(y-\frac{5}{3}\right)^2}-5e^{-(x-1)^2-y^2}-5e^{-(x+1)^2-y^2}.
\]


This potential is plotted on Figure \ref{Fig:Potential_Two_Channels}.
This potential has two global minima connected one to another by two
channels: the upper channel (which goes through the shallow minimum
around $(0,1.5)$) and the lower channel (which goes through the
saddle point around $(0,-0.5)$). The two global
minima are close to $m_A = (x_A,y_A)=(-1,0)$ and $m_B =
(x_B,y_B) = (1,0)$. For some $\rho\in ]0,1[$, we consider the sets $A$ and $B$ defined as the Euclidean open balls of radius $\rho$ around the two minima $m_A$ and $m_B$, namely
\[
\begin{cases}
A=\calB(m_A,\rho)=\left\{ (x,y)\in\R^2 : \sqrt{(x-x_A)^2+(y-y_A)^2}<\rho\right\}\\
B=\calB(m_B,\rho)=\left\{ (x,y)\in\R^2 : \sqrt{(x-x_B)^2+(y-y_B)^2}<\rho\right\}.
\end{cases}
\]
In the numerical applications, we take $\rho=0.05$. Most of the
trajectories starting from $x_0$ hit $A$ before $B$. Moreover, $A$ and
$B$ are metastable states: in the small temperature regime, starting
from $A$ (resp. $B$), it takes a lot of time to leave $A$ (resp. $B$).

We are interested in the estimation of the probability
$p=\P(\tau_B<\tau_A)$,
where the first hitting times $\tau_A$ and $\tau_B$ are defined by
\[
\tau_A=\inf\left\{m\in \N : X_m\in A\right\} \quad \text{ and }
\quad \tau_B=\inf\left\{m\in \N : X_m\in B\right\}.
\]

\begin{figure}[h]
\centering
\includegraphics[scale=0.45]{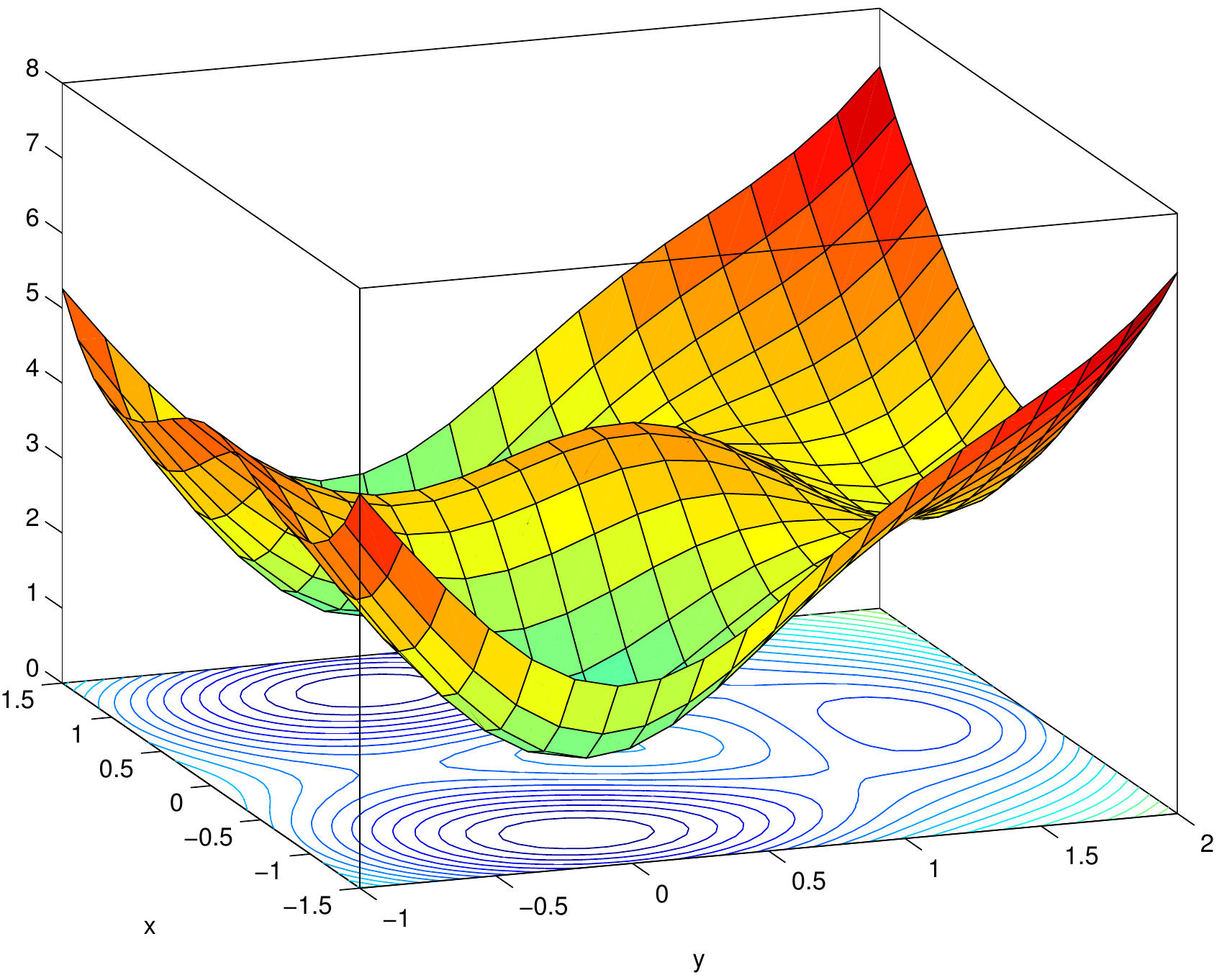}
\caption{Plot of the potential function for the bi-channel problem.}
\label{Fig:Potential_Two_Channels}
\end{figure}

We will consider the results of the AMS algorithm for the three reaction coordinates $\xi^i$ with $i\in\left\{1,2,3\right\}$:
\begin{enumerate}
\item the \textit{norm to the initial point $m_A$:} $\xi^{1}(x,y)=\sqrt{(x-x_A)^2+(y-y_A)^2}$, 
\item the \textit{norm to the final point $m_B$:} $\xi^{2}(x,y)=\xi^{1}(x_B,y_B)-\sqrt{(x- x_B)^2+(y-y_B)^2}$, 
\item the \textit{abscissa:} $\xi^{3}(x,y)=x$.
\end{enumerate}
The maximum levels used in the stopping criterion of the algorithm are
$z_{\rm max}^{1}=z_{\rm max}^{2}=1.9$ and $z_{\rm
  max}^{3}=0.9$. Notice that for $i \in \{1,2,3\}$, we have $B\subset
(\xi^i)^{-1}(]z_{\rm max}^i,+\infty[)$ (see~\eqref{eq:hyp_B}).

In Section~\ref{sec:bi-channel}, we take $k=1$ and the number of
replicas is $\nrep=100$. The values of $\beta$ belong to the set
$\left\{8.67,9.33,10\right\}$ which are associated with probabilities
$p$ ranging approximately from $2.10^{-9}$ to $1.10^{-10}$.

\subsubsection{Evolution of the empirical mean}

Let us first perform simulations with $N$ independent
runs of the algorithm, $N$ varying between $1$ and $6.10^6$. We represent on
Figure~\ref{Joran2D} the evolution as a function of $N$ of the
empirical mean $\overline{p}_{N}$ (defined by \eqref{eq:emp_mean}) and
of the associated $95\%$ confidence intervals
$[\overline{p}_{n}-\delta_N/2,\overline{p}_{n}+\delta_N/2]$ computed using the empirical variance, see \eqref{eq:emp_delta}.

The colors in the figures are as follows: green (solid line) for
$\xi^1$, red (line with crosses) for $\xi^2$ and  blue (line with circles) for $\xi^3$. The full lines represent the evolution of the upper and lower bounds of the confidence intervals, while dotted lines represent the evolution of the empirical means.

\begin{figure}[!htp]
\centering
\includegraphics[scale=0.45]{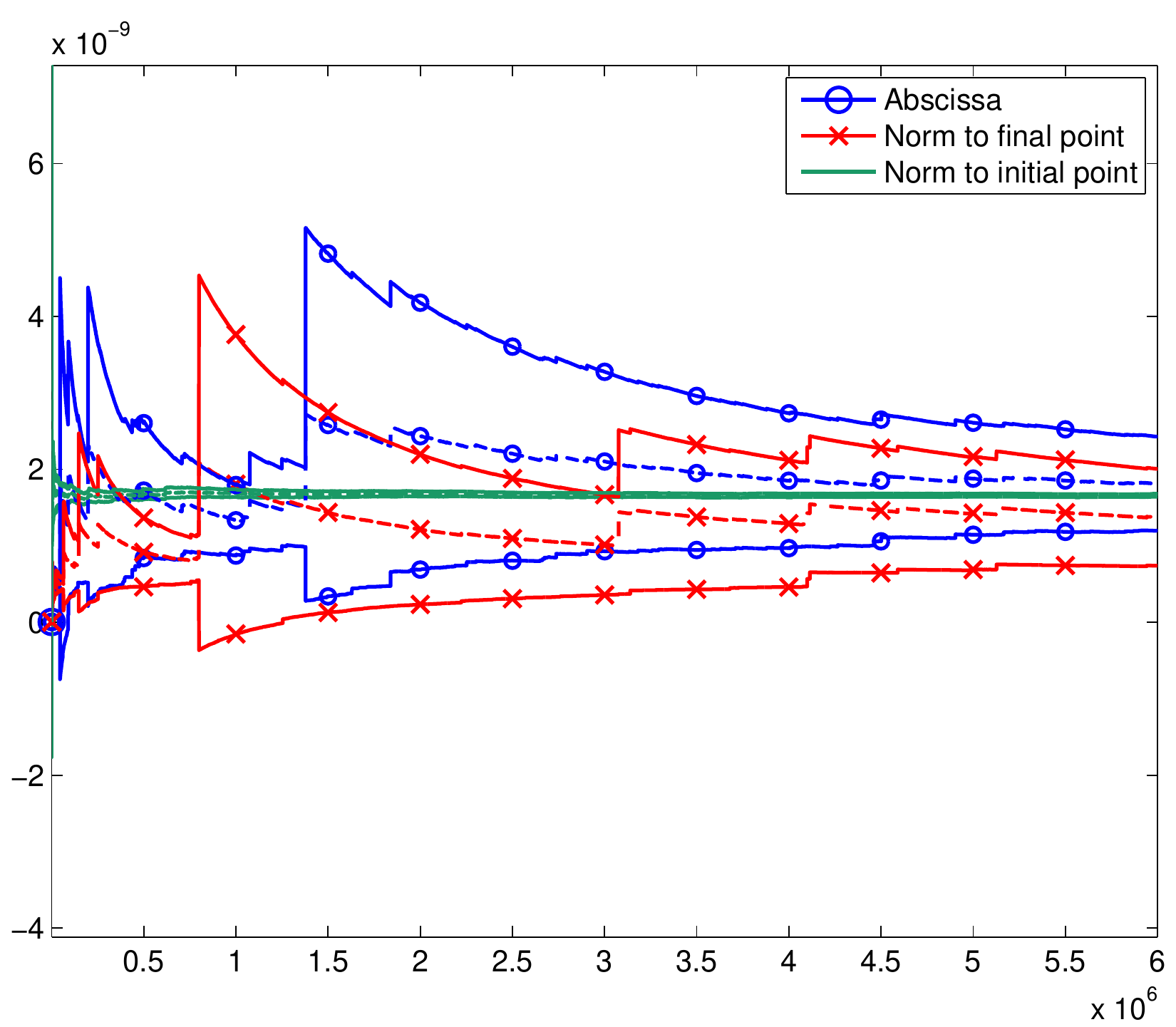} \includegraphics[scale=0.4]{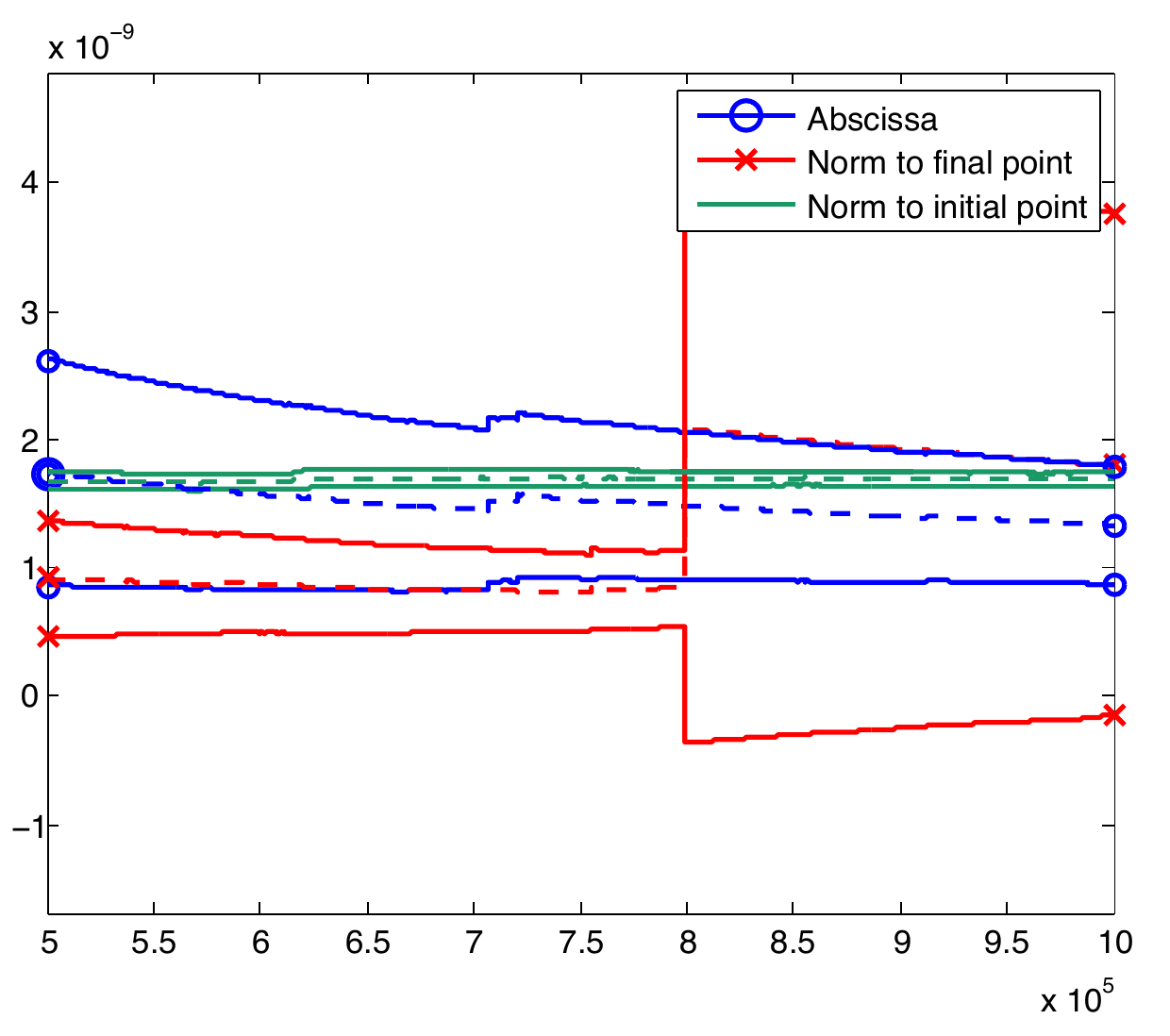}\\
\includegraphics[scale=0.45]{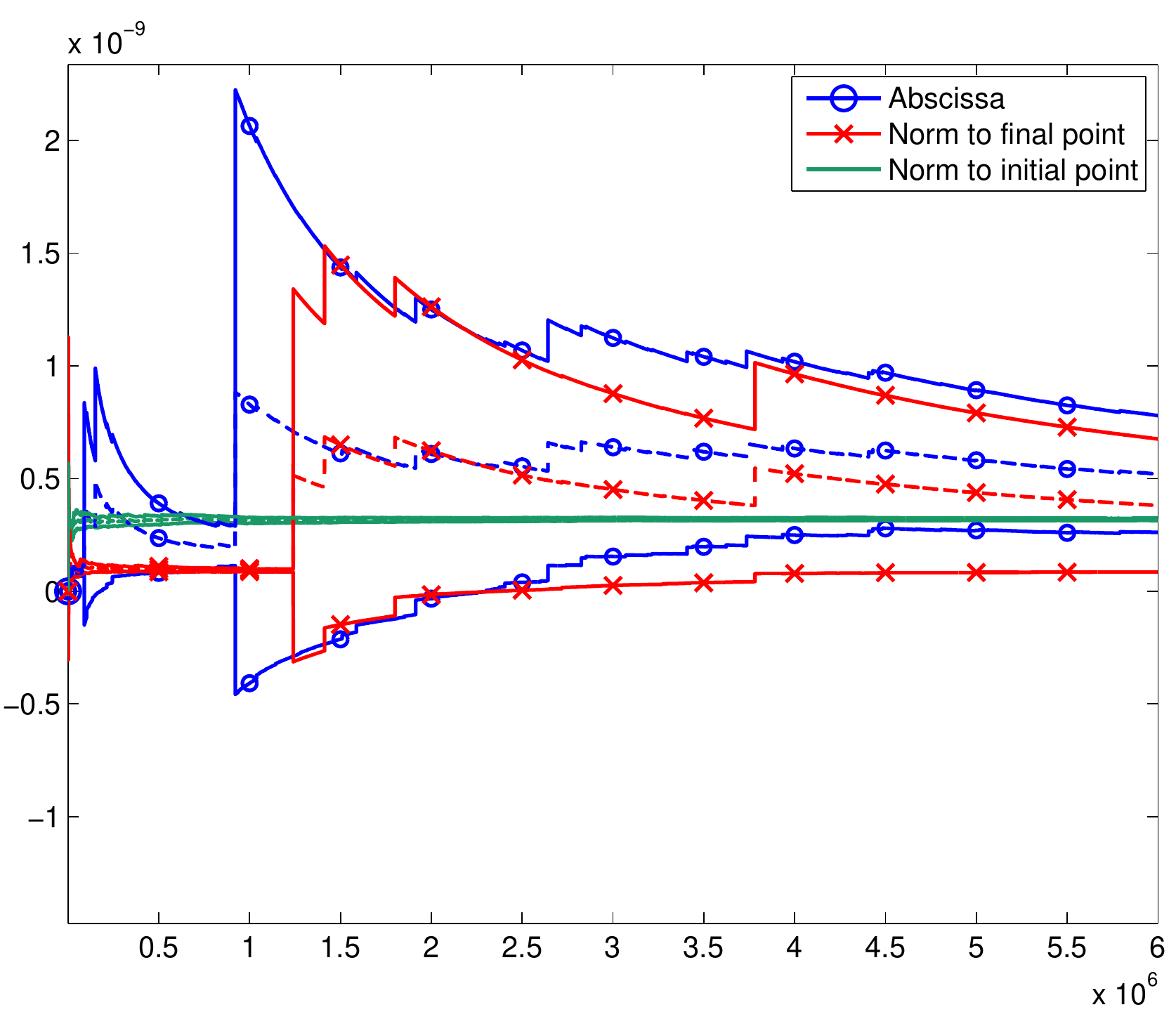} \includegraphics[scale=0.4]{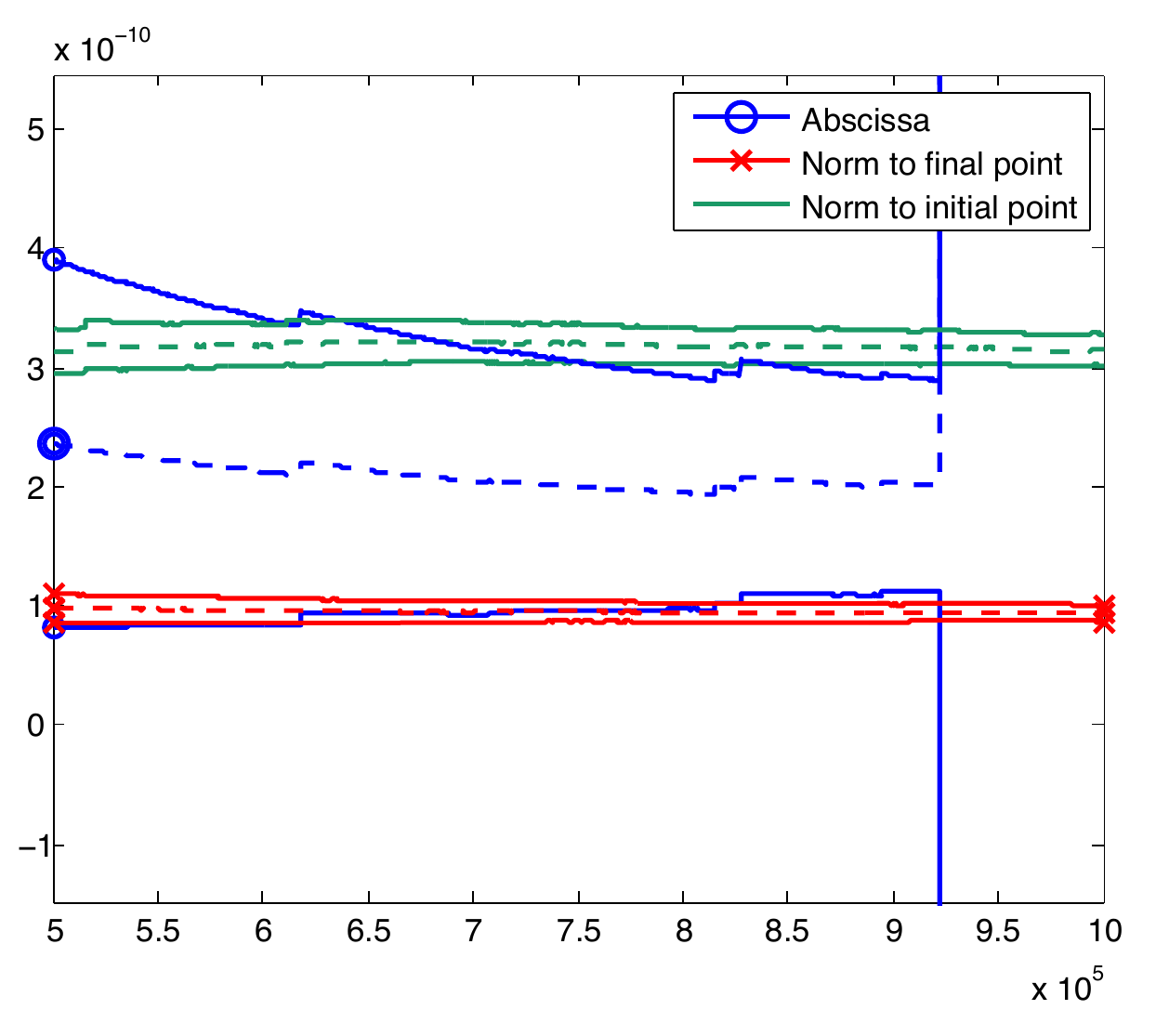} \\
\includegraphics[scale=0.45]{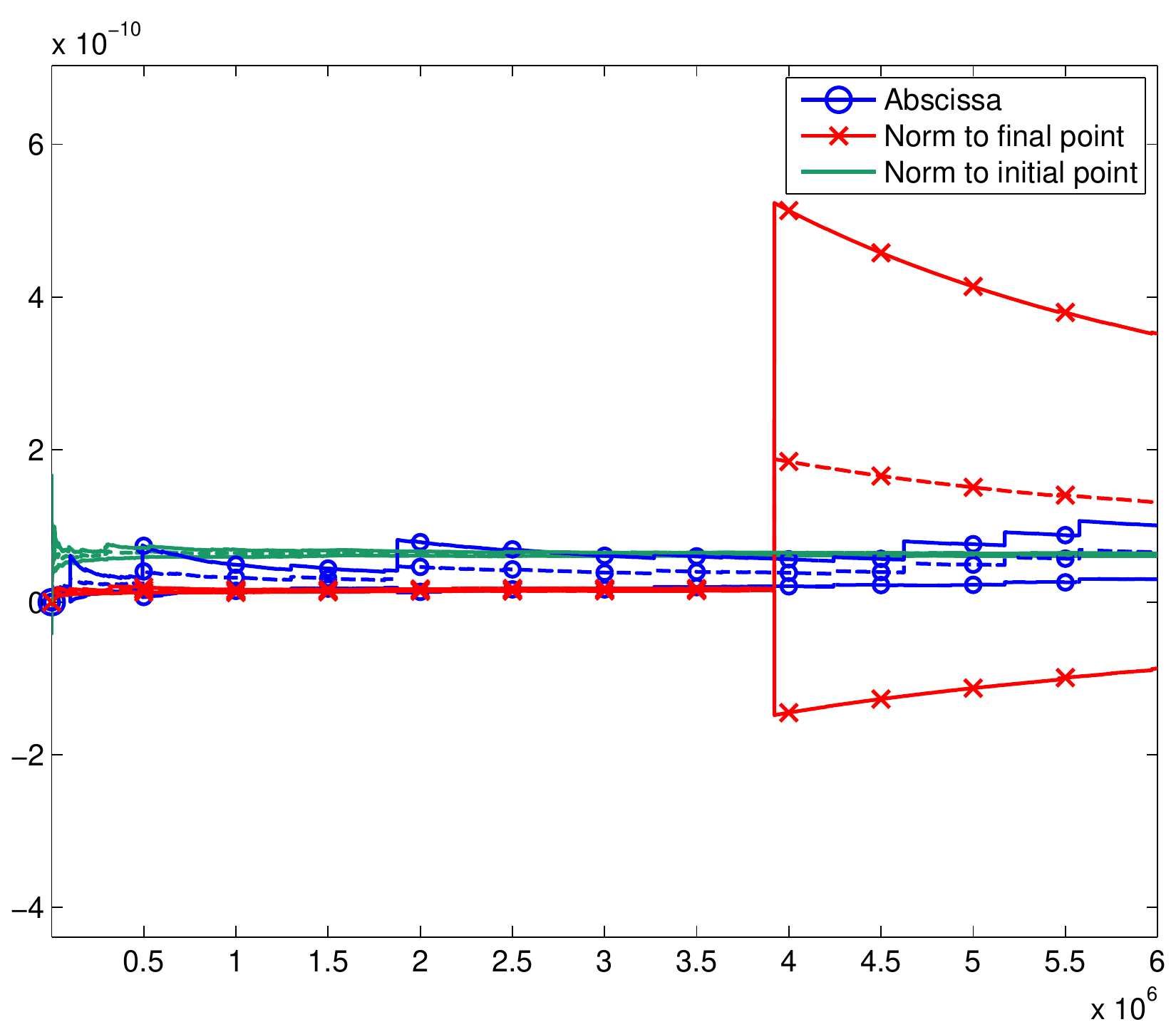}\includegraphics[scale=0.4]{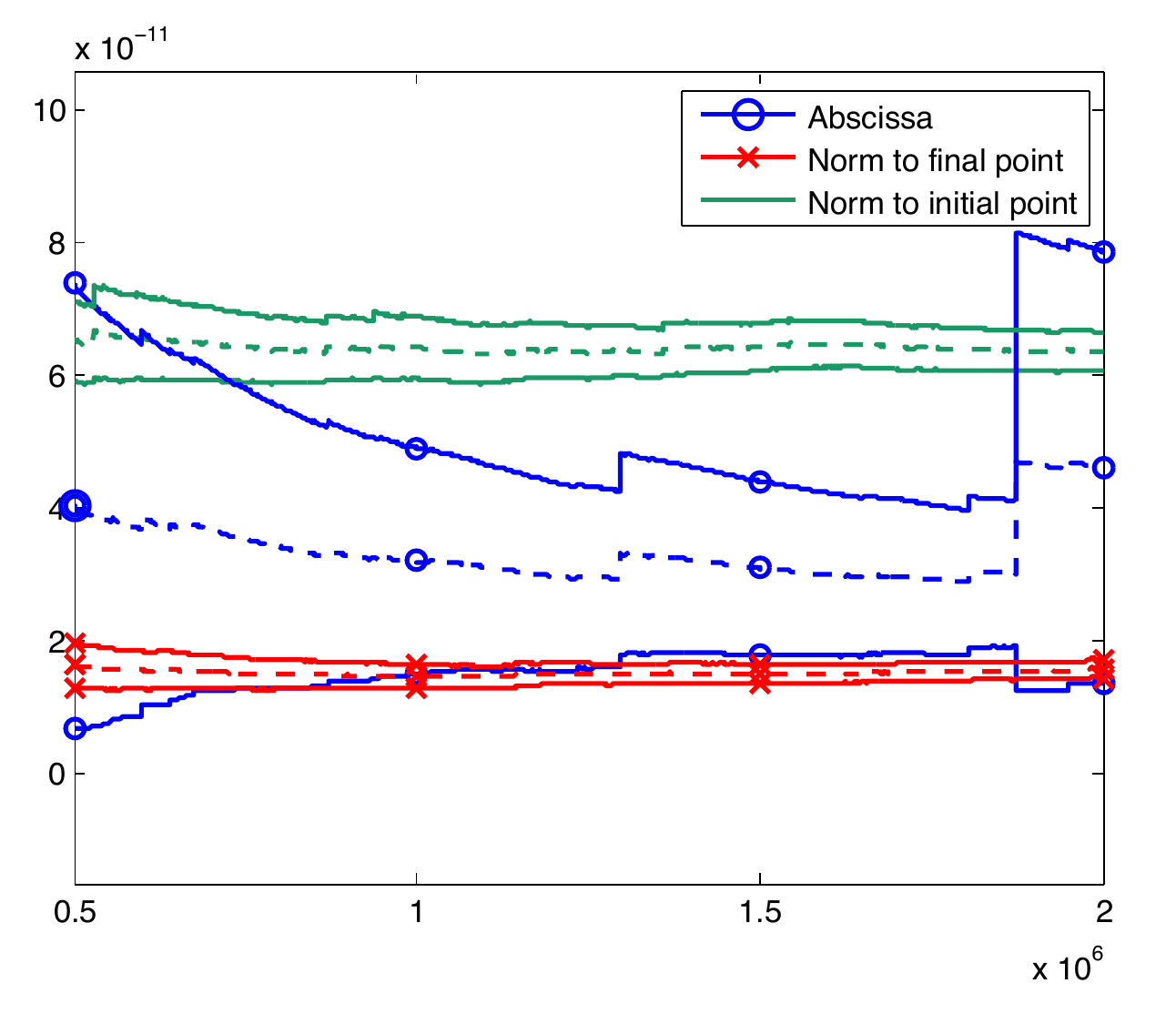}
\caption{Evolution as a function of $N$ of the
empirical mean $\overline{p}_{N}$ and
of the associated $95\%$ confidence intervals
$[\overline{p}_{n}-\delta_N/2,\overline{p}_{n}+\delta_N/2]$. Upper to lower $\beta=8.67, 9.33, 10$. The right inserts are 
  zooms of the left graphs on smaller values of $N$, in order to
  illustrate the ``apparent bias'' phenomenon.}
\label{Joran2D}
\end{figure}

From these simulations, we observe that:
\begin{itemize}
\item When $N$ is sufficiently large, the confidence intervals
  overlap. This is in agreement with the fact that $\hat{p}$
  is an unbiased estimator of $p$ whatever the choice of the reaction coordinate.
\item The statistical fluctuations depend a lot on the reaction
  coordinate. In particular, the results obtained with $\xi^1$ seem
  much better than with $\xi^2$ or $\xi^3$. We will come back to this
  in Section~\ref{sec:fluct_channel}.
\item The confidence interval being computed empirically, one may
  conclude that the algorithm is biased by considering the results for
  $N$ too small (see for example the graphs in the right column in Figure~\ref{Joran2D}). This is due to the fact that the empirical variance dramatically
  underestimates the real variance if $N$ is too small.  This is a well-known
  phenomenon for splitting algorithms in general called ``apparent
  bias'', see \cite{GlassermanHeidelbergerShahabuddinZajic1998}. As $\beta$ gets larger (namely as the
  temperature gets smaller), the number of independent runs $N$
  required to observe overlapping confidence intervals gets larger.
  For example, for $\beta=13$, and $N=6.10^6$, we were not able to get overlapping
confidence intervals for the three reaction
coordinates: this is actually an indication of the fact
that one should pursue computations with larger $N$ in order to get reliable estimates.
\end{itemize}

We observe that there are some realizations for which the estimator of the probability is very
large. These realizations have small probability but they dramatically increase the value of the
empirical mean and of the empirical variance. This explains the jumps which are observed on
the empirical average and confidence interval as a function of the number of realizations, see Figure \ref{Joran2D}.
In the next section, we illustrate this aspect. As is usually the case with
Monte Carlo simulations for rare event simulations, it is impossible
to decide \textit{a priori} if the sample size $N$ is sufficiently large to
give an accurate estimation.

\subsubsection{Heavy tails}

In this section, we give a more quantitative interpretation of the
above observations on the evolution of the empirical mean. For a given
inverse temperature $\beta$, and a given reaction coordinate, we
sample $N$ independent realizations of the algorithm denoted by
$\p{\hat{p}_{m}}_{1\leq m\leq N}$.

To illuminate the importance of the largest values of the estimator in
the empirical average, we compute partial empirical averages over the largest
values or over the smallest values among the $N$ realizations. More
precisely, for a fixed $N_0$, we compute
\begin{equation*}
\overline{p}_{N}^{N_0,{\rm
    large}}=\frac{1}{N_0}\sum_{\ell=N-N_0+1}^{N}\hat{p}_{(\ell)}\text{
  and }
\overline{p}_{N}^{N_0,{\rm small}}=\frac{1}{N-N_0}\sum_{\ell=1}^{N-N_0}\hat{p}_{(\ell)},
\end{equation*}
where $\p{\hat{p}_{(\ell)}}_{1\leq \ell \leq N}$ denotes the order
statistics of $\p{\hat{p}_{m}}_{1\leq m\leq N}$. In other words,
$\overline{p}_{N}^{N_0,{\rm large}}$ is the average over the $N_0$
largest realizations of $\hat{p}_{m}$, while
$\overline{p}_{N}^{N_0,{\rm small}}$ is the average over all the other
values. In particular,
$$\overline{p}_N=\frac{N_0}{N}\overline{p}_{N}^{N_0,{\rm large}}+\left(1-\frac{N_0}{N}\right)\overline{p}_{N}^{N_0,{\rm small}}.$$

In Table~\ref{TableExtremes}, we use $N=6.10^6$, and $N_0=100$ or
$N_0=1000$. The results are obtained with the same realizations as those used in the previous sections. We observe that the largest values
contribute a lot to the empirical average over all the runs whatever the reaction
coordinate. This is characteristic of a heavy tailed distribution.
Indeed, for $\beta=10$ for example, the value of
$\overline{p}_{N}^{N_0,{\rm small}}$ differs a lot from the value of
$\overline{p}_N$. In addition, notice that this difference is more important when
using $\xi^2$ and $\xi^3$, while when using $\xi^1$. This is in
agreement with the fact that we observed better results (in terms of
statistical fluctuations) with $\xi^1$ than with $\xi^2$ and $\xi^3$
in Figure~\ref{Joran2D}. Indeed, distributions with heavier tails
typically lead  to larger fluctuations.

\begin{table}[htbp]
\begin{center}
\begin{tabular}{c|c|c||c|c|c||c|c|c|c||}
 & $\beta$ & $\overline{p}_N$ & $N_0$ & $\overline{p}_{N}^{N_0,{\rm large}}$ & $\overline{p}_{N}^{N_0,{\rm small}}$  & $N_0$ & $\overline{p}_{N}^{N_0,{\rm large}}$ & $\overline{p}_{N}^{N_0,{\rm small}}$\\
$\xi^1$ & $8.67$ & 1.7e-09 & 1000 & 1.3e-e-06 & 1.5e-09 & 100 & 4.0e-06 & 1.6e-09 \\
$\xi^2$ & $8.67$ & 1.3e-09 & 1000 & 5.7e-06 & 4.2e-10 & 100 & 5.2e-05 & 5.1e-10\\
$\xi^3$ & $8.67$ & 1.8e-09 & 1000 & 8.0e-06 & 4.8e-10 & 100 & 6.5e-05 & 7.3e-10\\

$\xi^1$ & $9.33$ & 3.2e-10 & 1000 & 3.0e-07 & 2.7e-10 & 100 & 1.0e-06 & 3.0e-10 \\
$\xi^2$ & $9.33$ & 3.8e-10 & 1000 & 1.9e-06 & 6.4e-11 & 100 & 1.8e-05 & 8.2e-11\\ 
$\xi^3$ & $9.33$ & 5.2e-10 & 1000 & 2.7e-06 & 7.2e-11 & 100 & 2.5e-05 & 1.0e-10\\

$\xi^1$ & $10$ & 6.2e-11 & 1000 & 8.9e-08 & 4.8e-11 & 100 & 4.0e-07 & 5.6e-11 \\
$\xi^2$ & $10$ & 1.3e-10 & 1000 & 7.0e-07 & 9.1e-12 & 100 & 7.2e-06 & 1.3e-11\\
$\xi^3$ & $10$ & 6.5e-11 & 1000 & 3.0e-07 & 1.0e-11 & 100 & 3.0e-06 & 1.5e-11\\
\end{tabular}
\end{center}
\caption{Comparisons of the empirical averages $\overline{p}_N$ with partial empirical
  averages obtained over the $N_0$ largest and $N-N_0$ smallest values among 
  $N$ realizations.}
\label{TableExtremes}
\end{table}


\begin{Rem}[On heavy tails in the idealized setting]\label{rem:heavy_tail}
Let us mention another setting where it can be analytically checked
that the distribution of the estimator $\hat{p}$ has heavy tails in
some regime. In a one-dimensional case (or when the committor function
is used as a reaction coordinate), one can show that the AMS
algorithm can be related to the so-called exponential case,
namely when $\Xi(X)$ is distributed with exponential law with
parameter one (see for example~\cite{BrehierLelievreRousset2015}). It
is then easy to show that for $k=1$, and in the regime where $\nrep
\rightarrow +\infty$ and the
target probability is $p=\exp(-\sigma^2 \nrep)$ for some fixed
$\sigma>0$, then $\frac{\hat{p}}{p}$ converges in distribution to the
log-normal distribution $\exp(\sigma Z-\sigma^2/2)$ where
$Z\sim\mathcal{N}(0,1)$ is a standard Gaussian random variable (see~\cite[Proposition 3.4]{Brehier_15}). In particular, $\hat{p}$ has heavy
tails (the median $\exp(-\sigma^2/2)$ of this distribution is smaller than its expectation $1$, so that for a large $\sigma$ the empirical mean under-estimates the expectation).
\end{Rem}

\subsubsection{Fluctuations induced by the two channels}\label{sec:fluct_channel}

In this section, we compare the results when using two reaction coordinates: $\xi^1$ (norm to $m_A$) and $\xi^3$ (abscissa). Since the typical behavior we observe in Figure~\ref{Joran2D} and in Table~\ref{TableProbaCond} is the same for $\xi^2$ and $\xi^3$, we do not repeat the analysis for $\xi^2$.

As explained above, there are two possible channels for the reactive
trajectories going from $A$ to $B$: the upper channel and the lower
channel. For each realization $m$, one can distinguish  the
contributions to the estimator $\hat{p}_m$ of the replicas going
through the upper channel and the ones going through the lower
channel. In the following, for a given path, the trajectory is
associated to the upper (resp. lower) channel if the first hitting
point of the $y$-axis is such that $y>0.5$ (resp. such that
$y \le 0.5$). More precisely, let us define $\Pi_1(x,y)=x$ and
$\Pi_2(x,y)=y$ for any $(x,y)\in\R^2$. For a replica $X=(X_t)_{t\in \N}$ such that
$\tau=\inf\left\{t\in\N : \Pi_1\left(X_{t}\right)>0\right\}<\infty$,
$X\in\text{Upper}$ if $\Pi_2\left(X_{\tau}\right)>0.5$ and
$X\in\text{Lower}$ if $\Pi_2\left(X_{\tau}\right)\le 0.5$.

For each run of the algorithm, we compute the three following quantities:
\begin{itemize}
\item the number of replicas which reach $B$ before $A$:
\[\M^{B}=\sum_{n\in I_{{\rm on}}^{(Q_{{\rm iter}})}}\mathds{1}_{\tT_B(X^{(n,Q_{{\rm iter}})})<\tT_A(X^{(n,Q_{{\rm iter}})})}\]
\item the number of replicas which reach $B$ before $A$ and go through the upper channel:
$$\M^{B,{\rm upper}}=\sum_{n\in I_{{\rm on}}^{(Q_{{\rm iter}})}}\mathds{1}_{\tT_B(X_{}^{(n,Q_{{\rm iter}})})<\tT_A(X_{}^{(n,Q_{{\rm iter}})})}\mathds{1}_{X_{}^{(n,Q_{{\rm iter}})}\in \text{Upper}}$$
\item the number of replicas which reach $B$ before $A$ and go through the lower channel:
$$\M^{B,{\rm lower}}=\sum_{n\in I_{{\rm on}}^{(Q_{{\rm iter}})}}\mathds{1}_{\tT_B(X^{(n,Q_{{\rm iter}})})<\tT_A(X_{}^{(n,Q_{{\rm iter}})})}\mathds{1}_{X_{}^{(n,Q_{{\rm iter}})}\in \text{Lower}}$$
\end{itemize}
Notice that $\M^{B}=\M^{B,{\rm upper}}+\M^{B,{\rm lower}}$ and that
$\M^B \neq 0$ is equivalent to $\hat{p}\neq 0$. When
needed, we explicitly indicate the dependence of these quantities on the
realization by a lowerscript $m$: for $m \in \{1, \ldots ,N\}$, we thus denote
$\M^{B}_m$, $\M^{B,{\rm upper}}_m$ and $\M^{B,{\rm lower}}_m$ the
$m$-th realization of $\M^{B}$, $\M^{B,{\rm upper}}$ and $\M^{B,{\rm lower}}$.

Let us introduce the set
$\Ens_N=\left\{m : \hat{p}_m\neq 0\right\}$ of realizations which
lead to a non zero $\hat{p}$
and the proportion $R_N=\card \Ens_N /N$ of such
realizations. We now divide the realizations in $\Ens_N$ into three
disjoint subsets, with associated proportions.
\begin{itemize}
\item All replicas reaching $B$ before $A$ go through the upper channel:
\[\Ens_{N}^{\rm upper}=\left\{m\in\Ens_N : \M_{m}^{B,{\rm lower}}=0
\right\}  \text{ and }   \rho^{\rm upper}_N=\frac{\card \Ens_{N}^{\rm
    upper} }{\card \Ens_N }.\]
\item All replicas reaching $B$ before $A$ go through the lower channel:
\[\Ens_{N}^{\rm lower}=\left\{m\in\Ens_N : \M_{m}^{B,{\rm upper}}=0
\right\}  \text{ and }   \rho^{\rm lower}_N=\frac{\card\Ens_{N}^{\rm
    lower}}{\card \Ens_N}.\]
\item Both channels are used by the replicas reaching $B$ before $A$ :
\[\Ens_{N}^{\rm mix}=\Ens_N \setminus \left(\Ens_{N}^{\rm upper} \cup
  \Ens_{N}^{\rm lower} \right) \text{ and }  \rho^{\rm
  mix}_N=\frac{\card\Ens_{N}^{\rm mix}}{\card \Ens_N}.\]
\end{itemize}
Obviously, $\rho^{\rm mix}_N=1-\rho^{\rm upper}_N-\rho^{\rm lower}_N$.
Finally, we define conditional estimators for $\hat{p}$ associated
with the partition of $\Ens_N$ defined above:
\begin{equation*}
\tilde{p}_{N}^{\rm upper}=\frac{\sum_{m\in \Ens_{N}^{\rm
      upper}}\hat{p}_m}{\card \Ens_{N}^{\rm
      upper}} \,,\quad
\tilde{p}_{N}^{\rm lower}=\frac{\sum_{m\in \Ens_{N}^{\rm
      lower}}\hat{p}_m}{\card\Ens_{N}^{\rm
      lower}}\quad\text{   and   }\quad
\tilde{p}_{N}^{\rm mix}=\frac{\sum_{m\in \Ens_{N}^{\rm mix}}\hat{p}_m}{\card\Ens_{N}^{\rm mix}}.
\end{equation*}

Notice that
$$\overline{p}_{N}=R_N\left(\rho^{\rm upper}_N\tilde{p}_{N}^{\rm upper}+\rho^{\rm lower}_N\tilde{p}_{N}^{\rm lower}+\rho^{\rm mix}_N \tilde{p}_{N}^{\rm mix}\right).$$
In other words, we have separated the non-zero contributions to
$\overline{p}_{N}$ into (i) realizations for which all the replicas go
through the upper channel (first term in the parenthesis), (ii) realizations for which all the replicas go
through the lower channel (second term in the parenthesis), and 
finally (iii) realizations for which the two channels are used by the replicas (third term in the parenthesis).

 Contrary to $\overline{p}_{N}$, the limit when $N \to \infty$ of the estimators
 $R_N$, $\rho^{\rm upper}_N$, $\rho^{\rm lower}_N$, $\rho^{\rm mix}_N$,
$\tilde{p}_{N}^{\rm upper}$, $\tilde{p}_{N}^{\rm mix}$
or $\tilde{p}_{N}^{\rm lower}$ (for a given value of
$\nrep$) depends on the choice of the reaction coordinate $\xi$, see Remark~\ref{rem:on_the_limit} below.

From Table~\ref{TableProbaCond}, we observe that for $\xi^1$,
approximately half of the realizations use exclusively the upper
channel and the other half use the lower channel. The associated conditional
estimators $\tilde{p}_{N}^{\rm upper}$ and $\tilde{p}_{N}^{\rm
  lower}$ are very close. This is not the case for $\xi^3$: only very few
realizations go through the upper channel while the associated
probability $\tilde{p}_{N}^{\rm  upper}$ is much larger than the two
other ones $\tilde{p}_{N}^{\rm
  lower}$ and  $\tilde{p}_{N}^{\rm mix}$. This means that a few
realizations contribute a lot to the empirical average
$\overline{p}_{N}$. This explains the very large confidence intervals
observed with $\xi^3$ (in comparison with those observed for $\xi^1$) on Figure~\ref{Joran2D}.

\begin{table}[htbp]
\begin{center}
\begin{tabular}{c||c||c||c||c|c|c||c|c|c||c||}
 & $\beta$ & $N$ & $R_N$ & $\rho^{\rm upper}_N$  & $\rho^{\rm mix}_N$ & $\rho^{\rm lower}_N$ & $\tilde{p}_{N}^{\rm upper}$  & $\tilde{p}_{N}^{\rm mix}$ & $\tilde{p}_{N}^{\rm lower}$ & $\overline{p}_{N}$\\
$\xi^1$ & $8.67$ & $2.10^6$ & 0.81 & 0.45 & 0.03 & 0.52  & 2.7e-09 & 3.0e-09 & 2.3e-09 & 1.7e-09\\
$\xi^3$ & $8.67$ & $2.10^6$ & 0.99 & 0.0008 & 0.02 & 0.98  & 2.3e-06 & 5.9e-10 & 5.5e-10 & 2.4e-09\\
$\xi^1$ & $9.33$ & $4.10^6$ & 0.72 & 0.51 & 0.02 & 0.47  & 6.2e-10 & 6.3e-10 & 2.5e-10 & 3.2e-10 \\ 
$\xi^3$ & $9.33$ & $4.10^6$ & 0.97 & 0.0005 & 0.02 & 0.98  & 1.0e-06 & 5.6e-11 & 9.7e-11 & 6.0e-10\\
$\xi^1$ & $10$    & $6.10^6$ & 0.62 & 0.51 & 0.01 & 0.48  & 1.5e-10 & 1.4e-10 & 5.2e-11 & 6.2e-11\\ 
$\xi^3$ & $10$    & $6.10^6$ & 0.92 & 0.0004 & 0.01 & 0.99  & 1.4e-07 & 1.5e-11 & 1.8e-11 & 6.8e-11
\end{tabular}
\end{center}
\caption{The bi-channel case. Proportion and conditional
  probabilities for two reaction coordinates: the norm to the initial
  point ($\xi^1$) and the abscissa ($\xi^3$).}
\label{TableProbaCond}
\end{table}

\begin{Rem}[On $\rho^{\rm mix}_N$]
We observe that for both reaction coordinates the value of $\rho^{\rm mix}_N$ is very small: on most realizations, if at least one replica reaches $B$ before $A$ then all replicas reaching $B$ go through the same channel. This effect is due to the rather small number of replicas (namely $n_{\rm rep}=100$). When $n_{\rm rep}$ increases, we observe both upper and lower paths on each realization, see for instance \cite{CerouGuyaderLelievrePommier2011} for such experiments.
\end{Rem}

\begin{Rem}[On the efficiency of $\xi^1$]\label{rem:Rn}
We see in Table \ref{TableProbaCond} that $R_N$ is smaller for the best reaction coordinate $\xi^1$ (in terms of fluctuations) than for $\xi^3$. Indeed, the potential $V$ admits a local minimum at $x^{*}\simeq(0,1.54)$ with $\xi^1(x^{*})\simeq1.83$, which is very close to $z_{\rm max}^{1}=1.9$. Notice that the influence of the local minimum is much weaker when using the abscissa since $\xi^3(x^{*})=0$ which is far from $z_{\rm max}^{3}=0.9$ and since most of the trajectories go through the lower channel for this reaction coordinate.

Using $\xi^1$ as a reaction coordinate, replica $X$ going though the upper channel is likely to get trapped around the upper local minimum, and the resampling procedure may produce trajectories which go back to $A$ without increasing the level of the parent replica: as a consequence a higher rate of extinction is observed for $\xi^1$ than for $\xi^3$.

Moreover, it may happen that a replica satisfies $\tT_{z_{\rm max}
}(X)<\tT_A(X)$ without hitting $B$ before $A$; it has then no
contribution in the estimator of the probability
$\P(\tau_B<\tau_A)$. On the left plot on Figure~\ref{Fig:Rn_degenerescence}, we represent the
$\nrep=20$ replicas obtained at the end of one realization of the
algorithm where $\beta=6.33$, and using $\xi^1$: only $6$  replicas
out of $20$ reach $B$ even if all of them have a maximum level larger
than the stopping level $z_{\rm max}$. In such a case, $P_{\rm
  corr}=6/20$ (see \eqref{eq:Pcorr}).

\end{Rem}


\begin{Rem}[On the degeneracy of the branching tree]\label{rem:degenerescence}
On the right plot of Figure~\ref{Fig:Rn_degenerescence}, another
phenomenon is illustrated: for a small number of replicas, we
typically observe that the working replicas at the final iteration of
the algorithm have only a few common ancestors. For example, for the
realization of the algorithm represented on the right plot of Figure~\ref{Fig:Rn_degenerescence} with $\beta=6.33$ and using $\xi^3$ (abscissa) as the reaction coordinate, the $n_{\rm rep}=20$ working replicas are all issued from only two ancestors, at the end of the algorithm. Of course, the number of common ancestors increases when $\nrep$ is larger (see for example \cite{CerouGuyaderLelievrePommier2011}).
%
\end{Rem}



\begin{figure}[!htp]
\centering
\includegraphics[scale=0.4]{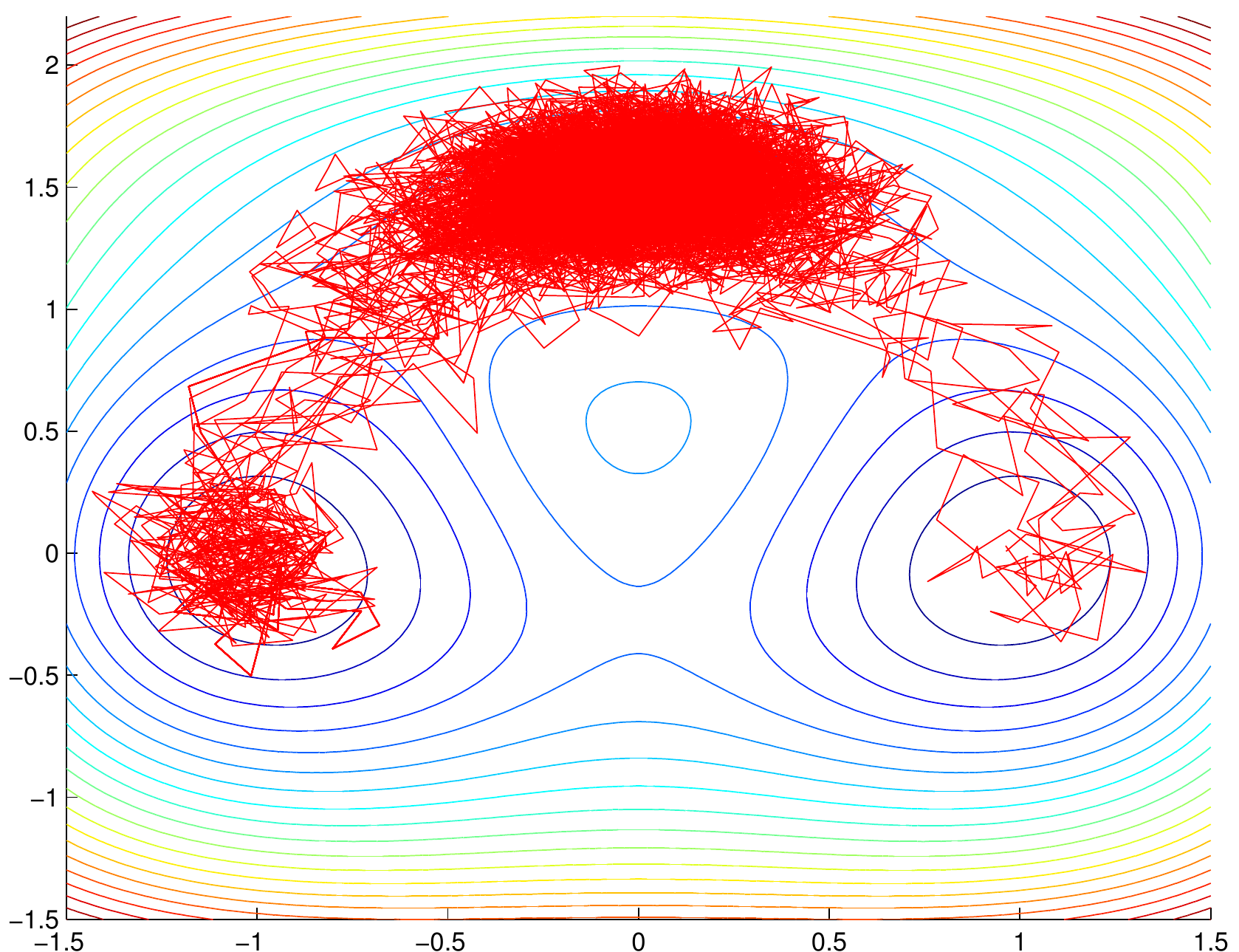}~~\includegraphics[scale=0.4]{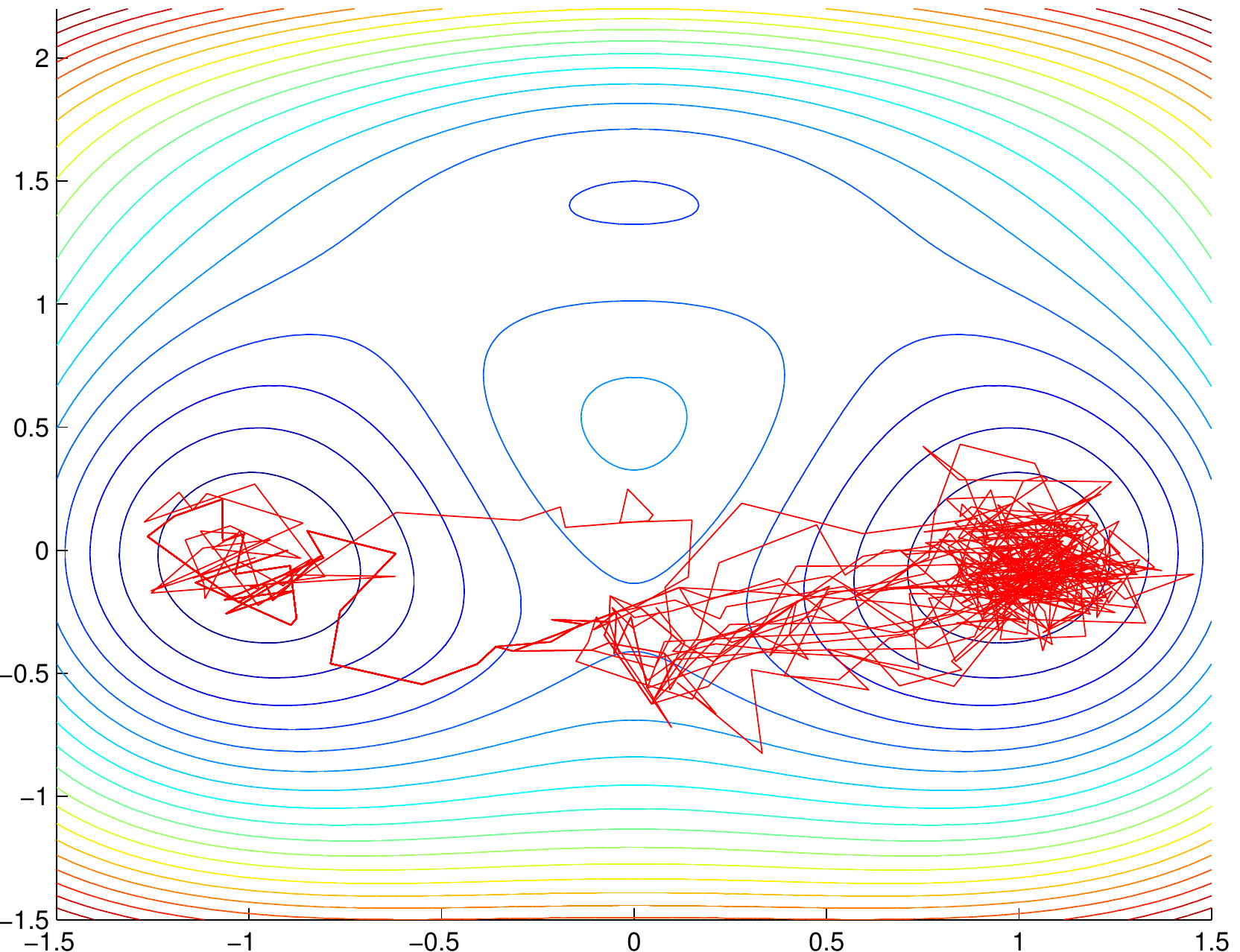}
\caption{Trajectories obtained at the end of a realization of the algorithm for which
  $\beta=6.33$, $n_{\rm rep}=20$. Left: the reaction coordinate is
  $\xi^{1}$ and only $6$ of the working replicas have reached $B$ ($P_{\rm corr}=6/20$), see Remark
  \ref{rem:Rn}. Right: the reaction coordinate is $\xi^3$ and the $20$
  replicas are all issued from only two ancestors, see Remark~\ref{rem:degenerescence}.}
\label{Fig:Rn_degenerescence}
\end{figure}

\begin{Rem}\label{rem:on_the_limit}[On the limit of $R_N$, $\rho^{\rm upper}_N$, $\rho^{\rm
    lower}_N$, $\rho^{\rm mix}_N$ when $N\to \infty$]
The fact that the limits when $N$ goes to infinity of the proportions
$R_N$, $\rho^{\rm upper}_N$, $\rho^{\rm lower}_N$, $\rho^{\rm mix}_N$
or the conditional probabilities $\tilde{p}_{N}^{\rm upper}$, $\tilde{p}_{N}^{\rm mix}$
or $\tilde{p}_{N}^{\rm lower}$  depend on $\xi$ is not in
contradiction with the unbiasedness result of Theorem~\ref{th:unbiased}. Indeed these quantities are not estimators of
the form~\eqref{eq:phi_hat}. In particular, there are not associated
with a functional of the dynamics of the
form~\eqref{eq:phi}. For example, $\tilde{p}_{N}^{\rm upper}$ should
not be confused with the empirical estimator $\overline{p_{N}^{\rm upper}}=\frac{1}{N}\sum_{m=1}^{N}\hat{p}_{m}^{\rm upper}$, where (for a fixed realization $m$)
\begin{align*}
&\hat{p}^{{\rm upper}}=\sum_{n\in I^{(Q_{\rm iter})}_{\rm on}}G^{(n,Q_{\rm iter})}\mathds{1}_{\tT_B(X^{(n,Q_{\rm iter})})<\tT_A(X^{(n,Q_{\rm iter})})}\mathds{1}_{X^{(n,Q_{\rm iter})}\in\text{Upper}}\\
&=\frac{\nrep-K^{(Q_{\rm iter})}}{\nrep}\ldots
\frac{\nrep-K^{(1)}}{\nrep}\Bigg(\frac{1}{\nrep}\sum_{n\in I^{(Q_{\rm
        iter})}_{\rm on}}\mathds{1}_{\tT_B(X^{(n,Q_{\rm iter})})<\tT_A(X^{(n,Q_{\rm iter})})}\mathds{1}_{X^{(n,Q_{\rm iter})}\in\text{Upper}}\Bigg)
\end{align*}
which is an unbiased estimator of $\P\left(\tT_B(X)<\tT_A(X) \text{
    and }X\in\text{Upper}\right)$. The (unbiased) estimator
$\hat{p}^{{\rm lower}}$ of $\P\left(\tT_B(X)<\tT_A(X) \text{ and } X\in\text{Lower}\right)$ is defined similarly, as well as the empirical average $\overline{p_{N}^{\rm lower}}=\frac{1}{N}\sum_{m=1}^{N}\hat{p}_{m}^{\rm lower}$.

Table \ref{TableUpLowProba} contains the values of $\overline{p_{N}^{\rm upper}}$ and $\overline{p_{N}^{\rm lower}}$ computed using the same realizations as above. We include the values of the width $\delta_{N}^{\rm upper}$ and $\delta_{N}^{\rm lower}$ of the confidence intervals associated with the Monte-Carlo procedure (see \eqref{eq:emp_delta}).

\begin{table}[htbp]
\begin{center}
\begin{tabular}{c||c||c||c|c||c|c||c||}
 & $\beta$ & $N$ & $\overline{p_{N}^{\rm lower}}$ & $\delta_{N}^{\rm lower}$ & $\overline{p_{N}^{\rm upper}}$ & $\delta_{N}^{\rm upper}$ & $\overline{p}_N$ \\
$\xi^1$ & $8.67$ & $2.10^6$ & 5.2e-10 & 5.1e-11 & 1.2e-09 & 4.4e-11 & 1.7e-09 \\
$\xi^3$ & $8.67$ & $2.10^6$ & 5.4e-10 & 5.0e-11 & 1.9e-09 & 3.5e-09 & 2.4e-09 \\
$\xi^1$ & $9.33$ & $4.10^6$ & 9.1e-11 & 9.4e-12 & 2.3e-10 & 7.5e-12 & 3.2e-10 \\
$\xi^3$ & $9.33$ & $4.10^6$ & 8.3e-11 & 7.9e-12 & 5.2e-10 & 7.7e-10 & 6.0e-10 \\
$\xi^1$ & $10$    & $6.10^6$ & 1.6e-11 & 3.0e-12 & 4.6e-11 & 1.8e-12 & 6.2e-11 \\ 
$\xi^3$ & $10$    & $6.10^6$ & 1.8e-11 & 2.2e-12 & 5.0e-11 & 7.0e-11 & 6.8e-11
\end{tabular}
\end{center}
\caption{The bi-channel case. Estimation of upper and lower channels probabilities for two reaction coordinates: the norm to the initial
  point ($\xi^1$) and the abscissa ($\xi^3$).}
\label{TableUpLowProba}
\end{table}

We observe good agreement between the values computed with the two
reaction coordinates, as predicted by the unbiasedness result of
Theorem~\ref{th:unbiased}. The estimates of the probability
$\P\left(\tT_B(X)<\tT_A(X) \text{ and }X\in\text{Lower}\right)$ and the corresponding empirical variances obtained with both reaction coordinates are very close. However, when using the abscissa $\xi^3$ we see that the confidence interval $[\overline{p_{N}^{\rm upper}}-\delta_{N}^{\rm upper}/2;\overline{p_{N}^{\rm upper}}+\delta_{N}^{\rm upper}/2]$ is much larger than the one when using $\xi^1$; the variance is larger because taking the upper channel is less likely when using $\xi^3$. Moreover, if we plot the equivalent of Figure \ref{Joran2D} for $\overline{p_{N}^{\rm upper}}$, we observe the same kind of behavior for the evolution of the confidence interval related with $\overline{p_{N}^{\rm upper}}$ as function of the number of realizations $N$: $\overline{p}_{N}$ and $\overline{p_{N}^{\rm upper}}$ jump at the same values of $N$.
\end{Rem}

\subsection{The second two-dimensional example}\label{sec:allen-cahen}

Let us finally consider another example in dimension two, already used
in~\cite{BrehierGazeauGoudenegeRousset}. Our aim is
to show that the very large fluctuations observed with some reaction
coordinates in the first two-dimensional example are related to the existence of multiple pathways from $A$
to $B$. This is actually very much related to some discussions in the
paper~\cite{GlassermanHeidelbergerShahabuddinZajic1998} about the origins of the so-called ``apparent bias''
which is observed with splitting algorithms.

In this example, one parameter governs the shape of the potential. The choice of the reaction coordinate and the statistical behavior of the estimator strongly depend on its value. One goes from a
situation where the estimator has the same statistical behavior
whatever the reaction coordinates  (when
there is only one pathway from $A$ to $B$) to
a situation where the estimation using one of the reaction coordinate
deteriorates for too small Monte-Carlo sample sizes (when two pathways
link $A$ to $B$), even though the estimated probability is
approximately the same in both situations.

\subsubsection{The model}

The second example is inspired by the space discretization of the
Allen-Cahn equation (see~\cite{BrehierGazeauGoudenegeRousset}). The
dynamics is again the overdamped Langevin
equation~\eqref{eq:overdamped} discretized using the Euler-Maruyama
method.
The potential function $\mathcal{E}_{\gamma}$ depends on a parameter $\gamma>0$ and is given by
\[
\mathcal{E}_{\gamma}(x,y)=\gamma(x-y)^2+\frac{1}{2}\left(V(x)+V(y)\right),
\]
where $V(z)=\frac{z^4}{4}-\frac{z^2}{2}$ is a double-well potential. 


\begin{figure}[h]
\centering
\includegraphics[scale=0.4]{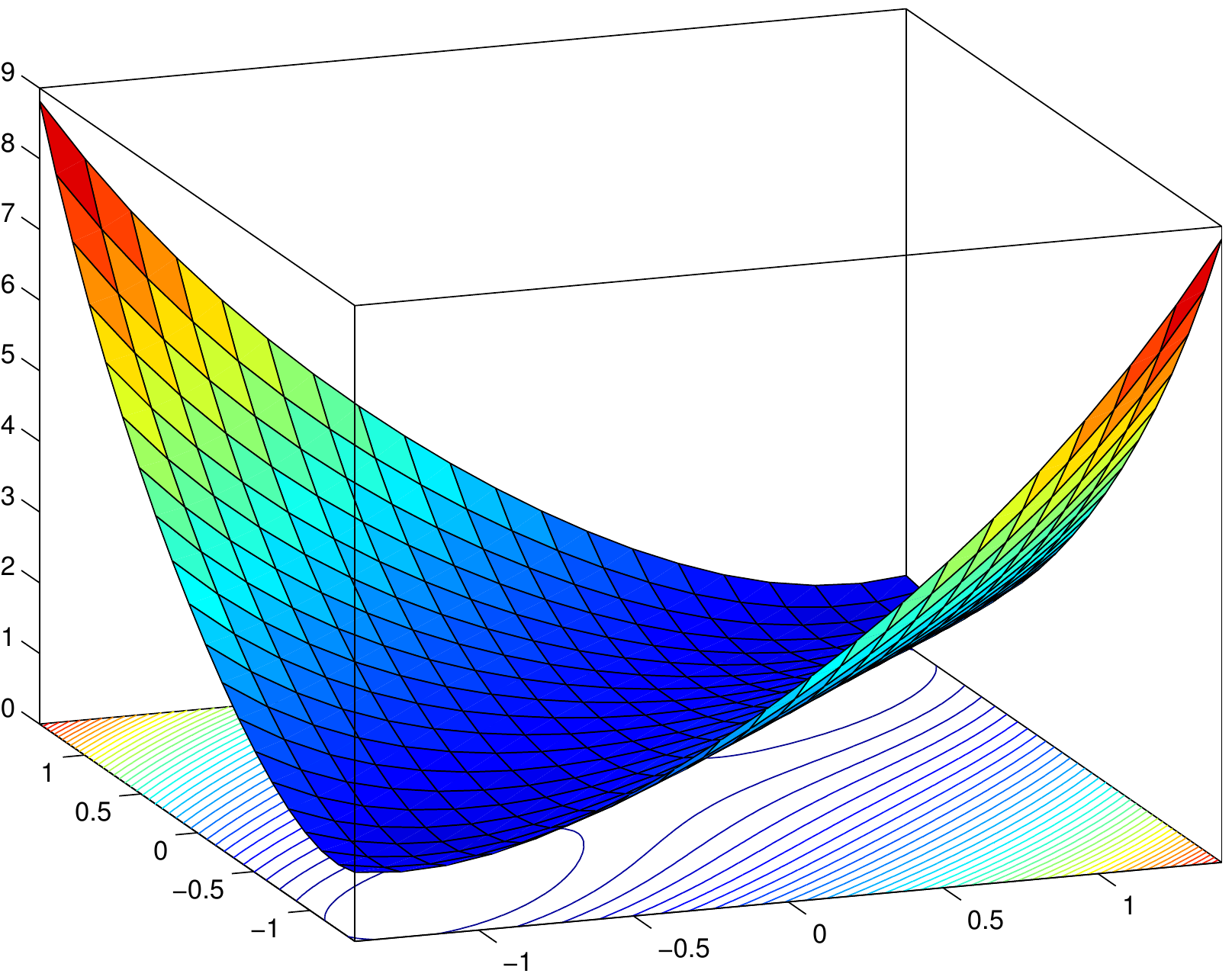}
\includegraphics[scale=0.4]{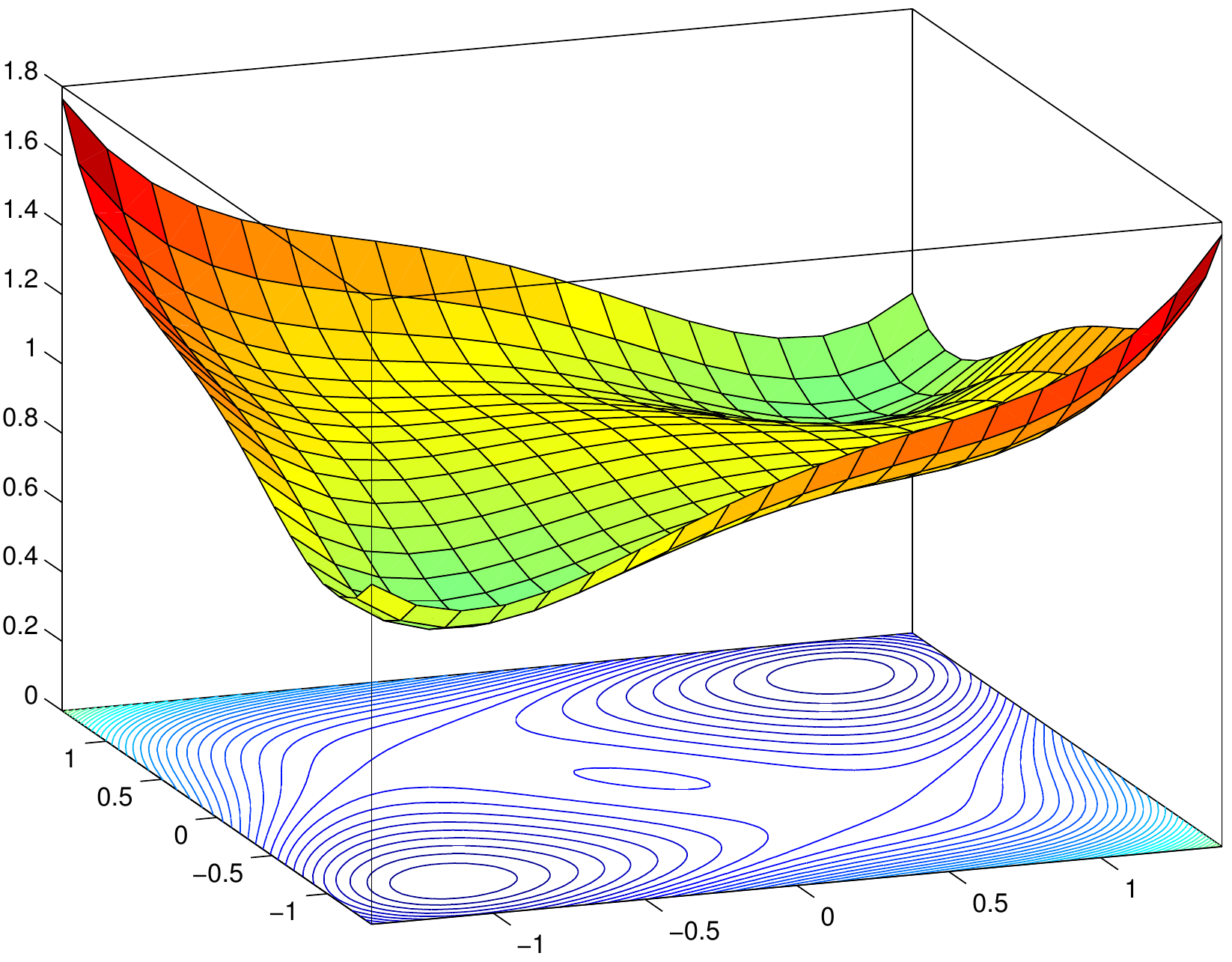}
\caption{Plot of the potential function for the discretized Allen-Cahn problem. Left: $\gamma=1$; right: $\gamma=0.1$.}
  \label{Fig:Potential_AC}
\end{figure}

For any value of $\gamma>0$, there are two global minima $m_A=(x_A,y_A)=(-1,-1)$ and $m_B=(x_B,y_B)= (1,1)$. Moreover, this model exhibits bifurcations with respect to the parameter $\gamma$. When $\gamma>1/8$, $(0,0)$ is the only saddle point, whereas for $\gamma<1/8$ the latter point degenerates into a local maximum and additional saddle points appear (as well as local minima if $\gamma$ is further decreased).

The AMS algorithm is tested with four reaction coordinates (the first
three ones being the same as in Section~\ref{sec:bi-channel}):
\begin{enumerate}
\item the \textit{norm to the initial point $m_A$:} $\xi^{1}(x,y)=\sqrt{(x-x_A)^2+(y-y_A)^2}$, 
\item the \textit{norm to the final point $m_B$:} $\xi^{2}(x,y)=\xi^{1}(x_B,y_B)-\sqrt{(x- x_B)^2+(y-y_B)^2}$, 
\item the \textit{abscissa:} $\xi^{3}(x,y)=x$,
\item the \textit{magnetization:} $\xi^4(x,y)=(x+y)/2$.
\end{enumerate}
The fourth reaction coordinate is called the magnetization because of
its interpretation in the original Allen-Cahn problem. In the figures
below, we associate a color to each reaction coordinate: green for
$\xi^1$, red for $\xi^2$,  blue (line with circles) for $\xi^3$ and cyan for
$\xi^4$.

For the reaction coordinate $\xi^i$, the maximum level used in the stopping criterion of the algorithm is denoted by $z_{\rm max}^{i}$. In the simulations, the following values are used: $z_{\rm max}^{1}=z_{\rm max}^{2}=\sqrt{7.6}$ and $z_{\rm max}^{3}=z_{\rm max}^{4}=0.9$.

As above, 
\[
\begin{cases}
A=\calB(m_A,\rho)=\left\{ (x,y)\in\R^2 : \sqrt{(x-x_A)^2+(y-y_A)^2}<\rho\right\}\\
B=\calB(m_B,\rho)=\left\{ (x,y)\in\R^2 : \sqrt{(x-x_B)^2+(y-y_B)^2}<\rho\right\},
\end{cases}
\]
with $\rho=0.05$. Notice that for $i \in \{1,2,3,4\}$, we have
$B\subset (\xi^i)^{-1}(]z_{\rm max}^i,+\infty[)$ (see~\eqref{eq:hyp_B}).

The deterministic initial condition is $X_0=x_0=(-0.9,-0.9)$, and we always take $k=1$. By
default, the number of replicas is taken equal to $n_{\rm rep}=100$,
and the empirical averages are computed with $N=10^6$ independent
realizations. When $\gamma=0.1$, we also take $n_{\rm rep}=10$ (with $N=6.10^6$) and
$n_{\rm rep}=1000$ (with $N=10^5$). Notice that we have also tested the algorithm in
the latter case when $k>1$: since we observe the same kind of behavior
as for $k=1$, we do not present the results of these numerical simulations.

\subsubsection{Simulations for $\gamma=1$ and $\beta \in\left\{10,20,40,80\right\}$}\label{sec:gentil}

Let us first consider the case $\gamma=1$. In this situation there is
only one reactive path to go from $A$ to $B$, going through the saddle point
$(0,0)$. Let us consider the following values for the inverse
temperature $\beta\in\left\{10,20,40,80\right\}$. We plot on Figure~\ref{AC2D_beta_20_gamma_1_nrep100} the evolution of the confidence
interval for the estimator $\hat{p}$ as a function of the
number of independent realizations, for $\nrep=100$. We observe that the confidence
intervals overlap, and that the statistical fluctuations are very
similar whatever the reaction coordinate.


\begin{figure}[!htp]
\centering
\includegraphics[scale=0.4]{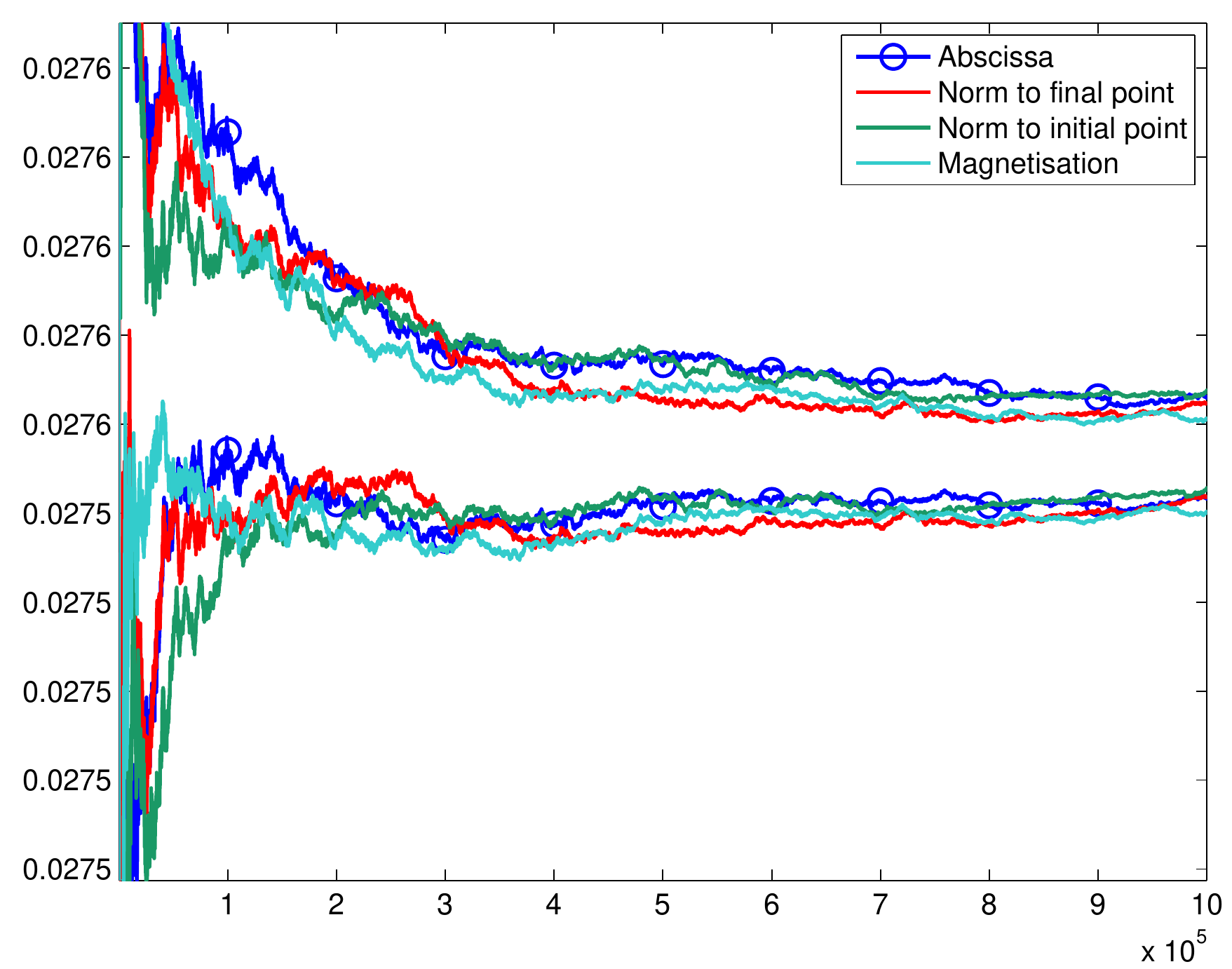} 
\includegraphics[scale=0.4]{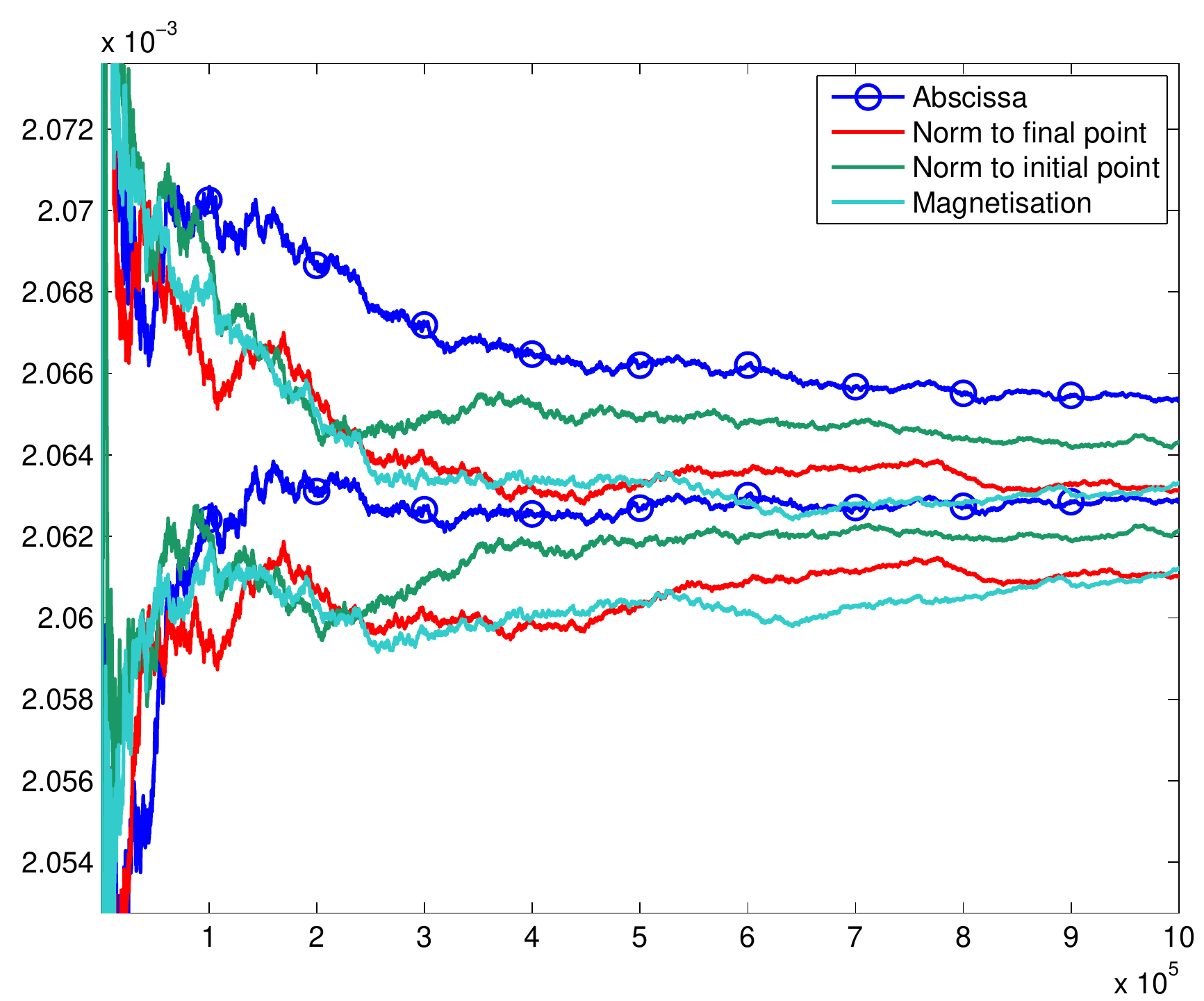} \\
\includegraphics[scale=0.4]{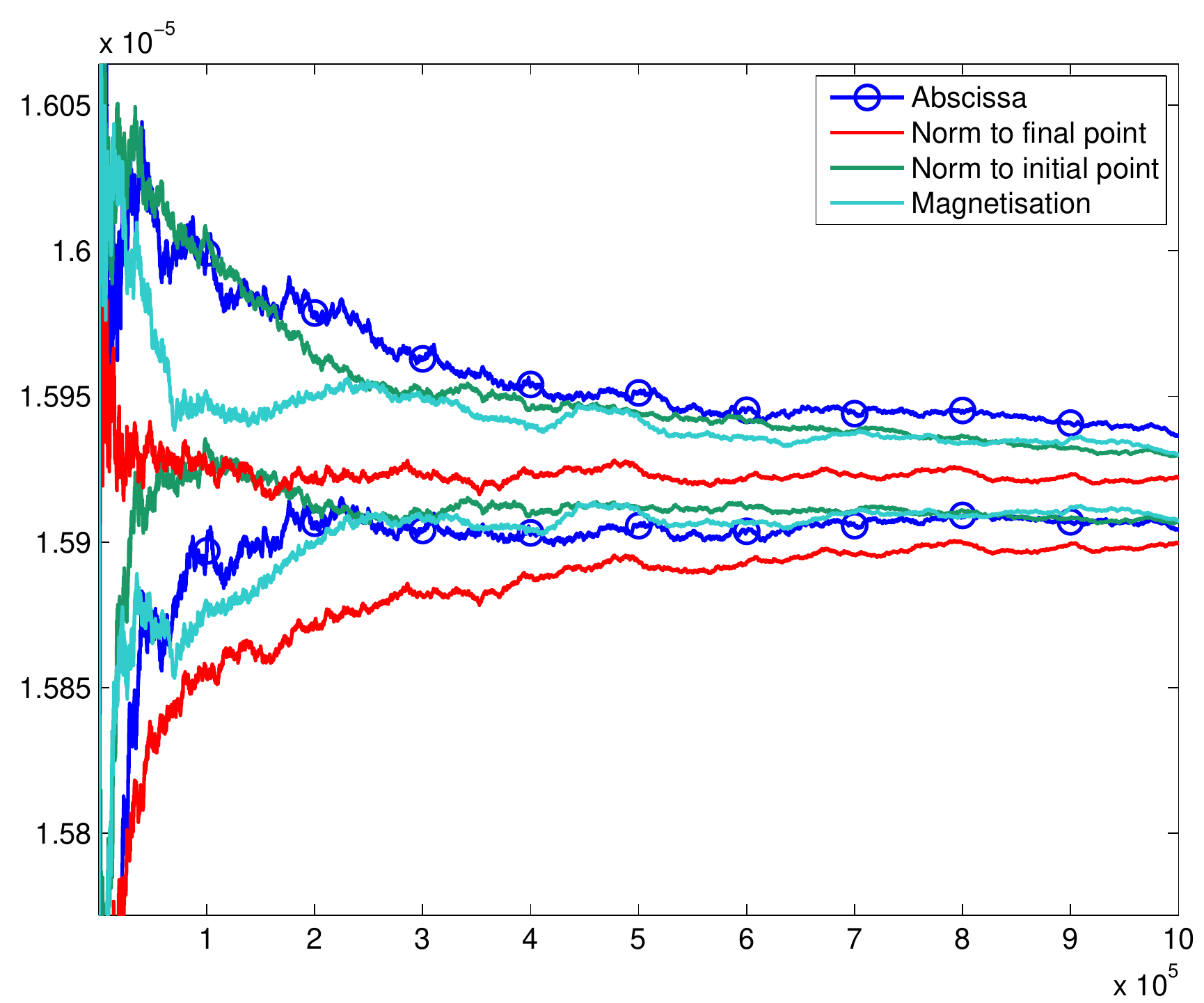} 
\includegraphics[scale=0.4]{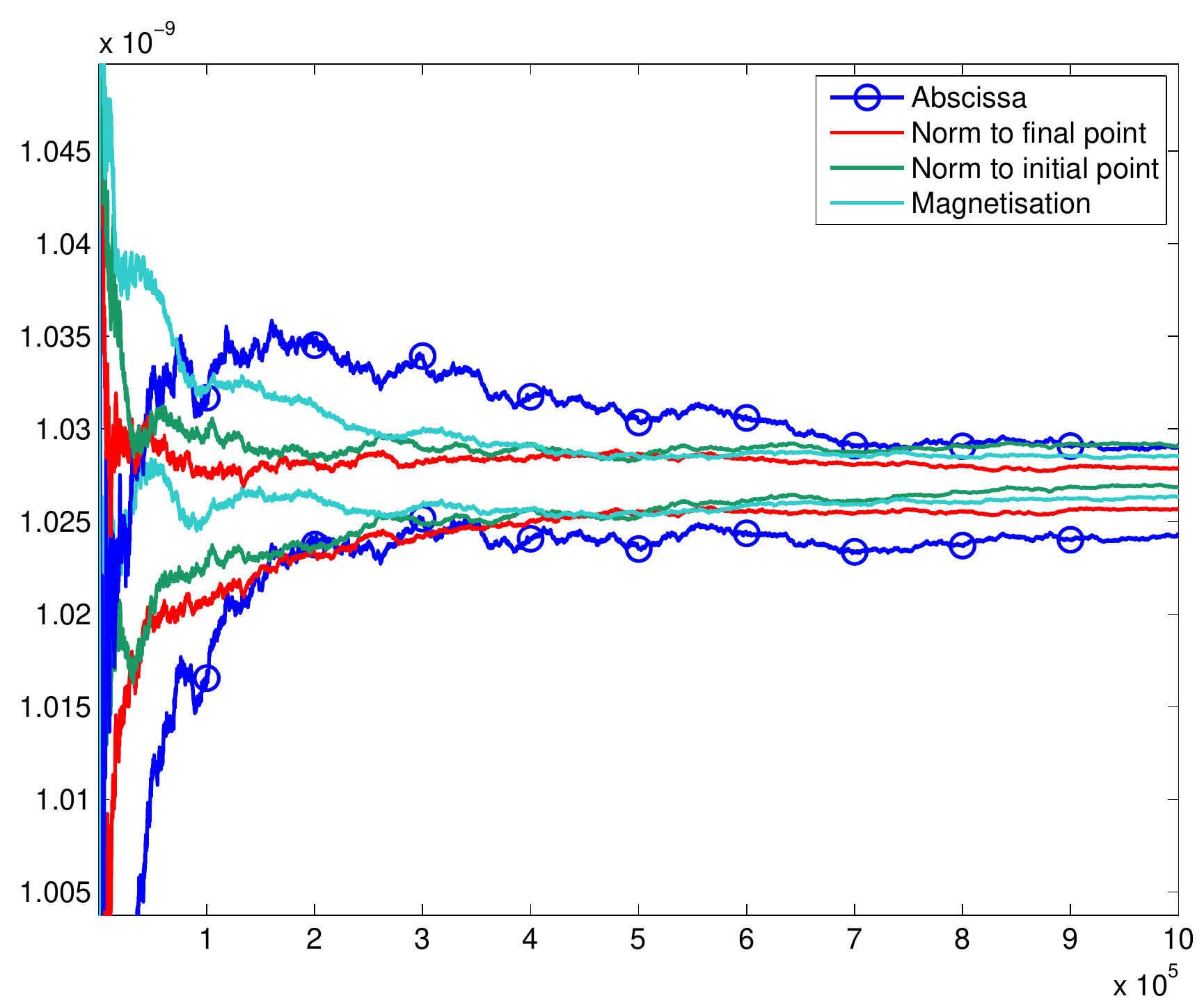}
\caption{Evolution as a function of $N$ of the
empirical mean $\overline{p}_{N}$ and
of the associated $95\%$ confidence intervals
$[\overline{p}_{n}-\delta_N/2,\overline{p}_{n}+\delta_N/2]$ with
$\gamma=1$. Upper left: $\beta=10$, upper right: $\beta=20$, lower
left: $\beta=40$, lower right $\beta=80$. }
\label{AC2D_beta_20_gamma_1_nrep100}
\end{figure}

For comparison, the results obtained with a standard direct
Monte-Carlo estimation are given in Table \ref{CMC_AC2D_gamma_1} for
$\beta\in\left\{10,20,40\right\}$ with $N=6.10^8$ realizations. The
results are consistent with those obtained by the AMS algorithm.
For
$\beta=80$, the probability is very small and we were not able to get
a reliable result by standard direct Monte-Carlo simulations with a
reasonable number of realizations.

\begin{table}[htbp]
\centering
\begin{tabular}{|c||c|c|c|}
$\beta$ & 10 & 20 & 40 \\
$\overline{p}_{N}^{\text{MC}}$ & 2.755e-2 & 2.062e-3 & 1.582e-5\\
$\delta_N$ & 0.003e-2 & 0.007e-3 & 0.063e-5\\
\end{tabular}
\caption{Standard direct Monte-Carlo estimation with $N=6.10^8$ realizations.}
\label{CMC_AC2D_gamma_1}
\end{table}

\subsubsection{Simulations for $\gamma=0.1$ and $\beta=80$}

When the parameter $\gamma$ is smaller than $1/8$, the point $(0,0)$
is no longer a saddle point (but instead a local maximum). In this case, there are
two saddle points on the line $\xi^{4}=0$ (they are symmetric with
respect to $(0,0)$). We take $\gamma=0.1$ and represent on
Figure~\ref{AC2D_beta_80_gamma_05_nrep100} the evolution of the
confidence interval for the estimator $\hat{p}$ as a function
of the number $N$ of independent runs of the algorithm and for different values of $\nrep$: $\nrep=10$ (with up to $N=6.10^6$ realizations), $\nrep=100$ (with up to $N=10^6$ realizations) and $\nrep=1000$ (with up to $N=10^5$ realizations). The inverse temperature is fixed to $\beta=80$.

We observe that the statistical fluctuations of the estimator for the
reaction coordinate $\xi^3$ (abscissa) are much larger than for the
three other ones. The estimator is unbiased, but a large number of
realizations is required to get a significant confidence interval. By
comparing with the results obtained in the Section~\ref{sec:bi-channel}, we thus
conclude that the origin of these fluctuations is related to the
existence of two pathways from $A$ to $B$, rather than to the
smallness of the estimated probability for example. In fact, the two pathways do not play symmetric roles when using the third reaction coordinate, although they play symmetric roles when using the other ones. This is in
agreement with some discussions in the paper~\cite{GlassermanHeidelbergerShahabuddinZajic1998} about the origins of the so-called ``apparent bias'' which is observed with splitting algorithms.

\begin{figure}[!htp]
\centering
\includegraphics[scale=0.4]{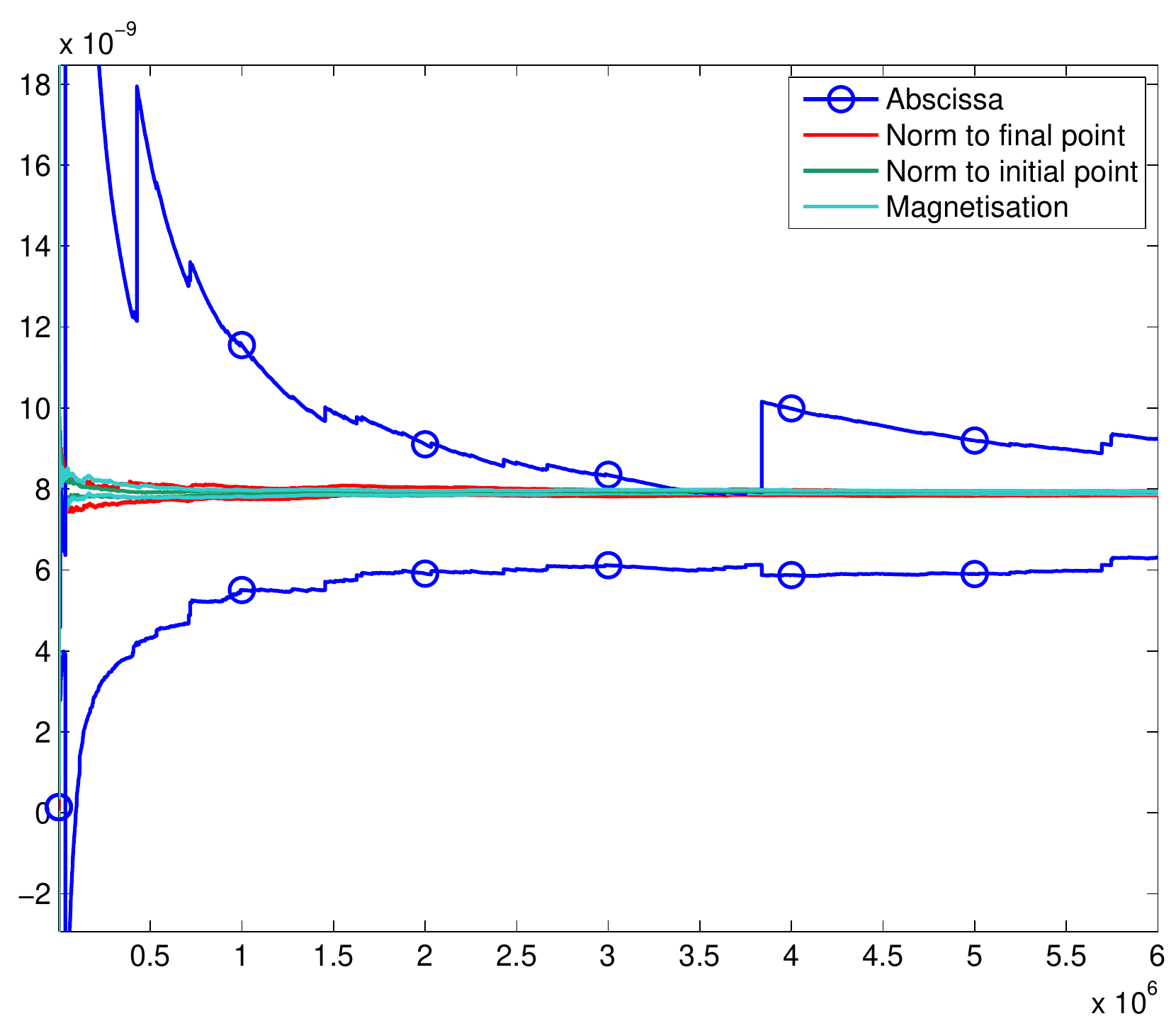}
\includegraphics[scale=0.4]{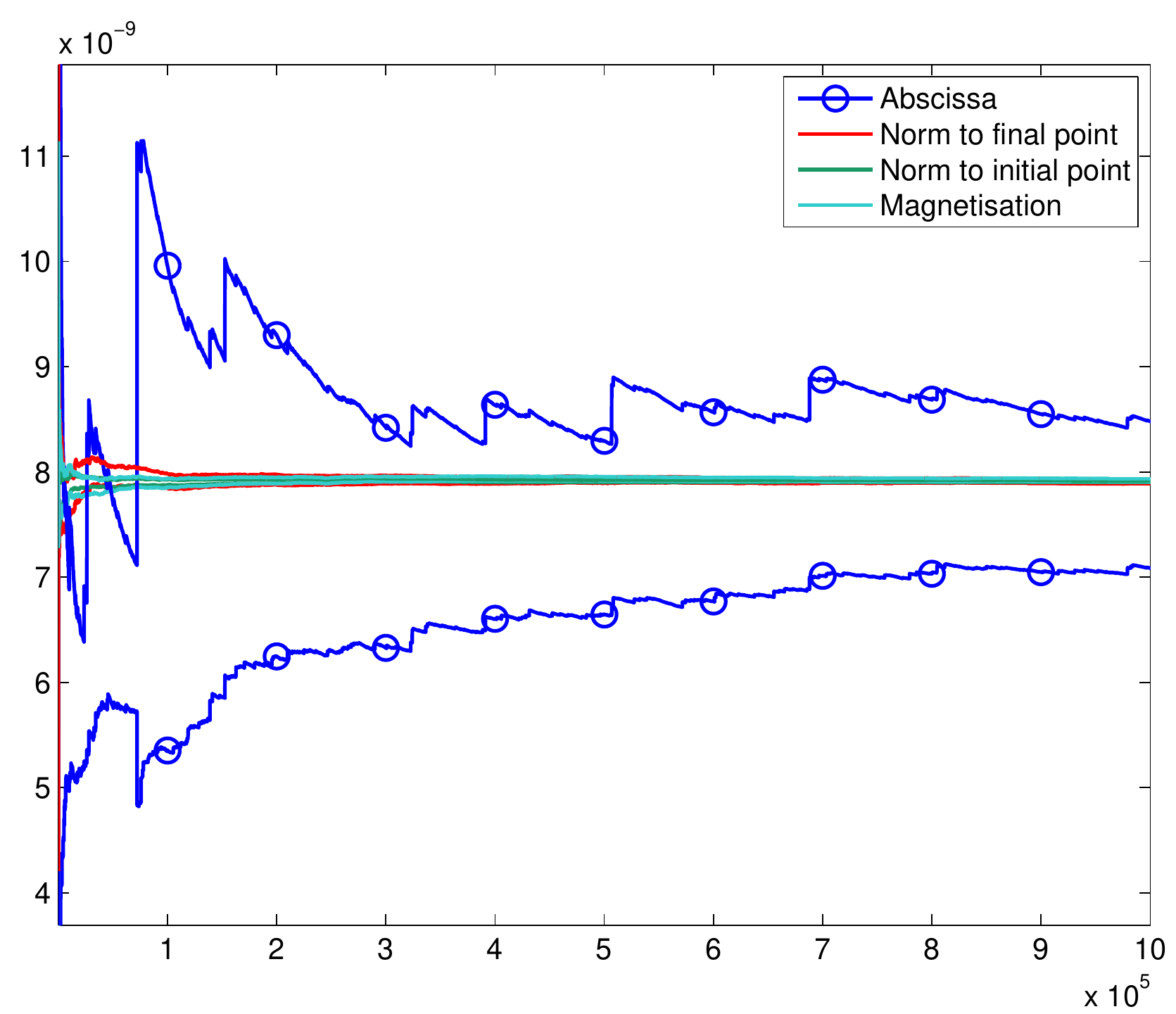}\\
\includegraphics[scale=0.4]{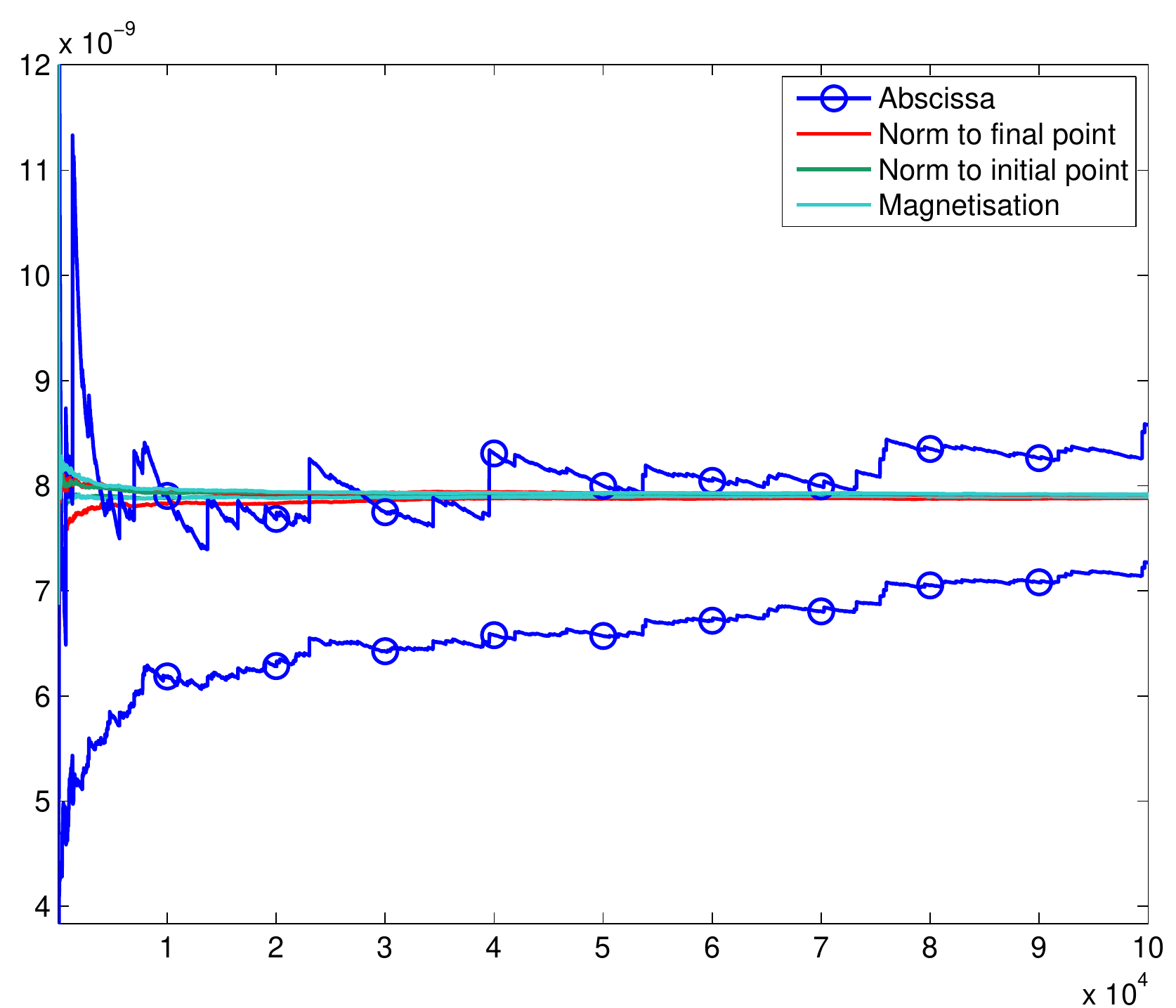}
\caption{Evolution as a function of $N$ of the
empirical mean $\overline{p}_{N}$ and
of the associated $95\%$ confidence intervals
$[\overline{p}_{n}-\delta_N/2,\overline{p}_{n}+\delta_N/2]$ with
$\gamma=0.1$ and $\beta=80$. Upper left $\nrep=10$, upper right
$\nrep=100$, lower $\nrep=1000$.}
\label{AC2D_beta_80_gamma_05_nrep100}
\end{figure}

\subsection{Conclusions and practical recommendations}\label{sec:conc}

Let us summarize our findings on these numerical simulations.
\begin{itemize}
\item We always observe that for sufficiently large values of $N$
  (number of independent Monte Carlo simulations), the confidence
  intervals of the estimator $\hat{p}$ overlap, whatever
  $\nrep$, $k$ or $\xi$. This is in accordance with our theoretical
  result on the unbiasedness of this estimator.
\item We observe numerically in our
  two-dimensional simulations that the estimator $\hat{p}$ has a heavy
  tail so that very few realizations
  contribute a lot to the empirical average. This explains why for some choices of the reaction coordinate
  the efficiency of the estimation may be very poor since
  the empirical average converges very slowly to its
  limit (namely $N$ should be taken sufficiently large to obtain
  significant results). Actually, as explained in
  Remark~\ref{rem:heavy_tail}, even in
  some idealized setting where the reaction coordinate is the committor
  function, in some small
  probability regime, the estimator has a log-normal distribution, and
  thus has heavy tails.
\item In
  multiple channel cases (namely when multiple pathways exist from
  $A$ to $B$), one may observe non-overlapping empirical
  confidence intervals of the estimator for different reaction
  coordinates if the number of independent realizations $N$ is too small. This is related to the fact
  that very large contributions to the average of the estimator are
  associated with trajectories going through very unlikely (for the
  considered reaction coordinate and value of $\nrep$) channels. This
  is a known phenomenon for splitting algorithms in general, see
  \cite{GlassermanHeidelbergerShahabuddinZajic1998}, where it is referred to as ``apparent bias''. In particular, a good reaction coordinate in a multiple
  channel case is such that, conditionally to reach a certain maximum
  level $z$, the relative likelihood of the channels used by the paths
  to reach this maximum level does not depend too much on $z$. For
  example, a reaction coordinate close to the committor function is a
  good candidate to achieve this purpose. This opens the route to
  adaptive algorithms, where the reaction coordinate would be updated
  in order to get closer and closer to the committor function as long as
  successive AMS algorithms are launched (see~\cite{CerouGuyaderLelievrePommier2011}). We
  intend to investigate this direction in future works.
\end{itemize}

As a conclusion to these numerical results, we thus recommend the
following in order to get reliable estimates of the probability $\P(\tau_B < \tau_A)$ with the AMS algorithm.
\begin{itemize}
\item One should be careful in the implementation of the splitting and
  branching steps, in particular in the treatment of replicas which
  have the same maximum level and in the definition of the branching
  point in the resampling procedure. For correct implementations, 
  unbiased estimators can be built, and the general framework of Section~\ref{sect:GAMS} yields many variants for the algorithm.
  
\item Thanks to the unbiasedness property, one should check the independence of the computed probability on the choice of the parameters: the number of replicas $\nrep$, the minimum number of resampled replicas $k$ and the reaction coordinates $\xi$. In particular, we recommend to perform simulations with  various reaction coordinates and to set the minimal number of independent realizations such that the empirical confidence intervals overlap.

\item Thanks to the unbiasedness property, one can perform
  many independent realizations of the algorithm with a relatively
  small number of replicas, instead of using a few independent
  realizations with a large number of replicas. Indeed, assume  that we are in a regime where the variance scales like
  $\frac{1}{\nrep \times N}$; this is the case for instance in the so-called ideal case for sufficiently large $\nrep$ and $N$, see \cite{BrehierLelievreRousset2015}.
  Since the
  parallelization of independent runs of the algorithm is trivial,
  for a fixed product $\nrep \times N$ (namely for a fixed CPU cost),
  the strategy with less replicas is thus much more interesting in terms of
  wall-clock time (which scales like $\nrep$) than the strategy with
  more replicas.
  
  
   \end{itemize}

\section*{Acknowledgements}
C.-E.~Br\'ehier  is grateful to INRIA Rocquencourt for having funded his
postdoctoral position (September 2013-December 2014), and acknowledges support  from the SNF grant
200020-149871/1. M.~Gazeau is grateful to INRIA Lille Nord Europe
where a part of this research has been conducted and to Labex CEMPI
(ANR-11-LABX-0007-01). The work of T.~Leli\`evre and M.~Rousset is
supported by the European Research Council under the European Union's
Seventh Framework Programme (FP/2007-2013) / ERC Grant Agreement
number 614492. The authors are grateful to 
the Labex Bezout (ANR-10-LABX-58-01) which supported this project at
an early stage, during the CEMRACS 2013. Finally, the authors would like to thank 
F. Bouchet, F. C\'erou, A. Guyader, J. Rolland and E. Simonnet for many fruitful discussions.

\bibliographystyle{plain}
\bibliography{bibBGGLR}

\def\polhk#1{\setbox0=\hbox{#1}{\ooalign{\hidewidth
  \lower1.5ex\hbox{`}\hidewidth\crcr\unhbox0}}} \def\cprime{$'$}
\begin{thebibliography}{10}

\bibitem{billinsgley-99}
P.~Billingsley.
\newblock {\em Convergence of probability measures}.
\newblock Wiley Series in Probability and Statistics: Probability and
  Statistics. John Wiley \& Sons, Inc., New York, second edition, 1999.

\bibitem{Brehier_15}
C.-E. Br{\'e}hier.
\newblock Large deviations principle for the adaptive multilevel splitting
  algorithm in an idealized setting.
\newblock {\em Preprint}, 2015.
\newblock {\tt http://arxiv.org/abs/1502.06780}.

\bibitem{BrehierGazeauGoudenegeRousset}
C.-E. Br\'ehier, M.~Gazeau, L.~Gouden\`ege, and M.~Rousset.
\newblock Analysis and simulations of rare events for {SPDE}s.
\newblock {\em ESAIM Proceedings and Reviews}, 48:364--384, 2015.

\bibitem{BrehierGoudenegeTudela}
C.-E. Br{\'e}hier, L.~Gouden{\`e}ge, and L.~Tudela.
\newblock Central limit theorem for adaptive multilevel splitting estimators in
  an idealized setting.
\newblock {\em Preprint}, 2014.
\newblock {\tt http://arxiv.org/abs/1501.01399}.

\bibitem{BrehierLelievreRousset2015}
C.-E. Br{\'e}hier, T.~Leli{\`e}vre, and M.~Rousset.
\newblock Analysis of adaptive multilevel splitting algorithms in an idealized
  setting.
\newblock {\em ESAIM Probability and Statistics}, to appear, 2015.

\bibitem{Bucklew2004}
J.A. Bucklew.
\newblock {\em Introduction to rare event simulation}.
\newblock Springer Series in Statistics. Springer-Verlag, New York, 2004.

\bibitem{CerouDel-MoralFuronGuyader2012}
F.~C{\'e}rou, P.~Del~Moral, T.~Furon, and A.~Guyader.
\newblock Sequential {M}onte {C}arlo for rare event estimation.
\newblock {\em Stat. Comput.}, 22(3):795--808, 2012.

\bibitem{CerouGuyader2007}
F.~C{\'e}rou and A.~Guyader.
\newblock Adaptive multilevel splitting for rare event analysis.
\newblock {\em Stoch. Anal. Appl.}, 25(2):417--443, 2007.

\bibitem{CerouDel-MoralGuyaderMalrieu2014}
F.~C{\'e}rou and A.~Guyader.
\newblock Fluctuations of adaptive multilevel splitting.
\newblock {\em Preprint}, 2014.
\newblock {\tt http://arxiv.org/abs/1408.6366}.

\bibitem{CerouGuyaderLelievrePommier2011}
F.~C{\'e}rou, A.~Guyader, T.~Leli{\`e}vre, and D.~Pommier.
\newblock A multiple replica approach to simulate reactive trajectories.
\newblock {\em Journal of Chemical Physics}, 134(5), 2011.

\bibitem{Del-Moral2004}
P.~Del~Moral.
\newblock {\em Feynman-Kac Formulae}.
\newblock Springer, 2004.

\bibitem{DoucetDe-FreitasGordon2001}
A.~Doucet, N.~De~Freitas, and N.~Gordon, editors.
\newblock {\em {S}equential {M}onte {C}arlo methods in practice}.
\newblock Springer, 2001.

\bibitem{GarvelsKroeseVan-Ommeren2002}
M.J.J. Garvels, D.P. Kroese, and J.C.W. van Ommeren.
\newblock On the importance function in splitting simulation.
\newblock {\em European Transactions on Telecom- munications}, 13(4):363--371,
  2002.

\bibitem{GlassermanHeidelbergerShahabuddinZajic1998}
P.~Glasserman, P.~Heidelberger, P.~Shahabuddin, and T.~Zajic.
\newblock A large deviations perspective on the efficiency of multilevel
  splitting.
\newblock {\em IEEE Trans. Automat. Control}, 43(12):1666--1679, 1998.

\bibitem{GuyaderHentgartnerMatzner-Lober2011}
A.~Guyader, N.~Hengartner, and E.~Matzner-L{\o}ber.
\newblock Simulation and estimation of extreme quantiles and extreme
  probabilities.
\newblock {\em Appl. Math. Optim.}, 64(2):171--196, 2011.

\bibitem{JohansenDel-MoralDoucet2005}
A.M. Johansen, P.~Del~Moral, and A.~Doucet.
\newblock Sequential {M}onte {C}arlo samplers for rare events.
\newblock In {\em Proceedings of the 6th International Workshop on Rare Event
  Simulation, RESIM 2006, Bamberg}, pages 256--267, 2006.

\bibitem{KaratzasShreve91}
I.~Karatzas and S.E. Shreve.
\newblock {\em Brownian Motion and Stochastic Calculus}.
\newblock Graduate Texts in Mathematics. Springer New York, 1991.

\bibitem{metzner-schuette-vanden-eijnden-06}
P.~Metzner, C.~Sch\"{u}tte, and E.~Vanden-Eijnden.
\newblock Illustration of transition path theory on a collection of simp le
  examples.
\newblock {\em Journal of Chemical Physics}, 125(1), 2006.

\bibitem{park-sener-lu-schulten-03}
S.~Park, M.K. Sener, D.~Lu, and K.~Schulten.
\newblock Reaction paths based on mean first-passage times.
\newblock {\em J. of Chem. Phys.}, 119:1313--1319, 2003.

\bibitem{RollandSimonnet2015}
J.~Rolland and E.~Simonnet.
\newblock Statistical behaviour of adaptive multilevel splitting algorithms in
  simple models.
\newblock {\em Journal of Computational Physics}, 283:541 -- 558, 2015.

\bibitem{Shiryayev1984}
A.N. Shiryayev.
\newblock {\em Probability}, volume~95 of {\em Graduate Texts in Mathematics}.
\newblock Springer-Verlag, New York, 1984.
\newblock Translated from the Russian by R. P. Boas.

\bibitem{Simonnet2014}
E.~Simonnet.
\newblock Combinatorial analysis of the adaptive last particle method.
\newblock {\em Statistics and Computing}, pages 1--20, 2014.

\end{thebibliography}

\end{document}